\documentclass[tbtags,10pt,a4paper]{amsart}
\pdfoutput=1 

\numberwithin{equation}{section}
\allowdisplaybreaks 

\usepackage{amssymb,eucal,mathrsfs,setspace,bbm}
\usepackage[noadjust,nobreak]{cite}
\usepackage[left=20mm, right=20mm, top=24mm, bottom=24mm, head=10mm, foot=10mm]{geometry}

\usepackage{array}
\newcolumntype{C}{>{$}c<{$}} 

\usepackage[largesc,theoremfont,nohelv]{newtxtext} 
\usepackage[libertine,cmbraces,varbb,smallerops]{newtxmath}
\begingroup
  \makeatletter
  \@for\theoremstyle:=definition,remark,plain\do{%
    \expandafter\g@addto@macro\csname th@\theoremstyle\endcsname{%
      \addtolength\thm@preskip{.5\baselineskip plus .2\baselineskip minus .2\baselineskip}
      \addtolength\thm@postskip{.5\baselineskip plus .2\baselineskip minus .2\baselineskip}
    }%
  }
\endgroup

\usepackage{enumitem}
\setitemize{leftmargin=*}       
\setenumerate{leftmargin=*,     
  label=\textnormal{\arabic*.\@}
	}

\usepackage{mathtools} 
\usepackage{graphicx}

\usepackage{tikz}
\usetikzlibrary{calc}
\tikzset{%
	>=latex,
	wt/.style={circle, draw=black, fill=black, inner sep=2pt, outer sep=0pt, minimum size=5pt}, 
	wwt/.style={rectangle, fill=white, inner sep=1pt, outer sep=0pt, minimum size=5pt}, 
	fillcolourA/.style={fill=blue,fill opacity=0.3}, 
	fillcolourB/.style={fill=red,fill opacity=0.3}, 
	fillcolourC/.style={fill=green,fill opacity=0.3} 
}

\usepackage[colorlinks=true,citecolor=red,linkcolor=blue,linktoc=all]{hyperref} 

\usepackage[capitalise,noabbrev]{cleveref} 
  
  \makeatletter
      \def\lcfirstnamecrefs#1,#2\@nil{\lcnamecrefs{#1}}
      \newcommand{\lcfirstnamecref}[1]{\lcfirstnamecrefs #1,\@nil}
  \makeatother
  \newcommand{\lccrefs}[1]{\lcfirstnamecref{#1}}

\theoremstyle{plain}
\newtheorem{theorem}{Theorem}[section]
\newtheorem{corollary}[theorem]{Corollary}
\newtheorem{lemma}[theorem]{Lemma}
\newtheorem{proposition}[theorem]{Proposition}

\newtheorem{definition}[theorem]{Definition}
\theoremstyle{remark}
\newtheorem{remark}[theorem]{Remark}

\newcommand{\pd}{\partial}     
\newcommand{\wun}{\mathbbm{1}} 
\newcommand{\eps}{\varepsilon} 

\newcommand{\cc}{\mathsf{c}}   
\newcommand{\dd}{\mathrm{d}}   
\newcommand{\ee}{\mathrm{e}}   
\newcommand{\ii}{\mathrm{i}}   
\newcommand{\kk}{\mathsf{k}}   
\newcommand{\qq}{\mathsf{q}}   
\newcommand{\uu}{\mathsf{u}}   
\providecommand{\vv}{}\renewcommand{\vv}{\mathsf{v}} 
\newcommand{\yy}{\mathsf{y}}   
\newcommand{\zz}{\mathsf{z}}   

\renewcommand{\simeq}{\cong}              

\renewcommand{\ge}{\geqslant} 
\renewcommand{\le}{\leqslant} 

\DeclarePairedDelimiter{\brac}{\lparen}{\rparen}   
\DeclarePairedDelimiter{\sqbrac}{\lbrack}{\rbrack} 
\DeclarePairedDelimiter{\set}{\lbrace}{\rbrace}
\newcommand{\st}{\mspace{5mu} {:} \mspace{5mu}}    
\DeclarePairedDelimiter{\abs}{\lvert}{\rvert}
\DeclarePairedDelimiter{\norm}{\lVert}{\rVert}
\DeclarePairedDelimiterX{\comm}[2]{\lbrack}{\rbrack}{#1 , #2}  
\DeclarePairedDelimiterX{\bilin}[2]{\langle}{\rangle}{#1 , #2} 
\DeclarePairedDelimiterX{\varbilin}[2]{\lparen}{\rparen}{#1 \delimsize\vert\mathopen{} #2} 
\newcommand{\no}[1]{\mathopen{:} #1 \mathclose{:}} 

\DeclareMathOperator{\ad}{ad}

\DeclareMathOperator{\Gal}{Gal}
\DeclareMathOperator{\id}{id}

\DeclareMathOperator{\spn}{span}

\newcommand{\blank}{{-}}                           

\newcommand{\Ra}{\Rightarrow}

\newcommand{\lra}{\longrightarrow}
\newcommand{\ira}{\hookrightarrow}

\newcommand{\ses}[3]{0 \lra #1 \lra #2 \lra #3 \lra 0}

\newcommand{\fld}[1]{\mathbb{#1}}    
\newcommand{\alg}[1]{\mathfrak{#1}}  
\newcommand{\grp}[1]{\mathsf{#1}}    
\newcommand{\VOA}[1]{\mathsf{#1}}    
\newcommand{\Mod}[1]{\mathcal{#1}}   
\newcommand{\categ}[1]{\mathscr{#1}} 

\newcommand{\ZZ}{\fld{Z}}
\newcommand{\NN}{\ZZ_{\ge 0}} 
\newcommand{\QQ}{\fld{Q}}
\newcommand{\RR}{\fld{R}}
\newcommand{\CC}{\fld{C}}

\newcommand{\affine}[1]{\widehat{#1}}
\newcommand{\SLG}[2]{\grp{#1}_{#2}}                      
\newcommand{\SLA}[2]{\alg{#1}_{#2}}                      
\newcommand{\AKMA}[2]{\affine{\alg{#1}}_{#2}}            

\newcommand{\gltwo}{\SLA{gl}{2}}
\newcommand{\sltwo}{\SLA{sl}{2}}
\newcommand{\asltwo}{\AKMA{sl}{2}}
\newcommand{\slthree}{\SLA{sl}{3}}
\newcommand{\aslthree}{\AKMA{sl}{3}}

\newcommand{\csub}{\alg{h}}                              
\newcommand{\bgalg}{\alg{g}}                             

\newcommand{\killingsymb}{\kappa}                        
\DeclarePairedDelimiterXPP{\killing}[2]{\killingsymb}{\lparen}{\rparen}{}{#1,#2} 

\newcommand{\co}[1]{#1^{\vee}}                           
\newcommand{\srt}[1]{\alpha_{#1}}                        
\newcommand{\scrt}[1]{h^{#1}}                            
\newcommand{\fwt}[1]{\omega_{#1}}                        
\newcommand{\fcwt}[1]{g^{#1}}                            
\newcommand{\wvec}{\rho}                                 

\newcommand{\wlat}{\grp{P}}                              
\newcommand{\cwlat}{\co{\grp{P}}}                        
\newcommand{\rlat}{\grp{Q}}                              
\newcommand{\crlat}{\co{\grp{Q}}}                        
\newcommand{\pwlat}[1][\uu-3]{\wlat_{\ge}^{#1}}          

\DeclarePairedDelimiterXPP{\uealg}[1]{\mathsf{U}}{\lparen}{\rparen}{}{#1}   

\newcommand{\vac}{\varnothing}                           

\newcommand{\bpsymb}{\VOA{BP}}
\newcommand{\slthreesymb}{\VOA{A}_2}
\newcommand{\fgsymb}{\VOA{F}}
\newcommand{\bgsymb}{\VOA{G}}
\newcommand{\lsymb}{\Pi}

\newcommand{\heis}{\VOA{H}}                        
\newcommand{\fgvoa}{\fgsymb}                       
\newcommand{\bgvoa}{\bgsymb}                       
\newcommand{\lvoa}{\lsymb}                         

\newcommand{\uaff}[2]{\VOA{V}^{#1}(#2)}            
\newcommand{\saff}[2]{\VOA{L}_{#1}(#2)}            
\newcommand{\usl}[1][\kk]{\uaff{#1}{\slthree}}     
\newcommand{\ssl}[1][\kk]{\saff{#1}{\slthree}}     

\newcommand{\ubp}[1][\kk]{\bpsymb^{#1}}            
\newcommand{\bpminmod}[1][\uu,2]{\bpsymb(#1)}      
\newcommand{\slminmod}[1][\uu,2]{\slthreesymb(#1)} 

\newcommand{\qhrcomplex}{\VOA{C}}                  
\newcommand{\qhrmodule}[2]{\qhrcomplex^{#1}(#2)}   

\newcommand{\ccsl}[1][\kk]{\affine{\cc}(#1)}       

\newcommand{\ccbg}{\cc_{\bgsymb}}                  
\newcommand{\ccfg}{\cc_{\fgsymb}}                  
\newcommand{\cclvoa}[1][\kk]{\cc_{\lsymb}(#1)}     
\newcommand{\ccbp}[1][\kk]{\cc_{\bpsymb}(#1)}      

\newcommand{\conjsymb}{\upsilon}
\newcommand{\sfsymb}{\sigma}

\newcommand{\wref}[1]{\mathsf{w}_{#1}}         
\newcommand{\dynk}{\mathsf{d}}                 
\newcommand{\slconj}{\conjsymb}                
\newcommand{\slsf}[1]{\affine{\sfsymb}^{#1}}   

\newcommand{\bpconj}{\conjsymb_{\bpsymb}}      
\newcommand{\bpsf}[1]{\sfsymb_{\bpsymb}^{#1}}  

\newcommand{\bgconj}{\conjsymb_{\bgsymb}}      
\newcommand{\bgsf}[1]{\sfsymb_{\bgsymb}^{#1}}  

\newcommand{\lsf}[1]{\sfsymb_{\lsymb}^{#1}}    

\DeclarePairedDelimiterXPP{\zhu}[1]{\mathsf{Zhu}}{\lbrack}{\rbrack}{}{#1}

\newcommand{\admwts}[1][\uu,2]{\textnormal{Adm}_{#1}} 

\newcommand{\bgver}{\Mod{V}}                       
\newcommand{\bgrel}[1]{\Mod{W}_{[#1]}}             
\newcommand{\lmod}[1]{\lsymb_{[#1]}}               
\newcommand{\bpirr}[1]{\Mod{H}_{#1}}               
\newcommand{\slirr}[1]{\affine{\Mod{L}}_{#1}}      
\newcommand{\slver}[1]{\affine{\Mod{M}}_{#1}}      
\newcommand{\fslirr}[1]{\Mod{L}_{#1}}              
\newcommand{\fslver}[1]{\Mod{M}_{#1}}              
\newcommand{\slext}[1]{\affine{\Mod{N}}_{#1}}      
\newcommand{\slsem}[2]{\affine{\Mod{S}}_{#1,[#2]}} 
\newcommand{\slrel}[2]{\affine{\Mod{R}}_{#1,[#2]}} 

\newcommand{\bpconfwt}[1]{\Delta^{\bpsymb}_{#1}}   
\newcommand{\slconfwt}[1]{\affine{\Delta}_{#1}}    

\newcommand{\slcoh}[1]{\Mod{C}_{#1}}               

\newcommand{\ocat}{\categ{O}}     

\DeclareMathOperator{\tr}{tr}

\newcommand{\ssemi}{\textnormal{semi}}
\newcommand{\shw}{\textnormal{hw}}

\newcommand{\Gr}[1]{\sqbrac[\big]{#1}}                 

\newcommand{\traceover}[1]{\tr_{\raisebox{-2pt}{$\scriptstyle #1$}}} 

\DeclareMathOperator{\chmap}{ch}
\DeclareMathOperator{\schmap}{sch}

\newcommand{\ch}[1]{\chmap \Gr{#1}}                    
\newcommand{\fch}[2]{\ch{#1} \brac[\big]{#2}}          
\newcommand{\sch}[1]{\schmap \Gr{#1}}                  
\newcommand{\fsch}[2]{\sch{#1} \brac[\big]{#2}}

\newcommand{\jth}[1]{\vartheta_{#1}}                   
\newcommand{\fjth}[2]{\jth{#1} \brac{#2}}              

\newcommand{\Smat}{\mathsf{S}}
\newcommand{\slSmat}[1]{\affine{\Smat}_{#1}}           
\newcommand{\bgSmat}[1]{\Smat^{\bgsymb}_{#1}}          
\newcommand{\bpSmat}[1]{\Smat^{\bpsymb}_{#1}}          
\newcommand{\piSmat}[1]{\Smat^{\lvoa}_{#1}}            

\newcommand{\Tmat}{\mathsf{T}}
\newcommand{\slTmat}[1]{\affine{\Tmat}_{#1}}           

\newcommand{\fuse}{\mathbin{\boxtimes}}
\newcommand{\slfuscoeff}[4][\uu,2]{\affine{\mathbb{N}}_{#2,#3}^{(#1)\,#4}}   
\newcommand{\bpfuscoeff}[4][\bpsymb]{\mathbb{N}_{#2,#3}^{#1\,#4}}          

\newcommand{\qhrfunc}[1]{\Phi^{\textnormal{#1}}} 
\newcommand{\qhrmin}{\qhrfunc{min.}}             

\newcommand{\bp}{Bershadsky--Polyakov}
\newcommand{\ep}{Euler--Poincar\'{e}}

\newcommand{\gt}{Gelfand--Tsetlin}
\newcommand{\km}{Kac--Moody}

\newcommand{\pbw}{Poincar\'{e}--Birkhoff--Witt}
\newcommand{\wzw}{Wess-Zumino-Witten}
\newcommand{\zam}{Zamolodchikov}

\newcommand{\fdim}{finite-dimensional}
\newcommand{\infdim}{in\fdim}
\newcommand{\lhs}{left-hand side}

\newcommand{\rhs}{right-hand side}

\newcommand{\hw}{highest-weight}
\newcommand{\hwv}{\hw\ vector}
\newcommand{\hwvs}{\hwv s}
\newcommand{\hwm}{\hw\ module}
\newcommand{\hwms}{\hwm s}

\newcommand{\rhw}{relaxed highest-weight}
\newcommand{\rhwv}{\rhw\ vector}

\newcommand{\rhwm}{\rhw\ module}
\newcommand{\rhwms}{\rhwm s}

\newcommand{\vo}{vertex operator}
\newcommand{\voa}{\vo\ algebra}
\newcommand{\voas}{\voa s}
\newcommand{\svoa}{\vo\ superalgebra}

\newcommand{\va}{vertex algebra}

\newcommand{\cft}{conformal field theory}
\newcommand{\cfts}{conformal field theories}
\newcommand{\emt}{energy-momentum tensor}
\newcommand{\emts}{\emt s}
\newcommand{\ope}{operator product expansion}
\newcommand{\opes}{\ope s}
\newcommand{\qhr}{quantum hamiltonian reduction} 
\newcommand{\qhrs}{\qhr s}
\newcommand{\uea}{universal enveloping algebra}

\makeatletter
\renewcommand\author@andify{%
  \nxandlist {\unskip ,\penalty-1 \space\ignorespaces}%
    {\unskip {} \@@and~}%
    {\unskip \penalty-2 \space \@@and~}%
}
\makeatother

\DeclareRobustCommand{\SkipTocEntry}[5]{}

\usepackage{xpatch}
\makeatletter
\xpatchcmd{\@tocline}
          {\hfil\hbox to\@pnumwidth{\@tocpagenum{#7}}\par}
          {\ifnum#1<0\hfill\else\dotfill\fi\hbox to\@pnumwidth{\@tocpagenum{#7}}\par}
          {}{}
\makeatother

\begin{document}

\title{Modularity of admissible-level $\slthree$ minimal models with denominator $2$}

\author[J~Fasquel]{Justine Fasquel}
\address[Justine Fasquel]{
	School of Mathematics and Statistics \\
	University of Melbourne \\
	Parkville, Australia, 3010.
}
\email{justine.fasquel@unimelb.edu.au}

\author[C~Raymond]{Christopher Raymond}
\address[Christopher Raymond]{
	 Department of Mathematics \\
	 University of Hamburg \\
	 Hamburg 20146 Germany.
}
\email{christopher.raymond@uni-hamburg.de}

\author[D~Ridout]{David Ridout}
\address[David Ridout]{
	School of Mathematics and Statistics \\
	University of Melbourne \\
	Parkville, Australia, 3010.
}
\email{david.ridout@unimelb.edu.au}

\begin{abstract}
	We use the newly developed technique of inverse quantum hamiltonian reduction to investigate the representation theory of the simple affine vertex algebra $\slminmod$ associated to $\slthree$ at level $\kk = -3+\frac{\uu}{2}$, for $\uu\ge3$ odd.
	Starting from the irreducible modules of the corresponding simple Bershadsky-Polyakov vertex operator algebras, we show that inverse reduction constructs all irreducible lower-bounded weight $\slminmod$-modules.
	This proceeds by first constructing a complete set of coherent families of fully relaxed highest-weight $\slminmod$-modules and then noting that the reducible members of these families degenerate to give all remaining irreducibles.
	Using this fully relaxed construction and the degenerations, we deduce modular S-transforms for certain natural generalised characters of these irreducibles and their spectral flows.
	With this modular data in hand, we
	verify that the (conjectural) standard Verlinde formula predicts Grothendieck fusion rules with nonnegative-integer multiplicities.
\end{abstract}

\maketitle

\markleft{J~FASQUEL, C~RAYMOND AND D~RIDOUT} 

\tableofcontents

\onehalfspacing

\section{Introduction}

This paper is a sequel, in a sense, to \cite{KawAdm21}.
There, a classification was proven for the irreducible lower-bounded weight modules of the simple affine vertex algebras associated to $\slthree$ at admissible levels.
Moreover, the modular properties of the characters of these irreducibles were computed for the level $\kk=-\frac{3}{2}$.
Here, we use inverse \qhr\ to extend this latter result to levels of the form $\kk=-3+\frac{\uu}{2}$, where $\uu\in\set{3,5,7,\dots}$.
We will consider the much more challenging (and interesting!) case of denominators greater than $2$ in the future.

\subsection{Background}

Affine \voas\ are fundamental objects in modern mathematical physics.
They are also studied by pure mathematicians for their intrinsic beauty as well as their applications in combinatorics, geometry, number theory, representation theory and the theory of tensor categories (to name just a few).
Physically, they arise in many different contexts, but perhaps most famously as the chiral symmetry algebras of certain two-dimensional \cfts\ called \wzw\ models \cite{WitNon84}.
Under favourable circumstances, these models are rational \cfts: their quantum state space decomposes, as a module over two commuting copies of the corresponding affine \voa, into a finite direct sum of irreducibles.
However, there are also many reasons to study non-rational \wzw\ models, see \cite{BeeInf13} for an example.

From a physical point of view, one of the hallmarks of a consistent \cft\ is the modular invariance of its partition function (the character of its quantum state space).
On the mathematical side, a theorem of Zhu \cite{ZhuMod96} shows that certain generalised characters (one-point functions) span a representation of the modular group, assuming that the \voa\ satisfies some rather restrictive finiteness conditions.
If one further restricts to strongly rational \voas, then Huang has proven that the module category is a (finite semisimple) modular tensor category \cite{HuaVer04a}.
Given the physical expectation of modularity, it is therefore important to explore the modular properties of \voas\ that do not satisfy these restrictions.

One is thus naturally led to study the modularity of general affine \voas.
This is by no means straightforward as one has to first identify a candidate category, classify its simple objects and then determine their (generalised) characters.
If an action of the modular group on these characters can be verified, then one is left with the formidable task of verifying that this category closes under the fusion product and that this product equips the category with a tensor structure.
Finally, one wants to connect the modular and tensor structures by means of some sort of generalisation of the Verlinde formula \cite{VerFus88}.

Natural examples with which to start include the affine \voas\ associated to $\sltwo$.
However, early attempts to further the modular program for them floundered \cite{KacMod88,KohFus88}.
Eventually, it was realised that the seemingly natural choice for the category --- the analogue of the BGG category $\categ{O}$ --- was not closed under fusion \cite{GabFus01} and its characters did not carry an action of the modular group \cite{RidSL208}.
A better candidate is the category of finitely generated weight modules: its simple objects were classified in \cite{AdaVer95,FutIrr96,AdaWei23}, their characters were determined in \cite{KawRel18} and their modularity was demonstrated in \cite{CreMod12,CreMod13}.
However, this category is both nonfinite and nonsemisimple.

The latter works also showed that a natural generalisation of the Verlinde formula, called the standard Verlinde formula in \cite{CreLog13,RidVer14}, correctly connected the modular structure of the weight category to what was then known \cite{GabFus01,RidFus10} of the tensor structure.
Further work on elucidating this tensor structure has recently appeared \cite{CreRig22,CreTen23,NakRel24}.

With $\sltwo$ affine \voas\ reasonably well in hand, the next class of examples to attack are those associated with $\slthree$.
This turns out to be significantly more difficult.
The simple objects of the weight module category, now with the added restriction that the weight spaces are \fdim, were classified in \cite{AraWei16} using \gt\ combinatorics.
An alternative general classification scheme was subsequently proposed in \cite{KawRel19} based on Mathieu's theory \cite{MatCla00} of coherent families and twisted localisation.
This was carried out for $\slthree$ in \cite{KawAdm21} where one can find not only the classification of irreducibles, described in terms of coherent families, but also the detailed structures of the reducible members of these families.
This was used to successfully investigate the modularity in the special case in which the level $\kk$ of the affine \voa\ was set to $-\frac{3}{2}$.

As mentioned above, our goal with this paper is to generalise the modularity results of \cite{KawAdm21} to levels of the form $\kk = -3+\frac{\uu}{2}$, for $\uu\in\set{3,5,7,\dots}$.
This is by no means straightforward for the simple reason that $\kk=-\frac{3}{2}$, hence $\uu=3$, is the only level for which the characters of the irreducibles are linearly independent and thus suitable to determine the action of the modular group.
Unfortunately, the standard generalised characters (Zhu's one-point functions) do not appear to be computable.
We need another approach.

\subsection{Inverse \qhr}

An exciting approach to the representation theory of affine \voas\ was recently introduced by Adamovi\'{c} in \cite{AdaRea17}.
Building on an old idea of Semikhatov \cite{SemInv94}, this work showed how to construct weight modules for the $\sltwo$ affine \voa\ from modules of the Virasoro algebra and a certain free-field algebra.
More generally, the idea is to construct affine modules from those of their W-algebras (and free-field algebras).
As W-algebras are typically defined from affine \voas\ via \qhr\ \cite{FeiAff92,KacQua03}, this construction technology has come to be known as inverse \qhr.
We refer to \cite{AdaRea20,AdaRel21,FehSub21,AdaWei23,FehInv23,FasOrt23,CreStr24} for recent illustrations of this technology in action.

One big advantage of inverse reduction is that it yields characters for those affine weight modules that can be explicitly constructed from W- and free-field modules.
If the modularity of the latter are known and inverse reduction constructs sufficiently many weight modules, then one can attack the modularity of the affine weight category.
Happily, there are many W-algebras that are known to be strongly rational, hence modular, including the regular (or principal) ones at nondegenerate admissible level \cite{AraRat12b} and, more generally, the so-called exceptional ones of \cite{AraRat19}.
Moreover, techniques to prove that inverse reduction constructs essentially all weight modules have recently been developed in \cite{AdaRea20,AdaWei23}.

In the case of $\slthree$ affine \voas, there are two non-trivial hamiltonian reductions: the regular reduction, which produces the $\VOA{W}_3$- (or \zam) algebra, and the minimal reduction, which gives the \bp\ algebra.
Inverse reduction constructs \bp\ modules from $\VOA{W}_3$-modules \cite{AdaRea20} and affine modules from \bp\ modules \cite{AdaRel21}.
The modularity corresponding to the former inverse was analysed in \cite{FehMod21}.
Our modularity study will use instead the latter one.

It is important here that we restrict to a level $\kk$ whose denominator is $2$.
Then, the \bp\ \voa\ is an exceptional W-algebra, hence is strongly rational \cite{AraRat10}.
Moreover, this implies that the weight modules constructed by inverse reduction will have \fdim\ weight spaces and so we have a classification of irreducibles \cite{KawRel19}.
In this case, we expect that the category of finitely generated weight modules with \fdim\ weight spaces will be a (nonfinite nonsemisimple) modular tensor category.
Our results here constitute a highly nontrivial first step towards verifying this expectation.

We remark that when the denominator is greater than $2$, inverse reduction will construct $\slthree$ affine \voa\ modules with \infdim\ weight spaces \cite{AdaRel21}.
Unfortunately, it appears that there are currently no classification results for these modules, nor even a good mathematical theory for them.
We intend to report on progress towards such a theory in the future.

\subsection{Outline and results}

We start \cref{sec:sl3} with a quick overview of $\slthree$ affine \voas\ and their representations.
In particular, we introduce admissibility \cite{KacMod88} as it pertains to \hwms, characters and modular transformations.
We also emphasise the all-important issue of convergence regions for \hw\ characters.
The special case in which the level is a nonnegative integer is discussed, as it and the corresponding fusion rules will be useful later.

\cref{sec:sl3cohfam} takes us beyond \hwms, recalling Mathieu's coherent families and (twisted) localisation functors.
These are used to understand the irreducible weight modules of $\slthree$ affine \voas, following \cite{KawRel19,KawAdm21}.
In \cref{sec:sl3rhwm}, we study localisations of the \hwms\ when $\kk=-3+\frac{\uu}{2}$, identifying the submodule structure of the results and proving that their characters vanish identically when convergence regions are ignored.

In \cref{sec:bp}, we introduce the minimal \qhrs\ of the $\slthree$ affine \voas, namely the \bp\ algebras.
These are strongly rational when $\kk=-3+\frac{\uu}{2}$ and our goal is to determine the modular properties of their irreducible characters and so derive their fusion rules.
In principle, these results could be deduced as special cases of the work of Arakawa and van~Ekeren on exceptional W-algebras \cite{AraRat19}.
However, we prefer a more pedagogical approach that computes the characters directly from the definition of minimal reduction.

We mention that our proof of the modularity of these characters is different and may be of independent interest.
The method employed in \cite[Thm.~10.4]{AraRat19} is transcendental, following \cite{FreCha92}, involving a l'H\^{o}pital limit in a direction specified by an (almost) arbitrary element $x$ in the weight space.
This results in a curious situation in which the S-matrix appears to depend on $x$, but actually (and rather nontrivially) does not.
Our approach, detailed in \cref{thm:bpmod}, avoids such a choice of $x$ and instead directly expresses the S-matrix in terms of $\slthree$ data.
More precisely, the \bp\ S-matrix is revealed to be a sum of $\slthree$ S-matrix coefficients, each dressed with a phase involving a ratio of T-matrix coefficients related by spectral flow and a localisation of the corresponding \hwm.
It would be very interesting to investigate generalisations of this identification in other examples.

\cref{sec:conventions} then reviews inverse \qhr\ in the special case in which one constructs modules for $\slthree$ affine \voas\ from those of the \bp\ algebra, the bosonic ghost algebra and its lattice bosonisation.
This inverse \qhr\ is originally due to \cite{AdaRel21}, but our explicit formulae are a little different.
We review the representation theory and modularity of the weight categories of the latter.
As far as we can tell, the modular properties of the lattice bosonisation are new (see \cref{prop:lmod}).

We finally commence in \cref{sec:mod} our study of the modularity of the weight module category of the $\slthree$ affine \voas\ with $\kk=-3+\frac{\uu}{2}$, $\uu\in\set{3,5,7,\dots}$.
Our first main result (\cref{thm:completeness}) is that inverse reduction constructs a complete set of ``fully relaxed'' modules.
This proceeds by comparing their top spaces with the coherent family classification of \cite{KawAdm21}.
For this, one needs to analyse the ``degeneration'' of the reducible fully relaxed modules into semirelaxed and \hwms.

Recalling that the spectral flows of the fully relaxed characters are linearly dependent unless $\uu=3$, we notice that inverse reduction naturally defines generalised characters that are always linearly independent.
The modular S-transform of these generalised characters then follows easily in \cref{thm:Sfully}.
This result, expressed in terms of \bp\ and free-field data is a little complicated.
Much of this complication however evaporates upon swapping the free-field parametrisation for a Lie-theoretic one.
The fully relaxed S-matrix is thus found (\cref{cor:Sfully}) to be proportional to the \bp\ S-matrix of \cref{thm:bpmod}.
Moreover, the proportionality constant is a pure phase depending only on $\slthree$ data, namely spectral flows and weights.

With the fully relaxed modular S-transform in hand, we turn to that of the semirelaxed modules.
Here, the degeneration of the reducible fully relaxed modules is converted into coresolutions for the semirelaxed modules in terms of fully relaxed ones \cite{CreRel11}.
From these, we extract the semirelaxed S-transforms (\cref{prop:Ssemi,thm:Ssemiw1}).
A similar, but more involved, analysis extends these results to the \hw\ generalised characters (\cref{thm:Shw}).
In particular, the S-transform of the vacuum module's generalised character is obtained in \cref{cor:Svac}.

We complete our modularity study by substituting our results into the standard Verlinde formula of \cite{CreLog13,RidVer14}.
This is a conjectural generalisation of the famous Verlinde formula for fusion coefficients that is expected to apply to weight categories for affine \voas\ and W-algebras, see \cite{CreMod12,CreMod13,RidAdm17,KawAdm21,FehMod21}.
In \cref{thm:relfus}, we find that this formula predicts that the fusion product of two fully relaxed modules decomposes, for generic parameters, as a direct sum of fully relaxed modules.
The explicit formula is a nontrivial generalisation of the corresponding result of \cite{KawAdm21} for $\uu=3$, from which it differs by the appearance of \bp\ fusion coefficients and the order-$3$ outer automorphism $\nabla$ of $\aslthree$.

We finish by further testing our results' consistency by deducing predictions for the (Grothendieck) fusion rules involving semirelaxed and \hwms\ (\cref{cor:allfusions}).
In each case, our computations result nontrivially in nonnegative-integer multiplicities, confirming our methodology.
We view these results as strong affirmations of our hypothesis that inverse \qhr\ is the tool of choice to explore the structure of weight module categories for general affine \voas\ and W-algebras.

\addtocontents{toc}{\SkipTocEntry}
\subsection*{Acknowledgements}

We thank Jethro van Ekeren for valuable discussions concerning the modularity of exceptional W-algebras.
JF's research is supported by a University of Melbourne Establishment Grant.
CR's research is supported by the Australian Research Council Discovery Project DP200100067.
DR's research is supported by the Australian Research Council Discovery Project DP210101502 and an Australian Research Council Future Fellowship FT200100431.

\section{Admissible-level $\slthree$ minimal models} \label{sec:sl3}

We start by constructing the universal and simple affine \voas\ associated to $\slthree$ at a noncritical level $\kk\ne-3$.
This is complemented by a review of some aspects of the representation theory of the latter, when $\kk$ is admissible, starting with \hwms\ and moving on to weight modules with finite multiplicities.

\subsection{Vertex operator algebras associated to $\slthree$} \label{sec:sl3voa}

Let $E_{ij}$ denote the $3 \times 3$ elementary matrix with $1$ in the $(i,j)$ entry and $0$ elsewhere.
We will work with the following basis of the complex simple Lie algebra $\slthree$:
\begin{equation}
	f^{3} = E_{31},
	\qquad
	\begin{aligned}
		f^{1} &= E_{21}, \\
		f^{2} &= E_{32},
	\end{aligned}
	\qquad
	\begin{aligned}
		h^{1} &= E_{11} - E_{22}, \\
		h^{2} &= E_{22} - E_{33},
	\end{aligned}
	\qquad
	\begin{aligned}
		e^{1} &= E_{12}, \\
		e^{2} &= E_{23},
	\end{aligned}
	\qquad
	e^{3} = E_{13}.
\end{equation}
The Lie bracket is given by the matrix commutator.
We let $\csub = \spn_{\CC}\set{h^1, h^2}$ be the Cartan subalgebra and normalise the Killing form such that the nonzero entries are given by
\begin{equation}
	\bilin{e^i}{f^j} = \bilin{f^j}{e^i} = \delta^{ij}, \quad \bilin{h^k}{h^{\ell}} = A^{k\ell}, \qquad i,j = 1,2,3; \ k,\ell = 1,2.
\end{equation}
Here, $A^{k\ell}$ denotes the corresponding entry of the Cartan matrix $A = \begin{psmallmatrix} 2 & -1 \\ -1 & 2 \end{psmallmatrix}$.
The bilinear form induced on the weight space $\csub^*$ will also be denoted by $\bilin{\blank}{\blank}$ as will the pairing between $\csub$ and its dual.

For $i=1,2$, let $\srt{i}$, $\scrt{i}$, $\fwt{i}$ and $\fcwt{i}$ denote the simple roots, simple coroots, fundamental weights and fundamental coweights of $\slthree$, respectively.
Their $\ZZ$-spans are the root lattice $\rlat \subset \csub^*$, coroot lattice $\crlat \subset \csub$, weight lattice $\wlat \subset \csub^*$ and coweight lattice $\cwlat \subset \csub$, respectively.
We denote the highest root by $\srt{3} = \srt{1}+\srt{2}$ and introduce counterparts $\scrt{3}=\scrt{1}+\scrt{2}$, $\fwt{3}=\fwt{1}-\fwt{2}$ and $\fcwt{3} = \fcwt{1}-\fcwt{2}$.
The Weyl reflection defined by $\srt{i}$, $i=1,2,3$, will be denoted by $\wref{i}$.
These reflections generate the Weyl group $\grp{S}_3$ of $\slthree$.

The untwisted affine \km\ algebra $\aslthree$ is the Lie algebra specified, for $x,y \in \slthree$ and $m,n \in \ZZ$, by
\begin{equation}
	\aslthree = \slthree[t,t^{-1}] \oplus \CC K, \quad
	\comm{x_m}{y_n} = \comm{x}{y}_{m+n} + m \bilin{x}{y} \delta_{m+n,0} K, \quad
	\comm{\aslthree}{K}= 0.
\end{equation}
Here for brevity, we set $x_m = xt^m$, where $x \in \slthree$ and $m \in \ZZ$.
We extend our notation for simple (co)roots, fundamental (co)weights and Weyl reflections from $\slthree$ to $\aslthree$ by allowing the index $i$ to also take value $0$, trusting that context will distinguish them where necessary.
For example, the Weyl vector of $\aslthree$ is $\wvec = \fwt{0} + \fwt{1} + \fwt{2}$.
\begin{definition}
	The universal affine \va\ $\usl$ at level $\kk \in \CC$ is the parabolic Verma $\aslthree$-module obtained by taking the trivial $\slthree[t]$-module, letting $K$ act as multiplication by $\kk$, and inducing.
	The \va\ structure is strongly and freely generated by fields
	\begin{equation}
		x(z) = \sum_{n \in \ZZ} x_{n} z^{-n-1}, \quad x \in \slthree,
	\end{equation}
	that satisfy the following \ope\ relations:
	\begin{equation}
		x(z)y(w) \sim \frac{\bilin{x}{y}\,\kk\,\wun}{(z-w)^2} + \frac{\comm{x}{y}(w)}{z-w}, \quad x,y \in \slthree.
	\end{equation}
\end{definition}
\noindent We shall throughout denote the identity field of a given \voa\ by $\wun$, trusting that this will not cause any confusion.

For $\kk \ne -3$, we can introduce a conformal structure on $\usl$ via the Sugawara construction.
The resulting \emt\ has the form
\begin{equation} \label{eq:emtsl3}
	T(z) = \frac{1}{2(\kk+3)} \sum_{i=1}^3 \sqbrac*{\frac{1}{3} \no{h^ih^i}(z) - \pd h^i(z) + 2 \no{e^if^i}(z)}.
\end{equation}
The mode expansion $T(z) = \sum_{n \in \ZZ} T_n z^{-n-2}$ generates a representation of the Virasoro algebra with central charge
\begin{equation}
	\ccsl = \frac{8 \kk}{\kk + 3}.
\end{equation}
The strong generators $x(z)$, $x \in \slthree$, of $\usl$ are primary fields of conformal weight $1$ with respect to $T(z)$.

\begin{theorem}[\protect{\cite[Thm.~0.2.1]{GorSim06}}] \label{thm:sl3simpquot}
	The \voa\ $\usl$, $\kk \ne -3$, is simple unless $\kk$ satisfies
	\begin{equation} \label{eq:admlevel}
		\kk + 3 = \frac{\uu}{\vv}, \quad \uu \in \ZZ_{\ge2}, \ \vv \in \ZZ_{\ge1}, \ \gcd \set{\uu,\vv} = 1.
	\end{equation}
\end{theorem}
\noindent Being a parabolic Verma module, $\usl$ has a unique simple quotient for $\kk\neq-3$.
\begin{definition} \label{def:admlevel}
	The level $\kk$ is admissible if \eqref{eq:admlevel} holds with $\uu \in \ZZ_{\ge3}$.
	When \eqref{eq:admlevel} holds, we shall denote the simple quotient of $\usl$ by $\slminmod[\uu,\vv]$ and refer to the resulting \voa\ as an $\slthree$ minimal model.
\end{definition}

\subsection{Highest-weight representation theory} \label{sec:sl3hwm}

It is well known \cite{FreVer92} that being a $\usl$-module is equivalent to being a ``smooth'' $\aslthree$-module of level $\kk$.
The subcategory of $\slminmod[\uu,\vv]$-modules is far more interesting.
For $\kk$ admissible, the irreducible \hw\ $\slminmod[\uu,\vv]$-modules were classified by Arakawa (see \cref{thm:sl3hwclass} below).
However, if $\vv\ne1$, then these are far from the only irreducible $\slminmod[\uu,\vv]$-modules.
For example, there are irreducible \rhwms\ and their spectral flows (see \cref{sec:sl3rhwm}).

In \cite{KacMod88}, the notion of admissible weight of an affine \km\ algebra was introduced.
We specialise this notion to $\aslthree$, following \cite{MatPri98}.
Let $\pwlat[n]$ denote the set of $\aslthree$-weights $\mu$ whose Dynkin labels $(\mu_0,\mu_1,\mu_2)$ are nonnegative integers that sum to $n$.
Let $\wref{} \cdot \mu = \wref{}(\mu+\wvec)-\wvec$ denote the shifted action of the Weyl group of $\aslthree$.
\begin{definition} \label{def:admwts}
	Given an admissible level $\kk$, a level-$\kk$ weight of $\aslthree$ is said to be admissible if it has the form
	\begin{equation} \label{eq:defadmwts}
		\mu = y_{\mu} \cdot \brac[\big]{\mu^I - \tfrac{\uu}{\vv} \mu^{F,y_{\mu}}}, \quad
		y_{\mu} \in \set{\id,\wref{1}},\ \mu^I \in \pwlat,\ \mu^{F,y_{\mu}} \in \pwlat[\vv-1],\ \mu^{F,\wref{1}}_1 \ne 0.
	\end{equation}
	We shall denote this $\kk$-dependent set of admissible $\aslthree$-weights by $\admwts[\uu,\vv]$.
\end{definition}
\noindent The set of admissible weights is naturally partitioned \cite[Prop.~2.1]{KacMod88b} as follows:
\begin{equation} \label{eq:admdisjoint}
	\admwts[\uu,\vv] = \admwts[\uu,\vv]^{\id} \sqcup \admwts[\uu,\vv]^{\wref{1}}.
\end{equation}
Here, $\admwts[\uu,\vv]^y$ denotes the subset of admissible weights $\mu$ with $y_{\mu} = y$.

Let $\slirr{\mu}$ denote the irreducible level-$\kk$ \hw\ $\aslthree$-module of highest weight $\mu$.
This module is smooth and so admits an action of $\usl$.
For $\kk\ne-3$, it is graded by the eigenvalue of the Virasoro zero mode $T_0$ and, by the Sugawara construction \eqref{eq:emtsl3}, the minimal eigenvalue is
\begin{equation}
	\slconfwt{\mu} = \frac{\bilin{\mu}{\mu+2\wvec}}{2(\kk+3)}.
\end{equation}

We recall the following quite general \lcnamecrefs{def:topspace}.
\begin{definition} \label{def:topspace}
	Consider a module that is graded by the eigenvalues of the action of some fixed Virasoro zero mode.
	\begin{itemize}
		\item Such a module is said to be lower bounded if the real parts of these grades are bounded from below.
		\item If there exists a minimal grade, then the top space of a lower-bounded module is the (generalised) eigenspace corresponding to this minimal grade.
	\end{itemize}
\end{definition}
\noindent In this paper, we only consider lower-bounded modules with top spaces.

Being \hw, $\slirr{\mu}$ is always lower bounded and its top space may be identified with the irreducible \hw\ $\slthree$-module $\fslirr{\bar{\mu}}$ of highest weight $\bar{\mu} = \mu_1\fwt{1}+\mu_2\fwt{2}$.
If $\mu \in \admwts[\uu,\vv]^{\id}$, then the top space of $\slirr{\mu}$ is \fdim\ if and only if $\mu^{F,\id} = 0$.
If $\mu \in \admwts[\uu,\vv]^{\wref{1}}$, then the top space is always \infdim\ because \eqref{eq:defadmwts} gives $\mu_1 = -\mu^I_1 - 2 + \frac{\uu}{\vv} \mu^{F,\wref{1}}_1 \notin \NN$.

The admissible weights turn out to be relevant to the classification of \hw\ $\slminmod[\uu,\vv]$-modules.
\begin{theorem}[\protect{\cite[Main Thm.]{AraRat12}}] \label{thm:sl3hwclass}
	When $\kk$ is admissible, every \hw\ $\slminmod[\uu,\vv]$-module is irreducible.
	Moreover, $\slirr{\mu}$ is an $\slminmod[\uu,\vv]$-module if and only if $\mu \in \admwts[\uu,\vv]$.
\end{theorem}

Given a level-$\kk$ $\aslthree$-module $\slver{}$, its formal character is defined to be
\begin{equation} \label{eq:slformalchar}
	\ch{\slver{}} = \qq^{-\ccsl/24} \sum_{\lambda \in \affine{\csub}^*} \dim\slver{}(\lambda) \, \ee^{\lambda},
\end{equation}
where $\affine{\csub} = \spn_{\CC}\set{h^1_0,h^2_0,K,T_0}$, $\slver{}(\lambda)$ is the weight space of weight $\lambda$, and $\ee^{\lambda}$ is a formal exponential.
It is common to make this concept more concrete by writing $\lambda$ as a linear combination of basis elements of $\affine{\csub}$ and defining new variables as exponentials of the coefficients.
For example,
\begin{equation} \label{eq:slchar}
	\fch{\slver{}}{\zz_1,\zz_2;\qq} = \traceover{\slver{}} \zz_1^{\scrt{1}_0} \zz_2^{\scrt{2}_0} \qq^{T_0-\ccsl/24}.
\end{equation}
We omit, for brevity, the variable corresponding to $K$ as it would only contribute a constant prefactor.

When $\slver{} = \slirr{\mu}$ and $\mu \in \admwts[\uu,\vv]$, there is a character formula generalising that of Weyl.
\begin{theorem}[\protect{\cite[Thm.~1]{KacMod88}}] \label{thm:sl3admchar}
	Let $\kk$ be admissible and let $\mu \in \admwts[\uu,\vv]$.
	Then, the formal character of $\slirr{\mu}$ is
	\begin{equation}
		\ch{\slirr{\mu}} = \frac{\sum_{w \in \affine{\grp{W}}(\mu)} \det w \, \ee^{w \cdot \mu}}{\sum_{w \in \affine{\grp{W}}(0)} \det w \, \ee^{w \cdot 0}},
	\end{equation}
	where $\affine{\grp{W}}(\lambda)$ is the integral affine Weyl group of $\lambda \in \affine{\csub}^*$, defined to be the subgroup generated by the affine Weyl reflections corresponding to the real roots $\alpha$ satisfying $\bilin{\lambda+\wvec}{\co{\alpha}} \in \ZZ$.
\end{theorem}
\noindent There are also formulae for the modular transforms of these characters.
For this, we write $\zz_1 = \ee^{2\pi\ii\zeta_1}$, $\zz_2 = \ee^{2\pi\ii\zeta_2}$ and $\qq = \ee^{2\pi\ii\tau}$.
\begin{theorem}[\protect{\cite[Thm.~2(c)]{KacMod88}}] \label{thm:sl3Smatrix}
	Let $\kk$ be admissible and let $\mu \in \admwts[\uu,\vv]$.
	Write $\mu$ in the form \eqref{eq:defadmwts} and let $\mu^F = y_{\mu}(\mu^{F,y_{\mu}})$.
	Then, the S-transform of the character of $\slirr{\mu}$ is given, up to an omitted automorphy factor, by
	\begin{subequations}
		\begin{gather}
			\fch{\slirr{\mu}}{\tfrac{\zeta_1}{\tau},\tfrac{\zeta_2}{\tau};-\tfrac{1}{\tau}}
			= \:\sum_{\mathclap{\nu \in \admwts[\uu,\vv]}}\: \slSmat{\mu,\nu} \fch{\slirr{\nu}}{\zeta_1,\zeta_2;\tau}, \label{eq:sl3Stransform} \\
			\slSmat{\mu,\nu} = \frac{-\ii}{\sqrt{3}\uu\vv} \det(y_{\mu}y_{\nu}) \, \ee^{2\pi\ii \brac[\big]{\bilin{\mu^I+\wvec}{\nu^F} + \bilin{\mu^F}{\nu^I+\wvec} - \bilin{\mu^F}{\nu^F} \uu/\vv}} \sum_{\wref{}\in \grp{S}_3} \det(\wref{}) \, \ee^{-2\pi\ii \bilin*{\wref{}(\mu^I+\wvec)}{\nu^I+\wvec} \vv/\uu}. \label{eq:sl3Smatrix}
		\end{gather}
	\end{subequations}
	The T-transform is given by
	\begin{equation} \label{eq:sl3Ttransform}
		\fch{\slirr{\mu}}{\zeta_1,\zeta_2;\tau+1} = \slTmat{\mu} \fch{\slirr{\mu}}{\zeta_1,\zeta_2;\tau}, \quad
		\slTmat{\mu} = \ee^{-2\pi\ii/3} \ee^{\pi\ii\norm{\mu+\wvec}^2 \vv/\uu}.
	\end{equation}
\end{theorem}

We mention that automorphy factors, which we shall ignore throughout, are certain nonconstant exponential prefactors that should appear in \eqref{eq:sl3Stransform}.
They depend on the modular parameters, here $\zeta_1$, $\zeta_2$ and $\tau$, but not on the module's parameters, here $\mu$.
These factors may be removed from \eqref{eq:sl3Smatrix}, without changing the S-matrix $\slSmat{}$, by reinserting in \eqref{eq:slchar} the additional variable corresponding to the central element $K$.
This new variable of course undergoes nontrivial modular transformations, see \cite[\S13]{KacInf90} and \cite[\S14.5]{DiFCon97} for further discussion.

\begin{remark} \label{rem:modularconvergence}
	An important observation is that when $\vv>1$, the characters of the $\slirr{\mu}$ with $\mu \in \admwts[\uu,\vv]$ do not converge for all $(\zz_1,\zz_2) \in \CC^2 \setminus \set{0}$ and $0 < \abs{\qq} < 1$.
	In fact, there are infinitely many $\qq$-dependent disjoint convergence regions in which to expand a character in terms of $\zz_1$ and $\zz_2$, only one of which corresponds to \eqref{eq:slchar}.
	Worse, the modular S-transform \eqref{eq:sl3Stransform} does not respect these convergence regions and so \eqref{eq:sl3Smatrix} does not, strictly speaking, hold for the characters \eqref{eq:slchar} but only for their meromorphic extensions to $(\zz_1,\zz_2) \in \CC^2$ \cite{RidSL208}.
	We will explain in \cref{sec:bpmod} why this is not a problem for the application at hand.
\end{remark}

\subsection{Automorphisms and twists} \label{sec:sl3auts}

The automorphisms of the root system of $\slthree$ form the dihedral group $\grp{D}_6$.
This is the product of the Weyl group $\grp{S}_3$ with the order-two group generated by the Dynkin symmetry $\dynk$ that exchanges the simple roots.
The conjugation automorphism, which negates all roots, is $\slconj = \wref{3}\dynk$.
Each such automorphism $\Omega$ lifts to an automorphism of the root system of $\aslthree$.
It also lifts, albeit nonuniquely, to an automorphism of $\aslthree$ itself that fixes
$K$, maps the root space labelled by the root $\alpha$ to that labelled by $\Omega(\alpha)$ and (by extending to a completion) also fixes $T_0$.
However, the relations satisfied by the $\Omega$ need not be satisfied by these lifts.
Instead, we only have a projective action of the root system's automorphism group $\grp{D}_6$ on $\aslthree$.

This is enough for each root system automorphism $\Omega$ to induce a functor $\Omega^*$ on the category of level-$\kk$ $\aslthree$-modules (and thus to the categories of $\usl$- and $\slminmod[\uu,\vv]$-modules).
We define $\Omega^*$ elementwise.
\begin{definition} \label{def:invfunctors}
	Let $\Mod{M}$ be a level-$\kk$ $\aslthree$-module and $\Omega$ a $K$-fixing automorphism of $\aslthree$.
	The level-$\kk$ $\aslthree$-module $\Omega^*(\Mod{M})$ is then the vector space $\set[\big]{\Omega^*(v) \st v \in \Mod{M}}$ (isomorphic to $\Mod{M}$) equipped with the following $\aslthree$-action:
	\begin{equation} \label{eq:autaction}
		x \, \Omega^*(v) = \Omega^* \brac[\big]{\Omega^{-1}(x) v}, \quad x \in \aslthree,\ v \in \Mod{M}.
	\end{equation}
\end{definition}
\noindent For convenience, we will always restrict to $K$-fixing automorphisms $\Omega$ that also preserve the subspace $\csub \oplus \CC K$.
It is then easy to check that $\Omega^*(\Mod{M})$ is weight if $\Mod{M}$ is, hence that $\Omega^*$ is an invertible endofunctor of the category of level-$\kk$ weight $\aslthree$-modules.
In particular, these functors are exact and preserve irreducibility.

The action \eqref{eq:autaction} indeed maps a weight vector $v$ of weight $\lambda$ to one of weight $\Omega(\lambda)$.
Dropping the stars that distinguish functors from automorphisms, this follows from the orthogonality of $\Omega$:
\begin{equation}
	h \Omega(v) = \Omega \brac[\big]{\Omega^{-1}(h) v} = \bilin[\big]{\lambda}{\Omega^{-1}(h)} \Omega(v) = \bilin[\big]{\Omega(\lambda)}{h} \Omega(v), \quad h \in \csub.
\end{equation}
Applying the functor $\Omega$ to an irreducible \hw\ $\aslthree$-module thus results in an irreducible \hwm\ with respect to a different Borel subalgebra.
(The new module is nevertheless still lower bounded because $\Omega(T_0) = T_0$.)
Only when $\Omega = \id$ or $\dynk$ is the Borel preserved; the latter case gives
\begin{equation}
	\dynk(\slirr{\mu}) \cong \slirr{\dynk(\mu)}.
\end{equation}

As noted in \cref{def:invfunctors}, this construction of invertible endofunctors extends to $K$-preserving automorphisms of $\aslthree$.
In particular, the spectral flow automorphisms define functors that need not preserve conformal weights.
These are the pure translations of the extended affine Weyl group in which the coroot lattice $\crlat$ is replaced by the coweight lattice $\cwlat$.
We denote the spectral flow automorphism (and its corresponding invertible endofunctor) corresponding to the coweight $g \in \cwlat$ by $\slsf{g}$.
Its action on $\aslthree$ is given explicitly by
\begin{equation} \label{eq:slsf}
	\begin{aligned}
		\slsf{g}(e^i_n) &= e^i_{n-\bilin{\srt{i}}{g}}, & \slsf{g}(h_n) &= h_n - \bilin{g}{h} \delta_{n,0} K, \\
		\slsf{g}(f^i_n) &= f^i_{n+\bilin{\srt{i}}{g}}, & \slsf{g}(T_n) &= T_n - g_n + \tfrac{1}{2} \norm{g}^2 \delta_{n,0} K,
	\end{aligned}
	\qquad i=1,2,3;\ h \in \csub;\ n \in \ZZ.
\end{equation}
It is easy to check that the expected dihedral relation
\begin{equation} \label{eq:dihedral}
	\Omega \slsf{g} \Omega^{-1} = \slsf{\Omega(g)}, \quad \Omega \in \grp{D}_6,\ g \in \cwlat,
\end{equation}
holds projectively on $\aslthree$.

It follows that spectral flow defines invertible endofunctors on the category of level-$\kk$ $\aslthree$-modules, as per \cref{def:invfunctors}, and that the dihedral relations \eqref{eq:dihedral} are satisfied functorially.
These functors moreover restrict \cite{LiPhy97} to the categories of $\usl$- and $\slminmod[\uu,\vv]$-modules.
We remark that while the functors $\Omega$ induced from the automorphisms of the root system of $\slthree$ preserve lower-boundedness, the same is not true for spectral flow functors.

We conclude by recording how spectral flow modifies the $\slthree$-weight $\lambda$ and conformal weight $\slconfwt{}$ of a weight vector $v$ in a level-$\kk$ $\aslthree$-module.
Using \eqref{eq:autaction} and \eqref{eq:slsf}, we find that the $\slthree$-weight and conformal weight of $\slsf{g}(v)$ are given by
\begin{equation} \label{eq:sfwts}
	\lambda + \kk \bilin{g}{\blank} \quad \text{and} \quad \slconfwt{} + \bilin{\lambda}{g} + \tfrac{1}{2} \norm{g}^2 \kk,
\end{equation}
respectively, where $\bilin{g}{\blank} \in \csub^*$.
It follows that the character of a level-$\kk$ $\aslthree$-module $\slver{}$ and that of its spectral flows are related by
\begin{equation} \label{eq:slsfchar}
	\fch{\slsf{g}(\slver{})}{\zz_1,\zz_2;\qq} = \zz_1^{\bilin{g}{\scrt{1}} \kk} \zz_2^{\bilin{g}{\scrt{2}} \kk} \qq^{\norm{g}^2 \kk/2} \fch{\slver{}}{\zz_1\qq^{g_1},\zz_2\qq^{g_2};\qq}, \quad g \in \cwlat,
\end{equation}
where $g = g_1 \scrt{1} + g_2 \scrt{2}$.

\subsection{$\SLG{SU}{3}$ \wzw\ models} \label{sec:sl3v=1}

The simple affine \voas\ with $\kk \in \NN$, such levels being admissible with $\vv=1$, were introduced by Witten \cite{WitNon84} as the chiral algebras of noncritical string theories called \wzw\ models.
Here, the target space of the string theory is a compact simple Lie group that is connected and simply connected.
The models on $\SLG{SU}{3}$ thus correspond to the $\slminmod[\uu,1]$, with $\uu \in \set{3,4,5,\dots}$.
The representation theory was subsequently worked out in \cite{GepStr86} and reformulated rigorously in the language of vertex algebras in \cite{FreVer92}.
Here, we review a few aspects that will be needed in what follows.

We first note that when $\vv=1$, the admissible weights \eqref{eq:defadmwts} all have $y_{\mu} = \id$ and $\mu^{F,y_{\mu}} = 0$, hence $\mu = \mu^I \in \pwlat$.
More importantly, the $\slirr{\mu}$ with $\mu \in \pwlat$ form a complete set of irreducible $\slminmod[\uu,1]$-modules.
This is a consequence of the following fundamental results.
\begin{theorem}[\protect{\cite[Thm.~3.1.3]{FreVer92} and \cite[\S5.3]{ZhuMod96}}] \label{thm:sl3rational}
	The $\slthree$ minimal model \voas\ $\slminmod[\uu,1]$, with $\uu \in \set{3,4,5,\dots}$, are $C_2$-cofinite and rational.
\end{theorem}
\begin{theorem}[\protect{\cite[Cor.~5.7]{AbeRat02}}] \label{thm:ordinary}
	Every irreducible module of a $C_2$-cofinite \voa\ is \hw\ with a \fdim\ top space.
\end{theorem}

Another consequence, although this was already noted in \cite{KacInf74}, is that the characters of the irreducible $\slminmod[\uu,1]$-modules $\slirr{\lambda}$, $\lambda \in \pwlat$, converge for all $(\zz_1,\zz_2) \in \CC^2 \setminus \set{0}$ and $0<\abs{q}<1$, so that the issue mentioned in \cref{rem:modularconvergence} does not arise.
The S-matrix \eqref{eq:sl3Smatrix} also simplifies considerably for $\slminmod[\uu,1]$, taking the form \cite{KacInf84}:
\begin{equation} \label{eq:sl3Smatrixv=1}
	\slSmat{\lambda,\lambda'}^{(\uu,1)} = \frac{-\ii}{\sqrt{3}\uu} \sum_{\wref{}\in \grp{S}_3} \det(\wref{}) \, \ee^{-2\pi\ii \bilin*{\wref{}(\lambda+\wvec)}{\lambda'+\wvec}/\uu}, \quad \lambda,\lambda' \in \pwlat.
\end{equation}
If we use this formula to extend the definition of $\slSmat{\lambda,\lambda'}^{(\uu,1)}$ to $\lambda,\lambda' \in \wlat$, then the extension satisfies
\begin{equation} \label{eq:sl3Smatrixv=1W}
	\slSmat{\wref{}\cdot\lambda,\lambda'}^{(\uu,1)} = \slSmat{\lambda,\wref{}\cdot\lambda'}^{(\uu,1)} = \det(\wref{}) \, \slSmat{\lambda,\lambda'}^{(\uu,1)}, \quad \wref{} \in \grp{S}_3,\ \lambda,\lambda' \in \wlat.
\end{equation}
In fact, this identity also holds for $\wref{}$ in the affine Weyl group $\affine{\grp{S}}_3$ of $\aslthree$ because this extends $\grp{S}_3$ by translations whose shifted action amounts to adding to $\lambda$ or $\lambda'$ an element of $\uu \crlat$.

Another consequence of \cref{thm:sl3rational} is that the fusion products of $\slminmod[\uu,1]$ are completely reducible and the fusion multiplicities are given by the Verlinde formula \cite{VerFus88,HuaVer04a}.
Let $\vac = (\uu-3,0,0) \in \pwlat$ denote the $\aslthree$-weight corresponding to the vacuum module of $\slminmod[\uu,1]$.
Then, the fusion rules of $\slminmod[\uu,1]$ take the form
\begin{subequations} \label{eq:sl3fusion}
	\begin{equation} \label{eq:sl3fusionrules}
	\slirr{\lambda} \fuse \slirr{\mu} \cong \bigoplus_{\nu \in \pwlat} \slfuscoeff[\uu,1]{\lambda}{\mu}{\nu} \, \slirr{\nu}, \quad \lambda,\mu \in \pwlat,
	\end{equation}
	where the fusion multiplicities (or fusion coefficients) are given by
	\begin{equation} \label{eq:sl3fusioncoeffs}
	\slfuscoeff[\uu,1]{\lambda}{\mu}{\nu} = \sum_{\Lambda \in \pwlat} \frac{\slSmat{\lambda,\Lambda}^{(\uu,1)} \slSmat{\mu,\Lambda}^{(\uu,1)} \brac[\big]{\slSmat{\nu,\Lambda}^{(\uu,1)}}^*}{\slSmat{\vac,\Lambda}^{(\uu,1)}}, \quad \lambda,\mu,\nu \in \pwlat,
	\end{equation}
\end{subequations}
and the star denotes complex conjugation.

Although it is perhaps not obvious from \eqref{eq:sl3fusioncoeffs}, a fundamental fact about $\slminmod[\uu,1]$ fusion coefficients is that they vanish unless the projection of $\lambda+\mu-\nu$ onto $\csub^*$ belongs to $\rlat$.
This is an easy consequence of the definition of fusion in terms of $3$-point correlation functions.
It may also be viewed as a corollary of the Kac--Walton formula for fusion coefficients \cite{WalFus90}.

\subsection{Finite-multiplicity weight modules for $\slthree$} \label{sec:sl3cohfam}

We pause to briefly recall a few useful facts about the classification of weight $\slthree$-modules with finite multiplicities (meaning that each weight space has finite dimension).
This was first achieved in \cite{BriSim95}, using seminal work of Fernando \cite{FerLie90}.
We shall however review Mathieu's approach using coherent families of dense modules \cite{MatCla00} as it generalises to all reductive Lie algebras.
\begin{definition} \label{def:dense}
	A dense module over a reductive Lie algebra is a weight module whose set of weights coincides with a translate of the root lattice.
\end{definition}

Every irreducible finite-multiplicity weight module is isomorphic to one obtained by inducing a dense module over the Levi factor of a parabolic subalgebra \cite[Thms.~4.18 and~5.2]{FerLie90}.
For $\slthree$, this means that there are three distinct classes of irreducible finite-multiplicity weight modules:
\begin{itemize}
	\item The irreducible \hwms, corresponding to the Levi factor $\csub$ (and some choice of Borel subalgebra).
	\item The irreducible finite-multiplicity dense modules, corresponding to the Levi factor $\slthree$.
	\item The irreducible weight modules corresponding to a Levi factor isomorphic to $\gltwo$.
	The set of weights of one of these modules is a translate of the ``half-lattice'' $\NN\alpha \oplus \ZZ\beta$, for some roots $\alpha,\beta$ with $\alpha \ne \pm\beta$.
\end{itemize}
\begin{definition} \label{def:semidense}
	A semidense $\slthree$-module is a weight module whose set of weights is a translate of $\NN\alpha \oplus \ZZ\beta$, for some roots $\alpha,\beta$ with $\alpha \ne \pm\beta$.
\end{definition}
\noindent As semidense modules are induced from dense $\gltwo$-modules, we shall concentrate on the theory of dense modules.

It turns out that irreducible dense modules with finite multiplicities only exist when the simple ideals of the Lie algebra are of types A or C \cite{FerLie90}.
Moreover, the (nonzero) multiplicities of such a module are constant.
One of Mathieu's key observations was that these dense modules come in families whose members are (mostly) only distinguished by which translate of the root lattice coincides with their set of weights.
Taking a direct sum of such modules over all (distinct) translates then results in a module dubbed a coherent family.
\begin{definition}[\protect{\cite{MatCla00}}]
	A coherent family for a simple Lie algebra $\alg{g}$, with Cartan subalgebra $\alg{g}_0$, is a weight module $\slcoh{}$ satisfying the following properties:
	\begin{itemize}
		\item Every $\mu \in \alg{g}_0^*$ is a weight of $\slcoh{}$ and every weight has the same (finite) multiplicity.
		\item Given any $U$ in the \uea\ of $\alg{g}$ that commutes with $\alg{g}_0$, the function taking $\mu \in \alg{g}_0^*$ to the trace of $U$ in the weight space of $\slcoh{}$ with weight $\mu$ is polynomial.
	\end{itemize}
	A coherent family is irreducible if it has an irreducible dense direct summand.
\end{definition}
\noindent We refer to the original article and to the introductions \cite[\S3.5]{MazLec10} and \cite[\S2]{KawRel19} for further information.

Because of the polynomial restriction in the definition, every irreducible coherent family includes some reducible members alongside its irreducible ones.
In fact, there are a finite number of members of any coherent family whose composition factors include an \infdim\ irreducible \hwm\ \cite{MatCla00}.
The multiplicities of these \hwms\ are uniformly bounded above by the constant multiplicity of the coherent family.

Mathieu's classification of irreducible dense modules invokes these \hwms.
\begin{theorem}[\protect{\cite[Props.~4.8, 5.7 and~6.2]{MatCla00}}] \label{thm:mathieu}
	Being composition factors of members of the same irreducible coherent family defines an equivalence relation on the set of isomorphism classes of \infdim\ irreducible \hwms\ with uniformly bounded multiplicities.
	Moreover, isomorphism classes of irreducible coherent families are in bijection with these equivalence classes.
\end{theorem}

We can recover an irreducible coherent family from any of its uniformly bounded \infdim\ irreducible \hwms\ using twisted localisation.
To explain, suppose that the irreducible \hw\ $\slthree$-module $\fslirr{\lambda}$ is \infdim\ with uniformly bounded multiplicities.
Then, $f^i$ acts injectively on $\fslirr{\lambda}$ for some $i=1,2,3$; assume that $i=1$ for definiteness.
We may therefore localise the \uea\ of $\slthree$ with respect to $f^1$.
This means that we extend the algebra by (formal) negative powers of $f^1$ (see \cite[\S3.5]{MazLec10} for more detail).

Next, we induce $\fslirr{\lambda}$ to a module over the localised algebra.
This means that we allow $(f^1)^{-1}$ to act freely, modulo the requirement that it acts as the inverse of $f^1$.
In particular, if $v = f^1 w \in \fslirr{\lambda}$, then $(f^1)^{-1} v = w$.
(Here, $f^1$ acting injectively ensures that $w$ is unique.)
However, if $v$ is not in the image of $f^1$, then $u = (f^1)^{-1} v$ is a new vector in the induced module (that is not in $\fslirr{\lambda}$).
Moreover, we obviously have $f^1 u = v$.
We conclude that $(f^1)^{-1}$ produces new weight vectors in the induced module unless $f^1$ acts bijectively between the two given weight spaces.
A consequence of these new vectors is that $f^1$ acts bijectively on the induced module.

Finally, we restrict the induced module back to an $\slthree$-module.
This restriction, called the localisation of $\fslirr{\lambda}$ with respect to $f^1$, is strictly bigger than $\fslirr{\lambda}$.
It is, in fact, a semidense $\slthree$-module whose multiplicities are uniformly bounded with the same maximal multiplicity as those of $\fslirr{\lambda}$, see \cref{fig:localisation}.
It is also reducible, with $\fslirr{\lambda}$ appearing as a submodule.
Moreover, the $\wref{1}$-twist of the quotient by $\fslirr{\lambda}$ has a \hwv\ of weight $\wref{1}(\lambda+\srt{1}) = \wref{1} \cdot \lambda$.

\begin{figure}
	\begin{tikzpicture}[scale=0.6]
		\draw[thick] (180:6) -- (0:0) -- (-60:3) -- ($(-60:3)+(-120:4)$);
		\node at (180:7) {$\cdots$};
		\foreach \r in {0,...,6} \node[wwt] at (180:\r) {$1$};
		\foreach \r in {1,2,3} \node[wwt] at (-60:\r) {$1$};
		\begin{scope}[shift={(-60:3)}]
			\foreach \r in {1,...,4} \node[wwt] at (-120:\r) {$1$};
			\node[rotate=-120] at (-120:5) {$\cdots$};
		\end{scope}
		\begin{scope}[shift={(-120:1)}]
			\node at (180:6) {$\cdots$};
			\foreach \r in {0,...,5} \node at (180:\r) {$2$};
			\foreach \r in {1,2} \node at (-60:\r) {$2$};
			\begin{scope}[shift={(-60:2)}]
				\foreach \r in {1,...,4} \node at (-120:\r) {$2$};
				\node[rotate=-120] at (-120:5) {$\cdots$};
			\end{scope}
		\end{scope}
		\begin{scope}[shift={(-120:2)}]
			\node at (180:6) {$\cdots$};
			\foreach \r in {0,...,5} \node at (180:\r) {$3$};
			\node at (-60:1) {$3$};
			\begin{scope}[shift={(-60:1)}]
				\foreach \r in {1,...,4} \node at (-120:\r) {$3$};
				\node[rotate=-120] at (-120:5) {$\cdots$};
			\end{scope}
		\end{scope}
		\begin{scope}[shift={(-120:3)}]
			\node at (180:5) {$\cdots$};
			\foreach \r in {0,...,4} \node at (180:\r) {$4$};
			\foreach \r in {1,...,4} \node at (-120:\r) {$4$};
			\node[rotate=-120] at (-120:5) {$\cdots$};
		\end{scope}
		\begin{scope}[shift={($(-120:4)+(180:1)$)}]
			\node at (180:4) {$\cdots$};
			\foreach \r in {0,...,3} \node at (180:\r) {$4$};
			\foreach \r in {1,2,3} \node at (-120:\r) {$4$};
			\node[rotate=-120] at (-120:4) {$\cdots$};
		\end{scope}
		\begin{scope}[shift={($(-120:5)+(180:2)$)}]
			\node at (180:2) {$\cdots$};
			\foreach \r in {0,1} \node at (180:\r) {$4$};
			\foreach \r in {1,2} \node at (-120:\r) {$4$};
			\node[rotate=-120] at (-120:3) {$\cdots$};
		\end{scope}
		\begin{scope}[shift={($(-120:6)+(180:3)$)}]
			\node at (180:1) {$\cdots$};
			\node at (0:0) {$4$};
			\node[rotate=-150] at (-135:1.25) {$\cdots$};
		\end{scope}
		\begin{scope}[shift={(4,-1)}]
			\draw[blue,->] (0:0) -- (0:1) node[below,black] {$\srt{1}$};
			\draw[blue,->] (0:0) -- (120:1) node[left,black] {$\srt{2}$};
			\draw[blue,->] (0:0) -- (60:1) node[right,black] {$\srt{3}$};
			\draw[thick,->] (-1.5,-2) -- node[above,red] {localisation} (1.5,-2);
		\end{scope}
		\begin{scope}[shift={(7,0)}]
			\draw[thick] (0:1) -- (0:5) -- ($(0:5)+(-60:3)$) -- ($(0:5)+(-60:3)+(-120:4)$);
			\draw[thick,dotted] (0:9) -- (0:6) -- ($(0:6)+(-120:7)$);
			\node at (0:0) {$\cdots$};
			\foreach \x in {1,...,9} \node[wwt] at (0:\x) {$1$};
			\node at (0:10) {$\cdots$};
			\begin{scope}[shift={(-60:1)}]
				\node at (0:0) {$\cdots$};
				\foreach \x in {1,...,8} \node[wwt] at (0:\x) {$2$};
				\node at (0:9) {$\cdots$};
			\end{scope}
			\begin{scope}[shift={(-60:2)}]
				\node at (0:-1) {$\cdots$};
				\foreach \x in {0,...,8} \node[wwt] at (0:\x) {$3$};
				\node at (0:9) {$\cdots$};
				\begin{scope}[shift={(-60:1)}]
					\node at (0:-1) {$\cdots$};
					\foreach \x in {0,...,7} \node[wwt] at (0:\x) {$4$};
					\node at (0:8) {$\cdots$};
				\end{scope}
			\end{scope}
			\begin{scope}[shift={(-60:4)}]
				\node at (0:-2) {$\cdots$};
				\foreach \x in {-1,...,7} \node[wwt] at (0:\x) {$4$};
				\node at (0:8) {$\cdots$};
				\begin{scope}[shift={(-60:1)}]
					\node at (0:-2) {$\cdots$};
					\foreach \x in {-1,...,6} \node[wwt] at (0:\x) {$4$};
					\node at (0:7) {$\cdots$};
				\end{scope}
			\end{scope}
			\begin{scope}[shift={(-60:6)}]
				\node at (0:-3) {$\cdots$};
				\foreach \x in {-2,...,6} \node[wwt] at (0:\x) {$4$};
				\node at (0:7) {$\cdots$};
				\begin{scope}[shift={(-60:1)}]
					\node at (0:-3) {$\cdots$};
					\foreach \x in {-2,...,5} \node[wwt] at (0:\x) {$4$};
					\node at (0:6) {$\cdots$};
				\end{scope}
			\end{scope}
			\begin{scope}[shift={(-60:8)}]
				\node[rotate=-150] at (0:-3) {$\cdots$};
				\foreach \x in {-2,-1,0} \node[rotate=-120] at (0:\x) {$\cdots$};
				\node[rotate=-90] at (0:1) {$\cdots$};
				\foreach \x in {2,3,4} \node[rotate=-60] at (0:\x) {$\cdots$};
				\node[rotate=-30] at (0:5) {$\cdots$};
			\end{scope}
		\end{scope}
	\end{tikzpicture}
\caption{%
	An illustration of how localisation transforms an irreducible \hwm, with uniformly bounded multiplicities, into a reducible semidense module, also with uniformly bounded multiplicities.
	At left, the position of each number specifies a weight of $\fslirr{\lambda}$ (we take $\lambda_1 \notin \ZZ$ and $\lambda_2=3$ for definiteness).
	The number itself gives the multiplicity of this weight.
	At right, the weights and multiplicities of the localisation of $\fslirr{\lambda}$ with respect to $f^1$ are pictured in the same fashion.
	The solid lines indicate the ``boundary'' of the (sub)module $\fslirr{\lambda}$.
	The dotted line indicates that of the quotient.
} \label{fig:localisation}
\end{figure}
This localisation construction works as above whenever we have a root vector acting on a module injectively.
If we try to localise with respect to a root vector that does not act injectively, then the submodule generated by the kernel of the root vector's action is set to zero.
On the other hand, if we localise with respect to a root vector that acts bijectively, then the result is isomorphic to the original module.

Note that when a root vector acts injectively but not bijectively, localisation produces a reducible module.
However, the localised algebra obtained by adjoining the formal inverses of a root vector admits automorphisms \cite{MatCla00} that translate the corresponding coroot by arbitrary complex multiples of the unit.
Twisting by these automorphisms before restricting thus allows one to also construct modules whose set of weights is obtained from that of the untwisted localisation by translation.
For $\slthree$, this twisted localisation therefore maps a single \hwm\ to a $1$-parameter family of (reducible and irreducible) semidense $\slthree$-modules.

Finally, this twisted localisation construction may be iterated if there is another root vector acting injectively (but not bijectively) on the result.
For example, the semidense module constructed from $\fslirr{\lambda}$ above may be localised with respect to $f^3$ to obtain a dense $\slthree$-module.
Including twists, one can therefore construct a $2$-parameter family, indexed by $\csub^*/\rlat$, of dense $\slthree$-modules from a single \hwm\ (with uniformly bounded multiplicities) \cite{MatCla00}.
Taking the direct sum of these dense modules then results in an irreducible coherent family.

\subsection{Relaxed \hw\ representation theory} \label{sec:sl3rhwm}

When $v>1$, the $\slthree$ minimal model $\slminmod[\uu,\vv]$ admits irreducible modules that are not \hw\ (with respect to any choice of Borel subalgebra).
A mild, but rather fundamental, generalisation of a \hwm\ has come to be known, following \cite{FeiEqu97}, as a \rhwm.
We recall the definition of \cite{RidRel15}.
\begin{definition} \label{def:rhwm}
	\leavevmode
	\begin{itemize}
		\item A \rhwv\ is a weight vector that is annihilated by every positively indexed mode.
		\item A \rhwm\ is a module that is generated by a single \rhwv.
	\end{itemize}
\end{definition}
\noindent A \rhwm\ is thus always lower bounded.
The converse is not true in general, though it holds if the module is irreducible and weight.
Similarly, a \hwm\ is always a \rhwm.
Again, the converse is false, but it holds for \voas\ that are $C_2$-cofinite and rational.

It will be convenient to distinguish three different types of \rhw\ $\usl$-module, based on the nature of the top space (see \cref{def:topspace}).
\begin{definition}
	A \rhw\ $\usl$-module is fully relaxed, semirelaxed or \hw, if its top space is a dense, semidense or \hw\ $\slthree$-module, respectively.
\end{definition}
\noindent This definition of a \hw\ $\usl$-module is of course equivalent to the usual one.
Recall \cite{ZhuMod96} that an irreducible lower-bounded module is determined, up to isomorphism, by its top space.
Here, the top space is regarded as a module over the Zhu algebra of the \voa.
The same is true for a certain generalisation known as an almost-irreducible lower-bounded module.
\begin{definition} \label{def:almostirreducible}
	A lower-bounded module is said to be almost irreducible if it is generated by its top space and each of its nonzero submodules has nonzero intersection with its top space.
\end{definition}
\noindent It follows \cite[Thm.~2.30]{DeSFin05} that there is a bijective correspondence between almost-irreducible lower-bounded modules over a \voa\ and modules over its Zhu algebra.

We are interested in the irreducible weight $\slminmod[\uu,\vv]$-modules with $\vv=2$.
The \hw\ ones are the $\slirr{\mu}$ with $\mu \in \admwts$ (\cref{thm:sl3hwclass}).
Otherwise, the top space is either dense or semidense, hence may be constructed from the top space of one of the $\slirr{\mu}$ by twisted localisation.
\begin{theorem}[\cite{KawRel19,KawRel20}] \label{thm:zhuloc}
	If $\Mod{M}$ is a module for the Zhu algebra of an affine \voa, then so is its twisted localisation with respect to any root vector.
\end{theorem}
\noindent We may thus apply twisted localisation to the top space of the $\slminmod[\uu,\vv]$-module $\slirr{\mu}$ to obtain a (possibly reducible) semidense or dense $\slthree$-module, knowing that the result is the top space of some almost-irreducible lower-bounded $\slminmod[\uu,\vv]$-module.

Denote the top space of $\slirr{\mu}$ by $\fslirr{\bar{\mu}}$ (the latter is then the irreducible \hw\ $\slthree$-module of highest weight $\bar{\mu} = \mu_1 \fwt{1} + \mu_2 \fwt{2}$).
It turns out that all the \infdim\ $\fslirr{\bar{\mu}}$ with $\mu \in \admwts$ have uniformly bounded multiplicities \cite[\S4.3]{KawAdm21}.
These constitute the following subset of $\admwts$:
\begin{equation} \label{eq:boundedwts}
	\set[\big]{\mu^I - \tfrac{\uu}{2} \fwt{1}, \mu^I - \tfrac{\uu}{2} \fwt{2}, \wref{1} \cdot (\mu^I - \tfrac{\uu}{2} \fwt{1}) \st \mu^I \in \pwlat}.
\end{equation}
Mathieu's equivalence relation for irreducible coherent families (\cref{thm:mathieu}) partitions this set into three-element equivalence classes of the form
\begin{equation} \label{eq:boundedequivclass}
	\set[\big]{\mu^I - \tfrac{\uu}{2} \fwt{1}, \nabla(\mu^I) - \tfrac{\uu}{2} \fwt{2}, \wref{1} \cdot (\mu^I - \tfrac{\uu}{2} \fwt{1})}, \quad \mu^I \in \pwlat.
\end{equation}
Here, we introduce a permutation $\nabla$ of $\pwlat$ (and more generally of level-$\kk$ $\aslthree$-weights) defined by
\begin{equation} \label{eq:defZ3}
	\nabla(\lambda_0,\lambda_1,\lambda_2) = (\lambda_1,\lambda_2,\lambda_0).
\end{equation}
This is of course a Dynkin symmetry (outer automorphism) of $\aslthree$.
Note that $\nabla(\mu^I) - \frac{\uu}{2} \fwt{2} = \wref{2}\wref{1} \cdot (\mu^I - \tfrac{\uu}{2} \fwt{1})$, so the equivalence classes \eqref{eq:boundedequivclass} are (partial) shifted Weyl orbits.

The twisted localisations of the $\fslirr{\bar{\mu}}$ are the top spaces of some almost-irreducible semirelaxed or fully relaxed $\slminmod$-modules.
When the latter are reducible, we wish to know their submodule structure.
For this, it will be useful to know the maximal multiplicity of the \hwms\ in each equivalence class \eqref{eq:boundedequivclass}.
\begin{lemma} \label{lem:ihwbounds}
	Given $\uu \in \set{3,5,7,\dots}$ and $\mu^I \in \pwlat$, the maximal multiplicities of the top spaces of $\slirr{\mu^I-\uu\fwt{1}/2}$, $\slirr{\nabla(\mu^I)-\uu\fwt{2}/2}$ and $\slirr{\wref{1}\cdot(\mu^I-\uu\fwt{1}/2)}$ are all $\mu^I_2+1$.
\end{lemma}
\begin{proof}
	It suffices to consider $\mu = \mu^I - \frac{\uu}{2} \fwt{1}$ because weights in the same equivalence class correspond to the same coherent family.
	The orbit of its projection $\bar{\mu}$ onto $\csub^*$ under the shifted action of the Weyl group of $\slthree$ consists of six weights, but only $\wref{2} \cdot \bar{\mu}$ differs from $\bar{\mu}$ by a (nonzero) element of the root lattice $\rlat$.
	As $\wref{2} \cdot \bar{\mu} - \bar{\mu} = -(\mu^I_2+1) \srt{2}$, it follows \cite[\S4.11]{HumRep08} that $\fslirr{\bar{\mu}} \cong \fslver{\bar{\mu}} / \fslver{\wref{2} \cdot \bar{\mu}}$ (where $\fslver{\bar{\mu}}$ denotes the Verma $\slthree$-module of highest weight $\bar{\mu}$).
	Defining $\slthree$-module characters as in \eqref{eq:slchar}, but without the formal parameter $\qq$, we have
	\begin{equation}
		\begin{split}
			\fch{\fslirr{\bar{\mu}}}{\zz_1,\zz_2}
			&= \fch{\fslver{\bar{\mu}}}{\zz_1,\zz_2} - \fch{\fslver{\wref{2} \cdot \bar{\mu}}}{\zz_1,\zz_2}
			= \frac{\zz_1^{\mu^I_1} \zz_2^{\mu^I_2} \brac[\big]{1-\zz_1^{\mu^I_2+1}\zz_2^{-2(\mu^I_2+1)}}}{(1-\zz_1^{-2}\zz_2^{\vphantom{1}}) (1-\zz_1^{\vphantom{1}}\zz_2^{-2}) (1-\zz_1^{-1}\zz_2^{-1})} \\
			&= \frac{\zz_1^{\mu^I_1} \zz_2^{\mu^I_2}}{(1-\zz_1^{-2}\zz_2^{\vphantom{1}}) (1-\zz_1^{-1}\zz_2^{-1})} \brac[\big]{1 + \zz_1^{\vphantom{1}}\zz_2^{-2} + \dots + \zz_1^{\mu^I_2}\zz_2^{-2\mu^I_2}}.
		\end{split}
	\end{equation}
	The maximal multiplicity is thus $\mu^I_2+1$, as required.
\end{proof}

\begin{proposition} \label{prop:hwlocalisation}
	Given $\uu \in \set{3,5,7,\dots}$, let $\mu = \mu^I-\frac{\uu}{2}\fwt{1}$ for some $\mu^I \in \pwlat$.
	Then, the localisation of the irreducible \hw\ $\slthree$-module $\fslirr{\bar{\mu}}$, $\fslirr{\wref{1} \cdot \bar{\mu}}$ or $\fslirr{\wref{2}\wref{1} \cdot \bar{\mu}}$, corresponding to the projections of the $\aslthree$-weights in \eqref{eq:boundedequivclass}, with respect to $f^i$, $i=1,2,3$, is either $0$ or a reducible semidense module whose quotient is (up to isomorphism) as in the following table (dashes indicate that the localisation is $0$).
	\begin{center}
		\begin{tabular}{C|CCC}
			& \fslirr{\bar{\mu}} & \fslirr{\wref{1} \cdot \bar{\mu}} & \fslirr{\wref{2}\wref{1} \cdot \bar{\mu}} \\
			\hline
			f^1 & \wref{1}(\fslirr{\wref{1} \cdot \bar{\mu}}) & \wref{1}(\fslirr{\bar{\mu}}) & \blank \\
			f^2 & \blank & \wref{2}(\fslirr{\wref{2}\wref{1} \cdot \bar{\mu}}) & \wref{2}(\fslirr{\wref{1} \cdot \bar{\mu}}) \\
			f^3 & \wref{2}\wref{1}(\fslirr{\wref{1} \cdot \bar{\mu}}) & \blank & \wref{1}\wref{2}(\fslirr{\wref{1} \cdot \bar{\mu}})
		\end{tabular}
	\end{center}
\end{proposition}
\begin{proof}
	The computations are essentially the same in each case, so we restrict ourselves to the localisation of $\fslirr{\bar{\mu}}$ with respect to $f^1$.
	In this case, we showed in \cref{sec:sl3cohfam} that the quotient contains a $\wref{1}$-twisted \hwv\ of weight $\bar{\mu}+\srt{1}$.
	By \cref{thm:sl3hwclass}, this vector generates a submodule of the quotient isomorphic to the irreducible module $\wref{1}(\fslirr{\wref{1} \cdot \bar{\mu}})$.
	We claim that this submodule is in fact the entire quotient.

	First, note that $\fslirr{\bar{\mu}}$ and $\wref{1}(\fslirr{\wref{1} \cdot \bar{\mu}})$ share the same maximal multiplicity.
	This also coincides with the maximal multiplicity of the semidense localisation because $f^1$ acts bijectively on the localisation but already acts bijectively between weight spaces of $\fslirr{\bar{\mu}}$ when their multiplicities are maximal.
	Fix a weight $\lambda$ of the localisation with maximal multiplicity.
	Then, for all sufficiently large $m$, $\lambda-m\srt{1}$ will be a weight of $\fslirr{\bar{\mu}}$ with maximal multiplicity and, for all sufficiently large $n$, $\lambda+n\srt{1}$ will be a weight of $\wref{1}(\fslirr{\wref{1} \cdot \bar{\mu}})$ with maximal multiplicity.
	If the quotient has any composition factor aside from $\wref{1}(\fslirr{\wref{1} \cdot \bar{\mu}})$, then the multiplicity of $\lambda+\ell\srt{1}$ in this factor will be $0$ for all but finitely many $\ell\in\ZZ$.
	But, $\mu_1 = \mu^I_1 - \frac{\uu}{2} \notin \NN$, so $\lambda_1 \notin \NN$ and thus this is impossible.
\end{proof}

\begin{theorem} \label{thm:hwlocalisation}
	For $\uu \in \set{3,5,7,\dots}$ and $\mu = \mu^I-\frac{\uu}{2}\fwt{1}$, $\mu^I \in \pwlat$, the almost-irreducible semirelaxed $\slminmod$-modules whose top spaces are the reducible semidense $\slthree$-modules characterised in the table of \cref{prop:hwlocalisation} have the same submodule-quotient characterisation, but with $\bar{\mu}$ replaced by $\mu$.
\end{theorem}
\begin{proof}
	We consider only the case in which the semidense top space has a submodule isomorphic to $\fslirr{\bar{\mu}}$ and quotient isomorphic to $\wref{1}(\fslirr{\wref{1}\cdot\bar{\mu}})$; the other cases are very similar.
	First, the top space has a \hwv\ of weight $\bar{\mu}$, hence the semirelaxed module has a \hwv\ of weight $\mu$.
	It generates a \hw\ $\slminmod$-module which must be isomorphic to $\slirr{\mu}$, by \cref{thm:sl3hwclass}.
	Similarly, the quotient of the semirelaxed module by $\slirr{\mu}$ has a twisted \hwv\ generating a copy of $\wref{1}(\slirr{\wref{1}\cdot\mu})$.
	But, the preimage of this twisted \hwv\ generates the semidense top space, hence it generates the semirelaxed module because the latter is almost-irreducible.
	It therefore also generates the quotient, hence the quotient is irreducible.
\end{proof}

One can continue this analysis to determine the submodule structures of the reducible fully relaxed $\slminmod$-modules, see \cite[\S4.5]{KawAdm21}.
Here however, we restrict ourselves to the following simple \lcnamecref{cor:localisationchar}.
\begin{corollary} \label{cor:localisationchar}
	For $\uu \in \set{3,5,7,\dots}$ and $\mu = \mu^I-\frac{\uu}{2}\fwt{1}$, $\mu^I \in \pwlat$, the almost-irreducible semirelaxed $\slminmod$-modules characterised in \cref{thm:hwlocalisation} have vanishing characters (not as formal power series, but in the meromorphically extended sense of \cref{rem:modularconvergence}).
\end{corollary}
\begin{proof}
	First, let $\Delta(\mu)$ denote the set of real roots $\alpha$ of $\aslthree$ satisfying $\bilin{\mu+\wvec}{\co{\alpha}} \in \ZZ$.
	Then, it easily follows that
	\begin{equation}
		\Delta(w\cdot\mu) = w\brac[\big]{\Delta(\mu)},
	\end{equation}
	for all $w$ in the affine Weyl group $\affine{\grp{S}}_3$.
	Second, recall the definition of the integral affine Weyl group in \cref{thm:sl3admchar} and the identity $\wref{w(\alpha)} = w\wref{\alpha}w^{-1}$, for all roots $\alpha$.
	The previous identity now implies that
	\begin{equation}
		\affine{\grp{W}}(w\cdot\mu) = w \affine{\grp{W}}(\mu) w^{-1}.
	\end{equation}
	Third, the natural action of $\affine{\grp{S}}_3$ on formal exponentials is given by $w(\ee^{\lambda}) = \ee^{w(\lambda)}$.
	Combining this with the definition \eqref{eq:slformalchar} then gives
	\begin{equation}
			\ch{w(\slirr{\mu})} = w\brac[\big]{\ch{\slirr{\mu}}}.
	\end{equation}

	Mixing these ingredients with the character formula of \cref{thm:sl3admchar} now results in
	\begin{align} \label{eq:vanishingchars}
		\ch{w(\slirr{w^{-1}\cdot\mu})}
		&= \frac{w\brac[\big]{\sum_{w'\in\affine{\grp{W}}(w^{-1}\cdot\mu)} \det w' \, \ee^{w'w^{-1}\cdot\mu}}}{w\brac*{\sum_{w'\in\affine{\grp{W}}(0)} \det w' \, \ee^{w'\cdot0}}}
		= \frac{\sum_{w'\in\affine{\grp{W}}(w^{-1}\cdot\mu)} \det w' \, \ee^{ww'w^{-1}\cdot\mu}}{\sum_{w'\in\affine{\grp{W}}(0)} \det w' \, \ee^{ww'\cdot0}} \\
		\intertext{(since $w(w'w^{-1}\cdot\mu) = ww'w^{-1}\cdot\mu + \wvec - w(\wvec)$ and $w(w'\cdot0) = ww'\cdot0 + \wvec - w(\wvec)$)}
		\notag &= \frac{\sum_{w''\in\affine{\grp{W}}(\mu)} \det (w^{-1}w''w) \, \ee^{w''\cdot\mu}}{\sum_{w''\in\affine{\grp{W}}(0)} \det (w^{-1}w'') \, \ee^{w''\cdot0}}
		= \det w \, \ch{\slirr{\mu}},
	\end{align}
	where we have noted that $\grp{W}(0) = \affine{\grp{S}}_3$ is closed under multiplication by $w^{-1}$.

	The vanishing character claims now follow.
	We explain two cases.
	First, the $f^1_0$-localisation of $\slirr{\mu}$ yields a semirelaxed \hwm\ with composition factors $\slirr{\mu}$ and $\wref{1}(\slirr{\wref{1}\cdot\mu})$.
	Taking $w=\wref{1}$ in \eqref{eq:vanishingchars}, we learn that the (meromorphically extended) characters satisfy $\ch{\wref{1}(\slirr{\wref{1}\cdot\mu})} = -\ch{\slirr{\mu}}$.
	The character of the semirelaxed module thus vanishes.
	Second, apply $\wref{2}$ to the previous (meromorphic) character identity to get $\ch{\wref{2}\wref{1}(\slirr{\wref{1}\cdot\mu})} = -\ch{\wref{2}(\slirr{\mu})}$.
	Because $\wref{2}\wref{1} \cdot (\mu^I - \tfrac{\uu}{2} \fwt{1}) = \nabla(\mu^I) - \frac{\uu}{2} \fwt{2}$, we have $\wref{2}(\slirr{\mu}) \cong \slirr{\mu}$.
	The character of the $f^3_0$-localisation of $\slirr{\mu}$ thus also vanishes.
\end{proof}

\section{\bp\ minimal models} \label{sec:bp}

We next introduce the \bp\ \voas, first abstractly and then as the minimal \qhrs\ of $\usl$.
The latter construction is more convenient for studying the irreducible \hw\ $\bpminmod$-modules, $\uu \in \set{3,5,7,\ldots}$.
These were classified by Arakawa in \cite{AraRat10}, but we will follow the explicit description deduced in \cite{FehCla20}.
Our goal is to determine their characters, modular properties and fusion rules.
This data may be extracted, in principle, from the general results reported for exceptional W-algebras in \cite{AraRat19}.
However, we find it instructive to derive everything directly.

\subsection{\bp\ \voas} \label{sec:ubp}

The \bp\ algebras were discovered in \cite{PolGau90,BerCon91} using a variant of \qhr\ that was subsequently identified \cite{KacQua03} as that corresponding to the minimal (and subregular) nilpotent orbit of $\slthree$.
For the inverse-reduction construction that follows in \cref{sec:iqhr}, we shall find it convenient to make a slightly nonstandard choice for the conformal structure.
\begin{definition} \label{def:bp}
	The universal level-$\kk$ \bp\ algebra $\ubp$, $\kk\ne-3$, is the \voa\ strongly generated by four elements $J$, $L$, $G^+$ and $G^-$, subject only to the following \opes:
	\begin{equation} \label{eq:bpopes}
		\begin{gathered}
			\begin{aligned}
				J(z)J(w) &\sim \frac{(2\kk+3) \wun}{3(z-w)^2}, &
				L(z)G^+(w) &\sim \frac{2G^+(w)}{(z-w)^2} + \frac{\pd G^+(w)}{z-w}, \\
				J(z)G^{\pm}(w) &\sim \pm \frac{G^{\pm}(w)}{z-w}, &
				L(z)G^-(w) &\sim \frac{G^-(w)}{(z-w)^2} + \frac{\pd G^-(w)}{z-w},
			\end{aligned}
			\\
			\begin{aligned}
				L(z)J(w) &\sim \frac{(2\kk+3) \wun}{3(z-w)^3} + \frac{J(w)}{(z-w)^2} + \frac{\pd J(w)}{z-w}, \\
				L(z)L(w) &\sim \frac{\ccbp \wun}{2(z-w)^4} + \frac{2L(w)}{(z-w)^2} + \frac{\pd L(w)}{z-w},
			\end{aligned}
			\qquad
			G^{\pm}(z)G^{\pm}(w) \sim 0, \\
			G^+(z)G^-(w) \sim \frac{(\kk+1)(2\kk+3) \wun}{(z-w)^3} + \frac{3(\kk+1)J(w)}{(z-w)^2} + \frac{3\no{J(w)J(w)} + \kk \pd J(w) - (\kk+3) L(w)}{z-w}.
		\end{gathered}
	\end{equation}
	The central charge is given by
	\begin{equation} \label{eq:bpcc}
		\ccbp = -\frac{4(\kk+1)(2\kk+3)}{\kk+3}.
	\end{equation}
\end{definition}

The conformal weights of the generators $J$, $L$, $G^+$ and $G^-$ are chosen to be $1$, $2$, $2$ and $1$, respectively.
The mode expansions of the generating fields are thus as follows:
\begin{equation}
	J(z) = \sum_{n \in \ZZ} J_n z^{-n-1}, \quad
	L(z) = \sum_{n \in \ZZ} L_n z^{-n-2}, \quad
	G^+(z) = \sum_{n \in \ZZ} G^+_n z^{-n-2} \quad \text{and} \quad
	G^-(z) = \sum_{n \in \ZZ} G^-_n z^{-n-1}.
\end{equation}
We record the commutation relations of the modes $G^+_m$ and $G^-_n$ for later convenience:
\begin{equation} \label{eq:commGG}
	\comm[\big]{G^+_m}{G^-_n} = 3 \no{JJ}_{m+n} - (\kk+3) L_{m+n} + \brac[\big]{(2\kk+3)(m+1) - nk} J_{m+n}
		+ \tfrac{1}{2} (\kk+1)(2\kk+3) m(m+1) \delta_{m+n,0} \wun.
\end{equation}

\begin{theorem}[\protect{\cite[Thm.~9.1.2]{GorSim06}}]
	The universal \bp\ \voa\ $\ubp$, $\kk\ne-3$, is simple unless $\kk$ satisfies \eqref{eq:admlevel}.
\end{theorem}

\begin{definition}
	When \eqref{eq:admlevel} holds, we shall denote the unique simple quotient of $\ubp$ by $\bpminmod[\uu,\vv]$ and refer to the resulting \voa\ as a \bp\ minimal model.
\end{definition}
\noindent We note that the trivial \bp\ minimal model corresponds to $\uu=3$ and $\vv=2$: $\bpminmod[3,2] = \CC\wun$.

The mode algebra of $\ubp$ admits a conjugation automorphism $\bpconj$ and a family of spectral flow automorphisms $\bpsf{\ell}$, $\ell \in \ZZ$, given by
\begin{subequations} \label{eq:bpauts}
	\begin{equation} \label{eq:bpconj}
			\bpconj(G^-_n) = -G^+_n, \quad
			\bpconj(J_n) = -J_n - \tfrac{1}{3}(2\kk+3) \delta_{n,0} \wun, \quad
			\bpconj(L_n) = L_n - nJ_n, \quad
			\bpconj(G^+_n) = G^-_n
	\end{equation}
	and
	\begin{equation} \label{eq:bpsf}
		\begin{gathered}
			\bpsf{\ell}(G^-_n) = G^-_{n+\ell}, \quad
			\bpsf{\ell}(J_n) = J_n - \tfrac{1}{3}(2\kk+3) \ell \delta_{n,0} \wun, \quad
			\bpsf{\ell}(G^+_n) = G^+_{n-\ell}, \\
			\bpsf{\ell}(L_n) = L_n - \ell J_n + \tfrac{1}{6}(2\kk+3) \ell(\ell-1) \delta_{n,0} \wun,
		\end{gathered}
	\end{equation}
\end{subequations}
respectively.
These automorphisms satisfy the dihedral relation $\bpconj \bpsf{\ell} \bpconj^{-1} = \bpsf{-\ell}$, $\ell \in \ZZ$.

Lifting these automorphisms to invertible functors on the category of weight $\ubp$-modules, as in \cref{sec:sl3auts}, it follows that a weight vector of weight $(j,\Delta)$ is mapped, under $\bpconj$ or $\bpsf{\ell}$, to a weight vector of weight
\begin{equation} \label{eq:bpconjsfwts}
	\brac[\big]{-j-\tfrac{1}{3}(2\kk+3), \Delta} \quad \text{or} \quad \brac[\big]{j+\tfrac{1}{3}(2\kk+3) \ell, \Delta + \ell j + \tfrac{1}{6}(2\kk+3) \ell(\ell+1)},
\end{equation}
respectively.
Here, a weight vector for $\ubp$ is a simultaneous eigenvector of $J_0$ and $L_0$.
Its weight is the pair $(j,\Delta)$ of eigenvalues of $J_0$ and $L_0$ (in that order).
As usual, conjugation preserves the property of being lower bounded while spectral flow generally does not.

Our primary interest here is the representation theory of the \bp\ minimal models $\bpminmod$.
For $\vv=2$, we have the following result.
\begin{theorem}[\protect{\cite[Thm.~5.10.2]{AraAss10} and \cite[Main Thm.]{AraRat10}}] \label{thm:bprational}
	The \bp\ minimal model \voas\ $\bpminmod$, $\uu \in \set{3,5,7,\dots}$, are $C_2$-cofinite and rational.
\end{theorem}
\noindent Every irreducible $\bpminmod$-module is thus \hw\ with a \fdim\ top space (\cref{thm:ordinary}).

We recall that a \hwv\ for $\ubp$ is a weight vector that is annihilated by the $J_n$, $L_n$, $G^+_n$ and $G^-_n$, with $n>0$, and $G^+_0$.
If a $\ubp$-module is generated by a \hwv, then it is called a \hwm.
We shall denote the irreducible \hw\ $\ubp$-module of highest weight $(j,\Delta)$ by $\bpirr{j,\Delta}$.

\subsection{Fermionic ghosts} \label{sec:fghosts}

To discuss the realisation of $\ubp$ as the minimal \qhr\ of $\usl$, we quickly review the fermionic ghost system $\fgvoa$.
This \svoa\ is strongly and freely generated by two fermionic fields $\psi(z)$ and $\psi^*(z)$ satisfying the \opes
\begin{equation}
	\psi(z)\psi^*(w) \sim \frac{\wun}{z-w}, \quad \psi(z)\psi(w) \sim 0 \sim \psi^*(z)\psi^*(w).
\end{equation}
There is a one-parameter family of \emts.
General \qhrs\ require more than one choice, but we shall only need one for what follows.
We thereby take the associated \emt\ to be
\begin{equation} \label{eq:emtfg}
	T^{\fgsymb}(z) = \no{\partial\psi\psi^*}(z),
\end{equation}
so that the central charge is $\ccfg=-2$ and the conformal weights of $\psi$ and $\psi^*$ are $0$ and $1$, respectively.

There is likewise a Heisenberg subalgebra of $\fgvoa$ generated by $\no{\psi^*\psi}$ and the eigenvalues of the zero mode of this field define a horizontal grading on $\fgvoa$.
The relevant \opes\ are as follows:
\begin{equation}
	\begin{gathered}
    \no{\psi^*\psi}(z)\psi(w) \sim -\frac{\psi(w)}{z-w}, \quad
    \no{\psi^*\psi}(z)\psi^*(w) \sim \frac{\psi^*(w)}{z-w}, \quad
    \no{\psi^*\psi}(z)\no{\psi^*\psi}(w) \sim \frac{\wun}{(z-w)^2}, \\
    T^{\fgsymb}(z)\no{\psi^*\psi}(w) \sim \frac{\wun}{(z-w)^3} + \frac{\no{\psi^*\psi}(w)}{(z-w)^2} + \frac{\pd \no{\psi^*\psi}(w)}{z-w}.
	\end{gathered}
\end{equation}

The representation theory of $\fgvoa$ is very simple.
If we ignore the global parity of a representation, then there is only one (untwisted) irreducible module: $\fgvoa$ itself.
Moreover, every finitely generated $\fgvoa$-module is completely reducible.
The character of $\fgvoa$ is likewise easily determined:
\begin{subequations} \label{eq:fg(s)char}
	\begin{equation} \label{eq:fgchar}
		\fch{\fgvoa}{\yy;\qq} = \traceover{\fgvoa} \yy^{\no{\psi^*\psi}_0} \qq^{T^{\fgsymb}_0-\ccfg/24}
		= \qq^{1/12} \prod_{i=1}^{\infty} (1+\yy\qq^i) (1+\yy^{-1}\qq^{i-1})
		= \yy^{-1/2} \frac{\fjth{2}{\yy;\qq}}{\eta(\qq)}.
	\end{equation}
	It will also prove useful to consider the supercharacter
	\begin{equation} \label{eq:fgschar}
		\fsch{\fgvoa}{\yy;\qq} = \traceover{\fgvoa} (-\yy)^{\no{\psi^*\psi}_0} \qq^{T^{\fgsymb}_0-\ccfg/24}
		= \qq^{1/12} \prod_{i=1}^{\infty} (1-\yy\qq^i) (1-\yy^{-1}\qq^{i-1})
		= \yy^{-1/2} \frac{\ii\fjth{1}{\yy;\qq}}{\eta(\qq)}.
	\end{equation}
\end{subequations}

\subsection{A minimal \qhr} \label{sec:bpqhr}

\expandafter\MakeUppercase \qhr\ refers to a collection of functors from a suitable module category for an affine \voa\ to a module category for one of its associated W-algebras.
Such a functor is specified \cite{KacQua03} by the simple Lie algebra $\alg{g}$ (for us, $\alg{g} = \slthree$) and a good pair $(x,f)$ of elements of $\alg{g}$.
In this context, being good means that
\begin{itemize}
	\item $\alg{g} = \bigoplus_{j\in\frac{1}{2}\ZZ} \alg{g}_j$, where $\alg{g}_j$ is the eigenspace of $\ad(x)$ with eigenvalue $j\in\frac{1}{2}\ZZ$.
	\item $f \in \alg{g}_{-1}$ and $\ad(f) \colon \alg{g}_j \to \alg{g}_{j-1}$ is injective for $j\ge\frac{1}{2}$ and surjective for $j\le\frac{1}{2}$.
\end{itemize}
To obtain the conformal weights chosen for the generating fields in \cref{sec:ubp}, we shall take $f=f^3$ and $x = \fcwt{1} = \frac{2}{3}h^1+\frac{1}{3}h^2$.
It is easy to check that $(x,f)$ is then a good pair.
As $f$ belongs to the minimal nilpotent orbit of $\slthree$, we shall refer to the corresponding reduction functor as minimal \qhr.

To define this functor, note that the gradation of $\alg{g}=\slthree$ induced by $x$ is given by $\slthree = \alg{g}_{-1} \oplus \alg{g}_0 \oplus \alg{g}_1$, where
\begin{equation}
	\alg{g}_{-1} = \spn \set{f^1,f^3}, \quad \alg{g}_0 = \spn \set{f^2,h^1,h^2,e^2} \quad \text{and} \quad \alg{g}_1 = \spn \set{e^1,e^3}.
\end{equation}
As per \cite{KacQua03}, we now tensor $\usl$ with two fermionic ghost systems $\fgvoa^i$, $i=1,3$, corresponding to $\alg{g}_{\pm1}$.
The action of the zero mode of $\no{{\psi^*}^i\psi^i}(z)$ defines a horizontal grading on each $\fgvoa^i$ and we grade the tensor product
\begin{equation}
	\qhrcomplex = \usl \otimes \fgvoa^1 \otimes \fgvoa^3
\end{equation}
by the action of their sum.

The zero mode of the element
\begin{equation} \label{eq:minqhrdiff}
	Q = e^1 {\psi^*}^1 + (e^3+\wun) {\psi^*}^3.
\end{equation}
now defines a degree-$1$ differential on the complex consisting of the graded subspaces of $\qhrcomplex$.
The degree-$0$ cohomology is, by definition, the minimal \qhr\ $\qhrmin\brac[\big]{\usl}$.
This is of course the universal \bp\ \voa\ $\ubp$ \cite{PolGau90,BerCon91}.

We remark that this construction suggests a conformal structure on $\qhrcomplex$.
If we modify the \emt\ \eqref{eq:emtsl3} of $\usl$ to $T + \pd x$, then the conformal weight of $e^3$ becomes $0$, as required for $Q$ to be homogeneous.
That of $e^1$ is also now $0$, so for $Q$ to have conformal weight $1$, both ${\psi^*}^1$ and ${\psi^*}^3$ must have conformal weight $1$.
This fixes the \emt\ of $\fgvoa^i$, $i=1,3$, to be as in \eqref{eq:emtfg}.
The \bp\ \emt\ is thus identified with the cohomology class
\begin{subequations} \label{eq:bpgens}
	\begin{equation}
		L = \sqbrac[\big]{T+\pd x+T^{\fgvoa^1}+T^{\fgvoa^3}},
	\end{equation}
	where we have omitted the tensor product symbols for brevity.
	The central charge is easily verified to be
	\begin{equation}
		\ccsl - 8\kk - 4 = \ccbp.
	\end{equation}
	It is also straightforward to find representatives for the remaining strong generators of $\ubp$:
	\begin{equation}
		\begin{gathered}
			J = \sqbrac[\big]{\tfrac{1}{3}(h^2-h^1)+\no{{\psi^*}^1\psi^1}}, \qquad
			G^- = \sqbrac[\big]{f^2-\psi^1{\psi^*}^3}, \\
			G^+ = -\sqbrac[\big]{f^1 + (\kk+2) \pd e^2 + \no{h^1 e^2} - e^2 (2\no{{\psi^*}^1\psi^1}+\no{{\psi^*}^3{\psi}^3}) + h^1{\psi^*}^1\psi^3 + \kk \pd {\psi^*}^1\psi^3 + (\kk+3) {\psi^*}^1\pd \psi^3}.
		\end{gathered}
	\end{equation}
\end{subequations}

We are now finally ready to define the minimal \qhr\ functor $\qhrmin$.
\begin{definition} \label{def:minqhr}
	Given a $\usl$-module $\slver{}$, let $\qhrmodule{}{\slver{}} = \slver{} \otimes \fgvoa^1 \otimes \fgvoa^3$.
	We denote by $\qhrmodule{n}{\slver{}}$ the subspace of $\qhrmodule{}{\slver{}}$ whose total ghost number (the eigenvalue of $\no{{\psi^*}^1\psi^1}_0 + \no{{\psi^*}^3\psi^3}_0$) is $n$.
	The minimal \qhr\ of $\slver{}$ is then defined to be the zeroth cohomology group of the differential complex $\brac[\big]{\qhrmodule{\bullet}{\slver{}},Q_0}$:
	\begin{equation}\label{eq:defminqhr}
		\qhrmin(\slver{}) = H^0 \brac[\big]{\qhrmodule{}{\slver{}},Q_0}.
	\end{equation}
\end{definition}
\noindent This obviously agrees with the definition given above for $\slver{} = \usl$ and it follows readily that $\qhrmin(\slver{})$ is, in general, a $\ubp$-module.

The following properties of $\qhrmin$ will be useful in what follows.
Recall that $\slirr{\mu}$ denotes the irreducible \hw\ $\usl$-module of highest weight $\mu$.
Let $\slver{\mu}$ denote its Verma cover and let $\ocat^{\kk}$ and $\ocat_{\kk}$ denote the BGG categories of $\usl$- and $\ssl$-modules, respectively.
\begin{theorem} \label{thm:minqhr}
	Let $\kk\ne-3$.
	Then:
	\begin{itemize}
		\item \cite[Thm.~6.7.1 and Cor.~6.7.3]{AraRep04} For all $\slver{} \in \ocat^{\kk}$, the $Q_0$-cohomology of $\qhrmodule{}{\slver{}}$ is concentrated in degree $0$.
		$\qhrmin$ is thus an exact functor from $\ocat^{\kk}$ to the category of $\ubp$-modules.
		\item \cite[Thm.~6.3]{KacQua03b} $\qhrmin$ maps $\slver{\mu}$ to the Verma $\ubp$-module of highest weight $(j_{\mu},\bpconfwt{\mu})$, where
		\begin{equation} \label{eq:bphw}
			j_{\mu} = -\frac{\mu_1-\mu_2}{3} \quad \text{and} \quad
			\bpconfwt{\mu} = \frac{(\mu_1-\mu_2)\brac[\big]{\mu_1-\mu_2-2(\kk+3)} + 3(\mu_1+\mu_2)\brac[\big]{\mu_1+\mu_2-2(\kk+1)}}{12(\kk+3)},
		\end{equation}
		as per \eqref{eq:bpgens}.
		Moreover, it also maps $\slver{\wref{0} \cdot \mu}$ to the same $\ubp$-module.
		\item \cite[Thm.~6.7.4]{AraRep04} $\qhrmin$ maps $\slirr{\mu}$ to $0$, if the zeroth Dynkin label $\mu_0$ is a nonnegative integer, and otherwise to the irreducible $\ubp$-module $\bpirr{j_{\mu},\bpconfwt{\mu}}$.
	\end{itemize}
\end{theorem}

\subsection{Irreducible \hwms} \label{sec:bpirreps}

Recall that $\pwlat[n]$ denotes the set of $\aslthree$-weights $\lambda$ whose Dynkin labels $(\lambda_0,\lambda_1,\lambda_2)$ are nonnegative integers that sum to $n$.
The admissible $\aslthree$-weights of level $-3+\frac{\uu}{2}$, $\uu \in \set{3,5,7,\dots}$, may be partitioned as in \eqref{eq:admdisjoint}:
\begin{equation} \label{eq:sl3admissible}
	\admwts =
	\set[\big]{\lambda - \tfrac{\uu}{2} \fwt{i} \st \lambda \in \pwlat\ \text{and}\ i=0,1,2} \sqcup
	\set[\big]{\wref{1} \cdot (\lambda - \tfrac{\uu}{2} \fwt{1}) \st \lambda \in \pwlat}.
\end{equation}
Note that the weights in the first subset with $i=1,2$ give irreducible \hw\ $\usl$-modules in the kernel of $\qhrmin$, by \cref{thm:minqhr}.
Those with $i=0$ do not and neither do the weights of the second subset.
In fact,
\begin{equation} \label{eq:qhr2->1}
	\qhrmin(\slirr{\lambda-\uu\fwt{0}/2}) \cong \qhrmin(\slirr{\wref{0} \cdot (\lambda-\uu\fwt{0}/2)}) \ne 0, \quad \lambda \in \pwlat.
\end{equation}
Moreover, $\wref{0} \cdot \blank$ exchanges the $i=0$ weights of the first subset of \eqref{eq:sl3admissible} with those of the second.

Arakawa's classification result \cite{AraRat10} for irreducible $\bpminmod$-modules may now be expressed as follows.
\begin{theorem} \label{thm:bphwclass}
	For $\uu \in \set{3,5,7,\ldots}$, the set of isomorphism classes of irreducible $\bpminmod$-modules is in bijection with $\pwlat$.
	Moreover, a bijection is explicitly given by
	\begin{equation} \label{eq:bphwclass}
		\lambda \in \pwlat \longleftrightarrow \bpirr{\lambda} \equiv \qhrmin(\slirr{\lambda - \uu\fwt{0}/2}).
	\end{equation}
\end{theorem}
\noindent As $\lambda - \frac{\uu}{2}\fwt{0} \in \admwts$, for all $\lambda \in \pwlat$, the $\slirr{\lambda - \uu\fwt{0}/2}$ are $\slminmod$-modules (\cref{thm:sl3hwclass}).
The restriction of $\qhrmin$ to the BGG category $\ocat_{\kk}$ of $\slminmod$-modules is therefore surjective.

Combining this classification with \cref{thm:minqhr} now gives the following identification.
\begin{proposition} \label{prop:bphwjD}
	For $\uu \in \set{3,5,7,\ldots}$, the irreducible \hw\ $\bpminmod$-module $\bpirr{\lambda}$, $\lambda \in \pwlat$, has highest weight $\brac[\big]{j_{\lambda},\bpconfwt{\lambda}}$ given by
	\begin{equation} \label{eq:bphwjD}
		j_{\lambda} = -\frac{\lambda_1-\lambda_2}{3}, \quad
		\bpconfwt{\lambda} = \frac{(\lambda_1-\lambda_2)(\lambda_1-\lambda_2-\uu) + 3(\lambda_1+\lambda_2)(\lambda_1+\lambda_2-\uu+4)}{6\uu}.
	\end{equation}
	In other words, $\bpirr{\lambda} \cong \bpirr{j_{\lambda},\bpconfwt{\lambda}}$.
\end{proposition}
\noindent Note that the vacuum module $\slirr{\kk\fwt{0}}$ of $\slminmod$ corresponds to $\lambda=(\uu-3,0,0)$ and is sent by $\qhrmin$ to the vacuum module $\bpirr{(\uu-3)\fwt{0}} = \bpirr{0,0}$ of $\bpminmod$.

The following \lccrefs{prop:bphwtop,prop:bpconjsf} emphasise the naturalness of this parametrisation of the irreducible $\bpminmod$-modules.
\begin{proposition}[\protect{\cite[Prop.~4.11]{FehCla20}}] \label{prop:bphwtop}
	For $\uu \in \set{3,5,7,\ldots}$, the top space of $\bpirr{\lambda}$, $\lambda \in \pwlat$, has dimension $\lambda_2+1$.
\end{proposition}
\begin{proof}
	To find the dimension of the top space, we need to determine the minimal $n \in \ZZ_{\ge1}$ for which $G^+_0 (G^-_0)^n v_\lambda = 0$, where $v_{\lambda}$ is the \hwv\ of $\bpirr{\lambda}$.
	This is a straightforward calculation using \eqref{eq:commGG}:
	\begin{equation}
		\begin{split}
			G^+_0 (G^-_0)^n v_{\lambda}
			&= \sum_{m=0}^{n-1} (G^-_0)^{n-1-m} \comm{G^+_0}{G^-_0} (G^-_0)^m v_{\lambda}
			= \sum_{m=0}^{n-1} (G^-_0)^{n-1-m} \brac[\big]{(\uu-3)J_0 + 3J_0^2 - \tfrac{\uu}{2} L_0} (G^-_0)^m v_{\lambda} \\
			&= \sum_{m=0}^{n-1} \brac[\big]{(\uu-3)(j_{\lambda}-m) + 3(j_{\lambda}-m)^2 - \tfrac{\uu}{2}\bpconfwt{\lambda}} (G^-_0)^{n-1} v_{\lambda} \\
			&= \tfrac{1}{2}n \brac[\big]{2n^2 - (6j_{\lambda}+\uu)n + 6j_{\lambda}^2-2 + (2j_{\lambda}+1-\bpconfwt{\lambda})\uu} (G^-_0)^{n-1} v_{\lambda}.
		\end{split}
	\end{equation}
	Substituting \eqref{eq:bphwjD} into this somewhat unappealing expression improves it greatly:
	\begin{equation}
		G^+_0 (G^-_0)^n v_{\lambda} = n(n-\lambda_2-1)(n+\lambda_1+1-\tfrac{\uu}{2}) (G^-_0)^{n-1} v_{\lambda}.
	\end{equation}
	As $\uu$ is odd, this only vanishes for $n=\lambda_2+1 \in \ZZ_{\ge1}$.
\end{proof}
\begin{proposition}[\protect{\cite[Props.~4.13 and~4.14]{FehCla20}}] \label{prop:bpconjsf}
	For $\uu \in \set{3,5,7,\ldots}$, the conjugate and spectral flow of the $\bpminmod$-module $\bpirr{\lambda}$, $\lambda \in \pwlat$, are respectively given by
	\begin{equation}
		\bpconj\brac[\big]{\bpirr{\lambda}} \cong \bpirr{\dynk\nabla(\lambda)} \quad \text{and} \quad
		\bpsf{}\brac[\big]{\bpirr{\lambda}} \cong \bpirr{\nabla(\lambda)},
	\end{equation}
\end{proposition}
\begin{proof}
	It follows from \eqref{eq:bpconj} and \cref{prop:bphwtop} that a \hwv\ of $\bpconj\brac[\big]{\bpirr{\lambda}}$ is obtained as the image under $\bpconj$ of any vector in $\bpirr{\lambda}$ of weight $(j_{\lambda}-\lambda_2,\bpconfwt{\lambda})$.
	Noting that $2\kk+3=\uu-3$, we conclude from \eqref{eq:bpconjsfwts} that this \hwv\ has weight $(-j_{\lambda}+\lambda_2-\frac{\uu-3}{3},\bpconfwt{\lambda})$.
	Setting this equal to $(j_{\mu},\bpconfwt{\mu})$ results in a unique solution with $\mu\in\pwlat$.
	The same method works for spectral flow once we note that a \hwv\ of $\bpsf{}\brac[\big]{\bpirr{\lambda}}$ is also obtained as the image of a vector in $\bpirr{\lambda}$ of weight $(j_{\lambda}-\lambda_2,\bpconfwt{\lambda})$.
\end{proof}
\noindent The invertible functors corresponding to the automorphisms \eqref{eq:bpauts} thus act on the isomorphism classes of the irreducible \hw\ $\bpminmod$-modules as permutations of the Dynkin labels of the $\aslthree$-weight $\lambda$.

\subsection{Irreducible characters} \label{sec:bpchars}

Our next aim is to deduce a formula for the characters
\begin{equation} \label{eq:bpchar}
	\fch{\bpirr{\lambda}}{\zz;\qq} = \traceover{\bpirr{\lambda}} \zz^{J_0} \qq^{L_0-\ccbp/24}.
\end{equation}
of the irreducible \hw\ $\bpminmod$-modules $\bpirr{\lambda}$, $\lambda \in \pwlat$.
Because these modules may all be obtained as \qhrs\ of $\slminmod$-modules (\cref{thm:bphwclass}), their characters follow by applying the \ep\ principle to the cohomology of the complex $\qhrmodule{}{\slirr{\lambda-\uu\fwt{0}/2}}$.
However, it is well known (at least in the principal case \cite{FreCha92}) that this method leads to an indeterminate form for the character.
A key point is that the prescription used to resolve this indeterminacy must be independent of the \hwm\ being reduced.

Given a $\usl$-module $\slver{}$, the \ep\ principle shows that the character of its minimal \qhr\ has the form
\begin{equation}
	\ch{\qhrmin(\slver{})}
	= \ch{H^0 \brac[\big]{\qhrmodule{}{\slver{}},Q_0}}
	= \sum_{n \in \ZZ} (-1)^n \ch{\qhrmodule{n}{\slver{}}}
	= \ch{\slver{}} \sum_{n \in \ZZ} (-1)^n \ch{(\fgvoa^1 \otimes \fgvoa^3)^n},
\end{equation}
since the cohomology is concentrated in degree $0$ (\cref{thm:minqhr}).
Using the definition \eqref{eq:fg(s)char} of the fermionic ghost (super)character, the sum over $n=r+s$ is easily simplified:
\begin{equation}
	\begin{split}
		\sum_{n \in \ZZ} (-1)^n \fch{(\fgvoa^1 \otimes \fgvoa^3)^n}{\yy_1,\yy_3;\qq}
		&= \sum_{r \in \ZZ} \traceover{(\fgvoa^1)^r} (-\yy_1)^r \qq^{T^{\fgvoa^1}_0-1/12} \sum_{s \in \ZZ} \traceover{(\fgvoa^3)^s} (-\yy_3)^s \qq^{T^{\fgvoa^3}_0-1/12} \\
		&= \fsch{\fgvoa^1}{\yy_1;\qq} \fsch{\fgvoa^3}{\yy_3;\qq}.
	\end{split}
\end{equation}
Recalling the definitions \eqref{eq:slchar} and \eqref{eq:bpchar} of $\usl$- and $\ubp$-characters, it now follows from the explicit formulae \eqref{eq:bpgens} for $L$, $J$ and $\ccbp$ that
\begin{equation} \label{eq:indetchar}
	\fch{H^0 \brac[\big]{\qhrmodule{}{\slver{}},Q_0}}{\zz;\qq}
	= \qq^{\kk/3} \fch{\slver{}}{(\zz\qq^2)^{-1/3}, (\zz/\qq)^{1/3}; \qq} \fsch{\fgvoa^1}{\zz;\qq} \fsch{\fgvoa^3}{1;\qq}.
\end{equation}
This is problematic because $\fsch{\fgvoa^3}{\yy_3;\qq}$ has a factor of the form $1-\yy_3^{-1}$ that vanishes when $\yy_3 = 1$.

If we now take $\slver{}$ to be a Verma module, then its character may be written as a quotient with an infinite product in the denominator.
The factor in this product corresponding to the highest root $\srt{3}$ of $\slthree$ has the form $1-\zz_1\zz_2\qq$. This vanishes for $\zz_1=(\zz\qq^2)^{-1/3}$ and $\zz_2=(\zz/\qq)^{1/3}$, compensating for that of the factor $1-\yy_3^{-1}$ in $\fsch{\fgvoa^3}{1;\qq} = 0$.
This is the advertised indeterminacy of the character formula \eqref{eq:indetchar}.
We resolve it by first setting
\begin{subequations}
	\begin{equation} \label{eq:prescription1}
		\yy_3 = \zz_1^{-1}\zz_2^{-1}\qq^{-1},
	\end{equation}
	cancelling the vanishing factors in $\sch{\fgvoa^3}$ and $\ch{\slver{}}$, and only then setting
	\begin{equation} \label{eq:prescription2}
		 \zz_1=(\zz\qq^2)^{-1/3} \quad \text{and} \quad \zz_2=(\zz/\qq)^{1/3}.
	\end{equation}
\end{subequations}
Using \eqref{eq:fgschar} and some standard properties of Jacobi theta functions, our prescription for computing the character of the reduction thus takes the following form.
\begin{proposition} \label{prop:bpchar}
	Let $\slver{}$ be a \hw\ $\usl$-module, $\kk\ne-3$.
	The character of its minimal \qhr\ is then given by
	\begin{equation} \label{eq:prescription}
		\fch{\qhrmin(\slver{})}{\zz;\qq}
		= \Bigl. \zz^{-1/2} \qq^{(2\kk+3)/6} \fch{\slver{}}{\zz_1,\zz_2;\qq} \frac{\ii\fjth{1}{\zz;\qq}\,\ii\fjth{1}{\zz_1\zz_2;\qq}}{\eta(\qq)^2} \Bigr\rvert^{\zz_1=(\zz\qq^2)^{-1/3}}_{\zz_2=(\zz/\qq)^{1/3}}.
	\end{equation}
\end{proposition}
\noindent It is easy to check that applying this to a Verma $\usl$-module reproduces the character of a Verma $\ubp$-module.

In principle, we could apply this prescription when $\slver{}$ is one of the irreducible \hw\ $\slminmod$-modules using the explicit character formula given in \cref{thm:sl3admchar}.
However, we are not particularly interested in the resulting explicit formulae for the characters of the irreducible \hw\ $\bpminmod$-modules.
Our interest is rather in their modular S-transforms.

\subsection{Modular transforms} \label{sec:bpmod}

Recall (\cref{thm:bprational}) that for all $\uu \in \set{3,5,7,\dots}$, $\bpminmod$ is $C_2$-cofinite and rational.
Because its conformal weights are nonnegative integers, it follows that its irreducible characters, which are linearly independent, span a representation of the modular group \cite{ZhuMod96}.
Our goal is to determine the S-transforms of these characters using those given for the (meromorphic extensions of the) irreducible \hw\ $\slminmod$-characters in \cref{thm:sl3Smatrix}.

By \cref{thm:bphwclass}, we may restrict our reductions to the $\slirr{\mu}$, with $\mu=\lambda-\frac{\uu}{2}\fwt{0}$ and $\lambda \in \pwlat$.
We let $\zz = \ee^{2\pi\ii\zeta}$, $\zz_1 = \ee^{2\pi\ii\zeta_1}$, $\zz_2 = \ee^{2\pi\ii\zeta_2}$ and
$\qq = \ee^{2\pi\ii\tau}$, as before.
With these new variables, the character formula \eqref{eq:prescription} becomes
\begin{equation} \label{eq:qhrchar}
	\fch{\bpirr{\lambda}}{\zeta;\tau}
		= \Bigl. \ee^{-\pi\ii\zeta} \ee^{\pi\ii(2\kk+3)\tau/3} \fch{\slirr{\mu}}{\zeta_1,\zeta_2;\tau} \frac{\ii\fjth{1}{\zeta;\tau}\,\ii\fjth{1}{\zeta_1+\zeta_2;\tau}}{\eta(\tau)^2} \Bigr\rvert^{\zeta_1=-\frac{\zeta+2\tau}{3}}_{\zeta_2=\frac{\zeta-\tau}{3}}.
\end{equation}
Note the nonconstant $\lambda$-independent exponential prefactors of this formula.
They and their S-transforms are examples of automorphy factors and hence will be ignored in the modular computations to follow.

We require a couple of preparatory \lccrefs{lem:Tids,lem:lindepchars}.
\begin{lemma} \label{lem:Tids}
	For $\kk$ admissible, the T-matrix \eqref{eq:sl3Ttransform} of $\slminmod[\uu,\vv]$ satisfies
	\begin{equation}
		\slTmat{\wref{0}\cdot\nu} = \ee^{2\pi\ii \brac{\uu/\vv-\bilin{\nu}{\srt{3}}}} \slTmat{\nu}
		\quad \text{and} \quad
		\slTmat{\nabla(\nu)} = \ee^{2\pi\ii \brac{\uu/3\vv-\bilin{\nu}{\fwt{1}}}} \slTmat{\nu},
		\qquad \nu \in \admwts[\uu,\vv].
	\end{equation}
\end{lemma}
\begin{proof}
	These follow by inserting
	\begin{equation} \label{eq:nablaaction}
		\wref{0}(\nu+\wvec) = \wref{3}(\nu+\wvec) + \tfrac{\uu}{\vv} \srt{3} \quad \text{and} \quad
		\nabla(\nu+\wvec) = \wref{2}\wref{1}(\nu+\wvec) + \tfrac{\uu}{\vv} (\fwt{2}-\fwt{0})
	\end{equation}
	into \eqref{eq:sl3Ttransform} and using the orthogonality of Weyl reflections.
\end{proof}
Recall from \cref{rem:modularconvergence,cor:localisationchar} that the meromorphic extensions of the characters of the $\grp{D}_6$-twists and spectral flows of the \hw\ $\slminmod$-modules are linearly dependent.
A special case is as follows.
\begin{lemma} \label{lem:lindepchars}
	For $\uu \in \set{3,5,7,\dots}$, the meromorphic extensions of the \hw\ $\slminmod$-characters satisfy
	\begin{equation} \label{eq:lindepchars}
		\ch{\slsf{\fcwt{1}}(\slirr{\nu})} = \eps(\nu) \ch{\slirr{\xi(\nu)}}, \quad \nu \in \admwts,
	\end{equation}
	where $\eps(\nu) \in \set{\pm1}$ and $\xi(\nu) \in \admwts$ are defined in the following table.
	\begin{center}
		\begin{tabular}{C|CCCC}
			& \nu=\lambda-\frac{\uu}{2}\fwt{0} & \nu=\lambda-\frac{\uu}{2}\fwt{2} & \nu=\lambda-\frac{\uu}{2}\fwt{1} & \nu=\wref{1} \cdot (\lambda-\frac{\uu}{2}\fwt{1}) \\
			\hline
			\eps(\nu) & +1 & -1 & +1 & -1 \\
			\xi(\nu) & \nabla^{-1}(\lambda)-\frac{\uu}{2}\fwt{1} & \wref{1}\cdot(\lambda-\frac{\uu}{2}\fwt{1}) & \lambda-\frac{\uu}{2}\fwt{0} & \nabla^{-1}(\lambda)-\frac{\uu}{2}\fwt{2}
		\end{tabular}
	\end{center}
	Here, we parametrise $\nu \in \admwts$ in terms of $\lambda \in \pwlat$ as per \eqref{eq:sl3admissible}.
\end{lemma}
\begin{proof}
	We consider each of the four cases in detail.
	\begin{enumerate}
		\item\label{it:mod1} $\nu=\lambda-\frac{\uu}{2}\fwt{0}$:
		Here, $\slsf{\fcwt{1}}(\slirr{\nu})$ is lower bounded, hence is isomorphic to $\slirr{\xi}$ for some $\xi \in \admwts$.
		Since
		\begin{equation}
			e^i_n \slsf{\fcwt{1}}(v) = \slsf{\fcwt{1}}(e^i_{n+\bilin{\srt{i}}{\fcwt{1}}} v) \quad \text{and} \quad
			f^i_n \slsf{\fcwt{1}}(v) = \slsf{\fcwt{1}}(f^i_{n-\bilin{\srt{i}}{\fcwt{1}}} v), \qquad i=1,2,3,
		\end{equation}
		$\slsf{\fcwt{1}}(v)$ is a \hwv\ if $e^1_1 v = e^2_0 v = f^3_0 v = 0$.
		It follows that we should take $v$ to be the image of the \hwv\ of $\slirr{\nu}$ under $\wref{1}\wref{2}$, see \cref{fig:fourtopspaces}.
		Using \eqref{eq:sfwts}, the weight of $\slsf{\fcwt{1}}(v)$ is thus
		\begin{equation}
			\xi = \wref{1}\wref{2}(\nu) + \kk(\fwt{1}-\fwt{0}) = (\lambda_2,\lambda_0,\lambda_1)-\tfrac{\uu}{2}\fwt{1}.
		\end{equation}
		\item\label{it:mod2} $\nu=\lambda-\frac{\uu}{2}\fwt{2}$:
		In this case, there is no $v \in \slirr{\nu}$ for which $\slsf{\fcwt{1}}(v)$ is a \hwv\ (because $\slsf{\fcwt{1}}(\slirr{\nu})$ is not lower bounded).
		But, $\nu_1+\nu_2 \notin \ZZ_{\ge-1}$ implies that $f^3_0$ acts injectively on $\slirr{\nu}$.
		Localisation therefore results in a reducible semirelaxed \hw\ $\slminmod$-module with $\slirr{\nu}$ as a composition factor.
		By \cref{thm:hwlocalisation}, the quotient $\slver{}$ is irreducible.
		It has a vector of weight $\wref{1}(\nu)+\srt{3}$ whose image under $\slsf{\fcwt{1}}$ is \hw, see \cref{fig:fourtopspaces}.
		This identifies $\slsf{\fcwt{1}}(\slver{})$ with $\slirr{\xi}$, where
		\begin{equation}
			\xi = \wref{1}(\nu)+\srt{3} + \kk(\fwt{1}-\fwt{0}) = (\kk+1-\lambda_2,\kk+1-\lambda_1,\kk+1-\lambda_0).
		\end{equation}
		But, $\ch{\slirr{\nu}} = -\ch{\slver{}}$, by \cref{cor:localisationchar}, hence $\ch{\slsf{\fcwt{1}}(\slirr{\nu})} = -\ch{\slirr{\xi}}$ (as meromorphic functions).
		\item\label{it:mod3} $\nu=\lambda-\frac{\uu}{2}\fwt{1}$:
		Again, there is no vector in $\slirr{\nu}$ whose image under $\slsf{\fcwt{1}}$ is \hw.
		As $\nu_1 \notin \NN$, we may proceed as in case~\ref{it:mod2} and localise with respect to $f^1_0$ to get a semirelaxed \hwm\ with irreducible quotient $\slver{}$.
		Unfortunately, this quotient also has no vector whose image under $\slsf{\fcwt{1}}$ is \hw, see \cref{fig:fourtopspaces}.
		But, \cref{thm:hwlocalisation} gives $\slver{}$ as the $\wref{1}$-twist of an irreducible \hwm\ of highest weight
		\begin{equation}
			\wref{1}(\nu+\srt{1}) = (\kk+1-\nu_2,-\nu_1-1,\kk+1-\nu_0).
		\end{equation}
		As $\kk+1-\nu_0 \notin \NN$, we may localise $\slver{}$ with respect to $\wref{1}(f^2_0) = f^3_0$ and obtain another semirelaxed \hwm.
		The quotient $\slext{}$ of this semirelaxed module does have a vector of the desired type, so we deduce from \cref{cor:localisationchar} that $\ch{\slirr{\nu}} = -\ch{\slver{}} = \ch{\slext{}}$, hence $\ch{\slsf{\fcwt{1}}(\slirr{\nu})} = \ch{\slirr{\xi}}$ with
		\begin{equation}
			\xi = \nu+\srt{1}+\srt{3} + \kk(\fwt{1}-\fwt{0}) = (\lambda_0,\lambda_1,\lambda_2)-\tfrac{\uu}{2}\fwt{0}.
		\end{equation}
		\item\label{it:mod4} $\nu=\wref{1}\cdot(\lambda-\frac{\uu}{2}\fwt{1})$:
		Our final case also requires a localisation with respect to $f^1_0$, as in case~\ref{it:mod3}, whence the introduction of the quotient $\slver{}$.
		However, this time $\kk+1-\nu_0$ lies in $\NN$ as it is the second Dynkin label of $\wref{1}\cdot\nu$.
		Thus, $\slver{}$ has a vector whose image under $\slsf{\fcwt{1}}$ is \hw, see \cref{fig:fourtopspaces}.
		Proceeding as above, it follows that $\ch{\slsf{\fcwt{1}}(\slirr{\nu})} = -\ch{\slirr{\xi}}$, where
		\begin{equation}
			\xi = \wref{3}(\nu+\srt{1}) + \kk(\fwt{1}-\fwt{0}) = (\lambda_2,\lambda_0,\lambda_1)-\tfrac{\uu}{2}\fwt{2}. \qedhere
		\end{equation}
	\end{enumerate}
\end{proof}
\begin{figure}
	\begin{tikzpicture}[scale=1]
		\draw[blue,->] (0:0) -- (0:1) node[black,right] {$e^1_0$};
		\draw[blue,->] (0:0) -- (120:1) node[black,left] {$e^2_0$};
		\draw[blue,->] (0:0) -- (-120:1) node[black,left] {$f^3_0$};
		\begin{scope}[shift={(-4,0)}]
			\fill[fillcolourA,draw] (0:1) -- (60:1) -- (120:1) -- (180:1) -- (-120:1) -- (-60:1) -- cycle;
			\node[wt,label=above right:{\footnotesize$\nu$}] at (60:1) {};
			\node[wt,label=left:{\footnotesize$\wref{1}\wref{2}(\nu)$}] at (180:1) {};
			\node at (0:0) {$\slirr{\nu}$};
			\node at (-90:1.5) {$\nu=\lambda-\frac{\uu}{2}\fwt{0}$};
		\end{scope}
		\begin{scope}[shift={(4,-1)}]
			\fill[fillcolourA] (0:1) -- (60:1) -- (120:1) -- (180:1) -- cycle;
			\draw (0:1) -- (60:1) -- (120:1) -- (180:1);
			\node[wt,label=above:{\footnotesize$\nu$}] at (60:1) {};
			\node[wt,label=left:{\footnotesize$\wref{1}(\nu)$}] at (120:1) {};
			\node at (90:0.25) {$\slirr{\nu}$};
			\node at (-90:0.5) {$\nu=\lambda-\frac{\uu}{2}\fwt{2}$};
			\begin{scope}[shift={($(120:1)+(60:0.5)$)}]
				\fill[fillcolourB] (0:0) -- (60:1) -- ($(60:1)+(0:1.25)$) -- (-60:1.25) -- cycle;
				\draw (-60:1.25) -- (0:0) -- (60:1);
				\node[wt,label=above left:{\footnotesize$\wref{1}(\nu)+\srt{3}$}] at (0:0) {};
				\node at (30:0.87) {$\slver{}$};
			\end{scope}
		\end{scope}
		\begin{scope}[shift={(-5,-5)}]
			\fill[fillcolourA] (-60:1) -- (0:1) -- (60:1) -- (120:1) -- cycle;
			\draw (-60:1) -- (0:1) -- (60:1) -- (120:1);
			\node[wt,label=above:{\footnotesize$\nu$}] at (60:1) {};
			\node at (30:0.25) {$\slirr{\nu}$};
			\node at ($(-90:1.5)+(1,0)$) {$\nu=\lambda-\frac{\uu}{2}\fwt{1}$};
			\begin{scope}[shift={($(60:1)+(0:0.5)$)}]
				\fill[fillcolourB] (0:0) -- (0:1.5) -- ($(-120:1.5)+(0:2.25)$) -- (-120:1.5) -- cycle;
				\draw (-120:1.5) -- (0:0) -- (0:1.5);
				\node[wt,label={[shift={(-70:0.8)}]\footnotesize$\nu+\srt{1}$}] at (0:0) {};
				\node at ($(-60:1.25)+(0.1,0.1)$) {$\slver{}$};
				\begin{scope}[shift={(60:0.5)}]
					\fill[fillcolourC] (0:0) -- (60:1) -- ($(60:1)+(0:1.25)$) -- (-60:1.25) -- cycle;
					\draw (-60:1.25) -- (0:0) -- (60:1);
					\node[wt,label={[shift={(175:0.6)}]\footnotesize$\nu+\srt{1}+\srt{3}$}] at (0:0) {};
					\node at (30:0.87) {$\slext{}$};
				\end{scope}
			\end{scope}
		\end{scope}
		\begin{scope}[shift={(3.5,-5)}]
			\fill[fillcolourA] ($(-60:0.5)+(0:1)$) -- (60:1) -- ($(120:1)+(180:0.6)$) -- ($(-60:0.5)+(180:1.35)$) -- cycle;
			\draw ($(-60:0.5)+(0:1)$) -- (60:1) -- ($(120:1)+(180:0.6)$);
			\node[wt,label=above:{\footnotesize$\nu$}] at (60:1) {};
			\node at (135:0.75) {$\slirr{\nu}$};
			\node at ($(-90:1.5)+(0.5,0)$) {$\nu=\wref{1}\cdot(\lambda-\frac{\uu}{2}\fwt{1})$};
			\begin{scope}[shift={($(60:1)+(0:0.5)$)}]
				\fill[fillcolourB] (0:0) -- (0:1) -- ($(-120:1)+(-60:1)$) -- (-120:1) -- cycle;
				\draw ($(-120:1)+(-60:1)$) -- (-120:1) -- (0:0) -- (0:1);
				\node[wt,label=above right:{\footnotesize$\nu+\srt{1}$}] at (0:0) {};
				\node[wt,label=left:{\footnotesize$\wref{3}(\nu+\srt{1})$}] at (-120:1) {};
				\node at ($(-60:0.5)$) {$\slver{}$};
			\end{scope}
		\end{scope}
	\end{tikzpicture}
\caption{%
	Depictions of the (convex hull of the weights of the) top space of $\slirr{\nu}$, where $\nu$ depends on $\lambda\in\pwlat$, and those of its colleagues $\slver{}$ and $\slext{}$ when needed, for each of the four cases analysed in the proof of \cref{lem:lindepchars}.
} \label{fig:fourtopspaces}
\end{figure}

\begin{theorem} \label{thm:bpmod}
	For all $\uu \in \set{3,5,7,\dots}$, the S-transform of the character of the irreducible $\bpminmod$-module $\bpirr{\lambda}$, $\lambda \in \pwlat$, is given, up to an omitted automorphy factor, by
	\begin{subequations}
		\begin{gather}
			\fch{\bpirr{\lambda}}{\tfrac{\zeta}{\tau};-\tfrac{1}{\tau}} = \:\sum_{\mathclap{\lambda'\in\pwlat}}\: \bpSmat{\lambda,\lambda'} \fch{\bpirr{\lambda'}}{\zeta;\tau}, \\
			\bpSmat{\lambda,\lambda'} = \frac{\ii}{\sqrt{3}\uu} \ee^{2\pi\ii(j_{\lambda}+j_{\lambda'}-\uu/3)} \sum_{\wref{} \in \grp{S}_3} \det(\wref{}) \ee^{-4\pi\ii \bilin*{\wref{}(\lambda+\wvec)}{\lambda'+\wvec}/\uu}. \label{eq:bpSmatrix}
		\end{gather}
	\end{subequations}
\end{theorem}
\begin{proof}
	We begin by substituting the known S-transforms of the $\ch{\slirr{\mu}}$ (given in \cref{thm:sl3Smatrix}), $\jth{1}$ and $\eta$:
	\begin{equation}
		\begin{split}
			\fch{\bpirr{\lambda}}{\tfrac{\zeta}{\tau};-\tfrac{1}{\tau}}
				&= \Bigl. \fch{\slirr{\mu}}{\zeta_1,\zeta_2;-\tfrac{1}{\tau}} \frac{\ii\fjth{1}{\tfrac{\zeta}{\tau};-\tfrac{1}{\tau}}\,\ii\fjth{1}{\zeta_1+\zeta_2;-\tfrac{1}{\tau}}}{\eta(-\tfrac{1}{\tau})^2} \Bigr\rvert^{\zeta_1=-\frac{\zeta-2}{3\tau}}_{\zeta_2=\frac{\zeta+1}{3\tau}} \\
				&= \Bigl. -\:\sum_{\mathclap{\nu \in \admwts}}\: \slSmat{\mu,\nu} \fch{\slirr{\nu}}{\zeta_1\tau,\zeta_2\tau;\tau} \frac{\ii\fjth{1}{\zeta;\tau}\,\ii\fjth{1}{\zeta_1\tau+\zeta_2\tau;\tau}}{\eta(\tau)^2} \Bigr\rvert^{\zeta_1=-\frac{\zeta-2}{3\tau}}_{\zeta_2=\frac{\zeta+1}{3\tau}}.
		\end{split}
	\end{equation}
	(Here, as always, we omit some automorphy factors.)
	Define $\zeta_1' = -\frac{\zeta+2\tau}{3}$ and $\zeta_2'=\frac{\zeta-\tau}{3}$, so that the given specialisations of $\zeta_1$ and $\zeta_2$ imply that
	\begin{equation}
		\zeta_1\tau = \zeta_1' + \tfrac{2}{3}(\tau+1) \quad \text{and} \quad
		\zeta_2\tau = \zeta_2' + \tfrac{1}{3}(\tau+1).
	\end{equation}
	Substituting for $\zeta_1$ and $\zeta_2$ throughout then gives
	\begin{equation}
		\fch{\bpirr{\lambda}}{\tfrac{\zeta}{\tau};-\tfrac{1}{\tau}}
			= \Bigl. -\:\sum_{\mathclap{\nu \in \admwts}}\: \slSmat{\mu,\nu} \fch{\slirr{\nu}}{\zeta_1'+\tfrac{2}{3}(\tau+1),\zeta_2'+\tfrac{1}{3}(\tau+1);\tau} \frac{\ii\fjth{1}{\zeta;\tau}\,\ii\fjth{1}{\zeta_1'+\zeta_2';\tau}}{\eta(\tau)^2} \Bigr\rvert^{\zeta_1'=-\frac{\zeta+2\tau}{3}}_{\zeta_2'=\frac{\zeta-\tau}{3}},
	\end{equation}
	where we once again ignore an automorphy factor.

	If we drop the primes in this expression, the factor multiplying $\slSmat{\mu,\nu}$ looks like the \rhs\ of \eqref{eq:qhrchar}, except for the shifts in the arguments of $\ch{\slirr{\nu}}$.
	However, we can shift the third argument to $\tau+1$ by applying an inverse T-transform.
	Comparing with \eqref{eq:slsfchar}, we see that the result is now the character of a spectral flow of $\slirr{\nu}$:
	\begin{equation} \label{eq:nearlythere}
		\fch{\bpirr{\lambda}}{\tfrac{\zeta}{\tau};-\tfrac{1}{\tau}}
			= \Bigl. -\ee^{-2\pi\ii\kk/3} \:\sum_{\mathclap{\nu \in \admwts}}\: \slSmat{\mu,\nu} \slTmat{\nu}^{-1} \fch{\slsf{\fcwt{1}}(\slirr{\nu})}{\zeta_1,\zeta_2;\tau+1} \frac{\ii\fjth{1}{\zeta;\tau}\,\ii\fjth{1}{\zeta_1+\zeta_2;\tau}}{\eta(\tau)^2} \Bigr\rvert^{\zeta_1=-\frac{\zeta+2\tau}{3}}_{\zeta_2=\frac{\zeta-\tau}{3}}.
	\end{equation}
	Here, we again drop some automorphy factors, but keep a constant exponential prefactor.

	Because $\bpminmod$ is $C_2$-cofinite and rational, the $\ch{\bpirr{\lambda}}$ are holomorphic for $\zeta \in \CC$ and $\tau$ in the upper-half plane.
	We are therefore justified in working with the meromorphic extensions of the $\ch{\slsf{\fcwt{1}}(\slirr{\nu})}$ to these domains.
	Invoking \cref{lem:lindepchars} (and using the notation $\eps(\nu)$ and $\xi(\nu)$ defined there), \eqref{eq:nearlythere} simplifies to
	\begin{equation} \label{eq:bpStransform1}
		\begin{split}
			\fch{\bpirr{\lambda}}{\tfrac{\zeta}{\tau};-\tfrac{1}{\tau}}
			&= \Bigl. -\ee^{-2\pi\ii\kk/3} \:\sum_{\mathclap{\nu \in \admwts}}\: \slSmat{\mu,\nu} \frac{\slTmat{\xi(\nu)}}{\slTmat{\nu}} \eps(\nu) \fch{\slirr{\xi(\nu)}}{\zeta_1,\zeta_2;\tau} \frac{\ii\fjth{1}{\zeta;\tau}\,\ii\fjth{1}{\zeta_1+\zeta_2;\tau}}{\eta(\tau)^2} \Bigr\rvert^{\zeta_1=-\frac{\zeta+2\tau}{3}}_{\zeta_2=\frac{\zeta-\tau}{3}} \\
			&= -\ee^{-\pi\ii\uu/3} \:\sum_{\mathclap{\nu \in \admwts}}\: \slSmat{\mu,\nu} \frac{\slTmat{\xi(\nu)}}{\slTmat{\nu}} \eps(\nu) \fch{\qhrmin(\slirr{\xi(\nu)})}{\zeta;\tau}
		\end{split}
	\end{equation}
	(dropping even more automorphy factors).

	Of the four cases for $\nu \in \admwts$ considered in \cref{lem:lindepchars}, cases~\ref{it:mod1} and~\ref{it:mod4} both contribute nothing to \eqref{eq:bpStransform1} because $\slirr{\xi(\nu)} \in \ker \qhrmin$, by \cref{thm:minqhr}.
	In case~\ref{it:mod3}, we have $\nu = \lambda'-\frac{\uu}{2}\fwt{1}$, for some $\lambda' \in \pwlat$, $\eps(\nu)=+1$ and $\xi(\nu)=\lambda'-\frac{\uu}{2}\fwt{0}$.
	The contribution to the coefficient of $\ch{\bpirr{\lambda'}} = \ch{\qhrmin(\slirr{\xi(\nu)})}$ in \eqref{eq:bpStransform1} is therefore
	\begin{equation}
		-\ee^{-\pi\ii\uu/3} \slSmat{\lambda-\frac{\uu}{2}\fwt{0},\,\lambda'-\frac{\uu}{2}\fwt{1}} \frac{\slTmat{\lambda'-\frac{\uu}{2}\fwt{0}}}{\slTmat{\lambda'-\frac{\uu}{2}\fwt{1}}}
		= \frac{\ii}{2\sqrt{3}\uu} \ee^{2\pi\ii\bilin{\lambda+\lambda'}{\fwt{1}}} \ee^{-2\pi\ii\uu/3} \sum_{\wref{} \in \grp{S}_3} \det(\wref{}) \ee^{-4\pi\ii \bilin*{\wref{}(\lambda+\wvec)}{\lambda'+\wvec}/\uu},
	\end{equation}
	by \eqref{eq:sl3Stransform} and \eqref{eq:sl3Ttransform}.
	In case~\ref{it:mod2}, we instead have $\nu = \lambda'-\frac{\uu}{2}\fwt{2}$, $\eps(\nu)=-1$ and $\xi(\nu) = \wref{0}\cdot\brac[\big]{\nabla^{-1}(\lambda')-\frac{\uu}{2}\fwt{0}}$.
	Replacing $\lambda'$ by $\nabla(\lambda') \in \pwlat$, the case~\ref{it:mod2} contribution to the coefficient of $\ch{\bpirr{\lambda'}} = \ch{\qhrmin(\slirr{\xi(\nu)})}$ now follows using \cref{lem:Tids} and \cref{eq:nablaaction}:
	\begin{equation}
		\ee^{-\pi\ii\uu/3} \slSmat{\lambda-\frac{\uu}{2}\fwt{0},\,\nabla(\lambda')-\frac{\uu}{2}\fwt{2}} \frac{\slTmat{\wref{0}\cdot(\lambda'-\frac{\uu}{2}\fwt{0})}}{\slTmat{\nabla(\lambda')-\frac{\uu}{2}\fwt{2}}}
		= \frac{\ii}{2\sqrt{3}\uu} \ee^{2\pi\ii\bilin{\lambda+\lambda'}{\fwt{1}}} \ee^{-2\pi\ii\uu/3} \sum_{\wref{} \in \grp{S}_3} \det(\wref{}) \ee^{-4\pi\ii \bilin*{\wref{}(\lambda+\wvec)}{\lambda'+\wvec}/\uu}.
	\end{equation}
	The contributions are thus equal and the proof is completed by noting that
	\begin{equation} \label{eq:j=<-,fwt1>}
		j_{\lambda} = \bilin{\lambda}{\fwt{1}} \mod{\ZZ}. \qedhere
	\end{equation}
\end{proof}

We finish with two simple identities that will be used in \cref{sec:mod}.
\begin{corollary} \label{cor:bpSsymm}
	For all $\uu \in \set{3,5,7,\dots}$, $\lambda,\lambda' \in \pwlat$ and $n\in\ZZ$, we have
	\begin{equation}\label{eq:bpSrel}
		\bpSmat{\nabla^n(\lambda),\lambda'}
		= \ee^{-2\pi\ii n(j_{\lambda'}-\uu/3)} \bpSmat{\lambda,\lambda'} \quad \text{and} \quad
		\bpSmat{\dynk(\lambda),\dynk(\lambda')}
		= \ee^{2\pi\ii (j_{\lambda}+j_{\lambda'})} \bpSmat{\lambda,\lambda'}.
	\end{equation}
\end{corollary}
\begin{proof}
	Since $j_{\nabla(\lambda)}-j_{\lambda} = \frac{\uu}{3} \bmod{\ZZ}$, the factor $\ee^{2\pi\ii n\uu/3}$ in the first identity is easy.
	For the other factor, consider
	\begin{equation}\label{eq:snabla}
		\mathfrak{S}_{\lambda,\lambda'}
		= \sum_{\wref{} \in \grp{S}_3} \det(\wref{}) \ee^{-4\pi\ii \bilin*{\wref{}(\lambda+\wvec)}{\lambda'+\wvec}/\uu}.
	\end{equation}
	As per \eqref{eq:nablaaction}, $\lambda\in \pwlat$ gives $\nabla(\lambda)+\wvec = \wref{2}\wref{1}(\lambda+\wvec) + \uu (\fwt{2}-\fwt{0})$, hence
	\begin{equation}
		\mathfrak{S}_{\nabla(\lambda),\lambda'}
		= \sum_{\wref{} \in \grp{S}_3} \det(\wref{}) \ee^{-4\pi\ii\bilin*{\wref{}(\lambda+\wvec)}{\lambda'+\wvec}/\uu}
		\ee^{-4\pi\ii \bilin*{\wref{}\wref{1}\wref{2}(\fwt{2}-\fwt{0})}{\lambda'+\wvec}}
		= \ee^{2\pi\ii\bilin{\fwt{2}}{\lambda'}} \mathfrak{S}_{\lambda,\lambda'},
	\end{equation}
	since the action of $\grp{S}_3$ on $\wlat$ shifts weights by elements of $\rlat$.
	This proves the first identity since $\fwt{2}=-\fwt{1} \bmod{\rlat}$.
	The second identity follows similarly using the orthogonality of $\dynk$ and $j_{\dynk(\lambda)} = -j_{\lambda}$.
\end{proof}

\subsection{Fusion rules} \label{sec:bpfusion}

Given $\uu \in \set{3,5,7,\dots}$, compare the S-matrix \eqref{eq:bpSmatrix} of the \bp\ minimal model $\bpminmod$ with that of the rational $\slthree$ minimal model $\slminmod[\uu,1]$ given in \eqref{eq:sl3Smatrixv=1}.
The only difference in the sums over $\grp{S}_3$ is that the $2$ in the exponent of \eqref{eq:sl3Smatrixv=1} is replaced by a $4$ in \eqref{eq:bpSmatrix}.
This suggests \cite{deBMar91} realising the latter as the result of applying a Galois symmetry to the former.
Since $\bilin{\fwt{i}}{\fwt{j}} \in \frac{1}{3}\ZZ$ for all $i,j=1,2$, any such Galois symmetry must map $\ee^{2\pi\ii/3\uu}$ to $\ee^{4\pi\ii/3\uu}$.

In fact, the factors of $\ii$ and $\sqrt{3}$ in these S-matrices indicate \cite{CosRem94} a Galois symmetry of $\QQ[\xi]$, where $\xi$ is a primitive $12\uu$-th root of unity.
We take $\xi = \ee^{\pi\ii/6\uu}$ for definiteness, so that $\ii = \xi^{3\uu}$ and $\sqrt{3} = \xi^{\uu}+\xi^{-\uu}$.
Requiring that $\xi^4$ be mapped to $\xi^8$ now gives four candidates, $\xi \mapsto \xi^2$, $\xi^{3\uu+2}$, $\xi^{6\uu+2}$ or $\xi^{9\uu+2}$, but only the second and fourth preserve the primitivity of $\xi$.
We choose the second, denoting the Galois symmetry $\xi \mapsto \xi^{3\uu+2}$ of $\QQ[\xi]$ by $\Gal$.
We thus have
\begin{equation} \label{eq:galoisprefactor}
	\begin{gathered}
		\Gal(\ii) = \ii^{3\uu+2} = (-1)^{(\uu-1)/2} \, \ii, \\
		\Gal(\sqrt{3}) = \ee^{\pi\ii(3\uu+2)/6}+\ee^{-\pi\ii(3\uu+2)/6} = (-1)^{(\uu+1)/2} \, \sqrt{3}
	\end{gathered}
	\quad \Ra \quad
	\Gal\brac*{\frac{-\ii}{\sqrt{3}\uu}} = \frac{\ii}{\sqrt{3}\uu}.
\end{equation}
It follows that the relation between the $\bpminmod$ and $\slminmod[\uu,1]$ S-matrices is
\begin{equation} \label{eq:galoisrelation}
	\bpSmat{\lambda,\lambda'} = \ee^{2\pi\ii(j_{\lambda}+j_{\lambda'}-\uu/3)} \Gal \brac*{\slSmat{\lambda,\lambda'}^{(\uu,1)}}, \quad \lambda,\lambda' \in \pwlat.
\end{equation}

As was first noted in \cite{CosRem94}, there exists a permutation $\pi$ of $\pwlat$ and a function $\epsilon \colon \pwlat \to \set{\pm1}$ satisfying
\begin{equation} \label{eq:galoispermutation}
	\Gal \brac[\big]{\slSmat{\lambda,\lambda'}^{(\uu,1)}} = \epsilon(\lambda') \slSmat{\lambda,\pi(\lambda')}^{(\uu,1)}, \quad \lambda,\lambda' \in \pwlat.
\end{equation}
The function $\epsilon$ is easily understood in this example.
First, it includes the sign $-1$ obtained by applying $\Gal$ to $\frac{-\ii}{\sqrt{3}\uu}$ in \eqref{eq:galoisprefactor}.
Otherwise, it describes the effect of applying $\Gal$ to the sum in \eqref{eq:sl3Smatrixv=1}.
This amounts to replacing $\lambda' \in \pwlat$ by $2\lambda'+\wvec$.
However,
\begin{equation}
	\bilin{\lambda'+\wvec}{\srt{i}} \notin \uu\ZZ \quad \Ra \quad \bilin*{(2\lambda'+\wvec)+\wvec}{\srt{i}} \notin \uu\ZZ, \qquad i=1,2,3,
\end{equation}
hence there exists a unique element of the affine Weyl group of $\aslthree$ mapping $2\lambda'+\wvec$ back into $\pwlat$.
Because of the symmetries of $\slSmat{\lambda,\lambda'}^{(\uu,1)}$ under affine Weyl reflections (\cref{sec:sl3v=1}), $\epsilon(\lambda')$ is identified with (the negative of) the determinant of this unique Weyl element.

The fusion rules of $\bpminmod$ now follow as a simple consequence of these considerations, if we assume that the fusion coefficients are given by the Verlinde formula.
As $\bpminmod$ is $C_2$-cofinite and rational, this would be so, by a famous theorem of Huang \cite{HuaVer04a}, except that this theorem requires the \voa\ under consideration to be self-contragredient with nonnegative-integer conformal weights.
Unfortunately, there is no choice of conformal structure for $\bpminmod$ satisfying both these requirements: the self-contragredient choice results in the conformal weights of both $G^+$ and $G^-$ being $\frac{3}{2}$.

We will therefore \emph{assume} for the following \lcnamecref{thm:bpfusion} that Huang's theorem can be generalised to cover $\bpminmod$.
To prove this, one could take the self-contragredient conformal structure and pass to the $\ZZ_2$-orbifold spanned by the vectors of integral conformal weight, reconstructing the Verlinde formula for $\bpminmod$ using induction as in \cite[App.~A]{CanFus15b}.
But, we will not pursue this proof here because the fusion rules reported below will only be used to elucidate those of $\slminmod$ in \cref{sec:fusion} and this elucidation will also rest upon further assumptions.
\begin{theorem} \label{thm:bpfusion}
	For each $\uu \in \set{3,5,7,\dots}$, the fusion rules of the irreducible $\bpminmod$-modules take the form
	\begin{subequations}
		\begin{equation}
			\bpirr{\lambda} \fuse \bpirr{\mu} \cong \:\bigoplus_{\mathclap{\nu \in \pwlat}}\: \bpfuscoeff{\lambda}{\mu}{\nu} \, \bpirr{\nu}, \quad \lambda,\mu \in \pwlat,
		\end{equation}
		where the fusion coefficients coincide with those of $\slminmod[\uu,1]$ given in \cref{eq:sl3fusioncoeffs}:
		\begin{equation}
			\bpfuscoeff{\lambda}{\mu}{\nu} = \slfuscoeff[\uu,1]{\lambda}{\mu}{\nu}, \quad \lambda,\mu,\nu \in \pwlat.
		\end{equation}
	\end{subequations}
\end{theorem}
\begin{proof}
	Since $\bpminmod$ and $\slminmod[\uu,1]$ are $C_2$-cofinite and rational (\cref{thm:sl3rational,thm:bprational}), their fusion products are completely reducible and, with the assumption described above, the multiplicity with which each irreducible appears is given by the Verlinde formula.
	Let $\vac = (\uu-3)\fwt{0} \in \pwlat$ denote the $\aslthree$-weight corresponding to the vacuum module of $\bpminmod$ and $\slminmod[\uu,1]$.
	Substituting \cref{eq:galoisrelation,eq:galoispermutation} into the $\bpminmod$ Verlinde formula then gives
	\begin{equation} \label{eq:bpVerlinde}
		\bpfuscoeff{\lambda}{\mu}{\nu}
		= \:\sum_{\mathclap{\Lambda \in \pwlat}}\: \frac{\bpSmat{\lambda,\Lambda}\bpSmat{\mu,\Lambda} (\bpSmat{\nu,\Lambda})^*}{\bpSmat{\vac,\Lambda}}
		= \ee^{2\pi\ii(j_{\lambda}+j_{\mu}-j_{\nu})} \:\sum_{\mathclap{\Lambda \in \pwlat}}\: \frac{\slSmat{\lambda,\pi(\Lambda)}^{(\uu,1)} \slSmat{\mu,\pi(\Lambda)}^{(\uu,1)} (\slSmat{\nu,\pi(\Lambda)}^{(\uu,1)})^*}{\slSmat{\vac,\pi(\Lambda)}^{(\uu,1)}}
		= \slfuscoeff[\uu,1]{\lambda}{\mu}{\nu},
	\end{equation}
	since the signs $\epsilon(\Lambda)$ and the phases $\ee^{2\pi\ii(j_{\Lambda}-\uu/3)}$ cancel, $\pi$ is a permutation, $j_{\lambda}+j_{\mu}-j_{\nu} = \bilin{\lambda+\mu-\nu}{\fwt{1}} \bmod{\ZZ}$ by \eqref{eq:j=<-,fwt1>}, and the $\slminmod[\uu,1]$ fusion coefficient on the \rhs\ vanishes unless the projection of $\lambda+\mu-\nu$ onto $\csub^*$ lies in $\rlat$ (\cref{sec:sl3v=1}).
\end{proof}

We conclude with two straightforward identities that the $\bpminmod$ fusion coefficients satisfy.
They may be deduced from \cref{thm:bpfusion} as consequences of the corresponding identities for $\slminmod[\uu,1]$ or by substituting the identities of \cref{cor:bpSsymm} into the Verlinde formula.
\begin{corollary} \label{cor:bpfussymm}
	For all $\uu\in\set{3,5,7,\dots}$, $\lambda,\lambda,\lambda''\in\pwlat$ and $n\in\ZZ$, we have
	\begin{equation}
		\bpfuscoeff{\nabla^n(\lambda)}{\lambda'}{\lambda''} = \bpfuscoeff{\lambda}{\nabla^n(\lambda')}{\lambda''} = \bpfuscoeff{\lambda}{\lambda'}{\nabla^{-n}(\lambda'')} \quad \text{and} \quad
		\bpfuscoeff{\dynk(\lambda)}{\dynk(\lambda')}{\dynk(\lambda'')} = \bpfuscoeff{\lambda}{\lambda'}{\lambda''}.
	\end{equation}
\end{corollary}

\section{Inverse \qhr} \label{sec:conventions}

In this \lcnamecref{sec:conventions}, we realise $\usl$ as a subalgebra of the tensor product of the universal \bp\ \voa\ $\ubp$ and some free-field algebras.
We also review the analogous realisation of $\slminmod[\uu,\vv]$, following \cite{AdaRel21} (to which we refer for further information).

\subsection{Bosonic ghosts} \label{sec:bghosts}

The bosonic ghost system, also known as a symplectic boson pair, is the simple vertex algebra $\bgvoa$ strongly and freely generated by bosonic fields $\beta(z)$ and $\gamma(z)$ satisfying the \opes
\begin{equation}
	\beta(z)\gamma(w) \sim \frac{-\wun}{z-w}, \quad \beta(z)\beta(w) \sim 0 \sim \gamma(z)\gamma(w).
\end{equation}
This vertex algebra admits a one-parameter family of conformal structures.
We shall choose the associated \emt\ to be
\begin{equation} \label{eq:emtbg}
	T^{\bgsymb}(z) = -\no{\partial \gamma \beta}(z),
\end{equation}
so that the central charge is $\ccbg=2$ and the fields $\beta(z)$ and $\gamma(z)$ are conformal primaries of weights $1$ and $0$, respectively.

The composite field $\no{\beta\gamma}(z)$ generates a rank-$1$ Heisenberg subalgebra of $\bgvoa$.
The action of the zero-mode $\no{\beta\gamma}_0$ gives rise to a second ``horizontal'' grading on $\bgvoa$.
We have
\begin{equation}
	\begin{gathered}
    \no{\beta\gamma}(z)\beta(w) \sim \frac{\beta(w)}{z-w}, \quad
    \no{\beta\gamma}(z)\gamma(w) \sim -\frac{\gamma(w)}{z-w}, \quad
    \no{\beta\gamma}(z)\no{\beta\gamma}(w) \sim -\frac{\wun}{(z-w)^2}, \\
    T^{\bgsymb}(z)\no{\beta\gamma}(w) \sim -\frac{\wun}{(z-w)^3} + \frac{\no{\beta\gamma}(w)}{(z-w)^2} + \frac{\pd \no{\beta\gamma}(w)}{z-w}.
	\end{gathered}
\end{equation}

The modes of the bosonic ghost \voa\ are defined by the expansions
\begin{equation}
	\beta(z) = \sum_{n \in \ZZ} \beta_n z^{-n-1} \quad \text{and} \quad \gamma(z) = \sum_{n \in \ZZ} \gamma_n z^{-n}.
\end{equation}
The mode algebra is then (a completion of) the \uea\ of an infinite-dimensional Lie algebra $\bgalg = \spn\set{\beta_n, \gamma_n, \wun \st n \in \ZZ}$ with $\wun$ central and the remaining Lie brackets being given by
\begin{equation}
	\comm{\beta_{m}}{\gamma_{n}} = -\delta_{m+n,0}\wun, \quad
	\comm{\beta_{m}}{\beta_{n}} = \comm{\gamma_{m}}{\gamma_{n}} = 0.
\end{equation}

This Lie algebra admits conjugation and spectral flow automorphisms, $\bgconj$ and $\bgsf{\ell}$, $\ell \in \ZZ$.
These preserve $\wun$ and are given by \cite{RidBos14}
\begin{equation} \label{eq:bgauts}
	\begin{aligned}
		\bgconj(\beta_n) &= \gamma_n, & \bgconj(\gamma_n) &= -\beta_n, &
		\bgconj(\no{\beta\gamma}_n) &= -\no{\beta\gamma}_n + \delta_{n,0} \wun, & \bgconj(T^{\bgsymb}_n) &= T^{\bgsymb}_n - n \no{\beta\gamma}_n, \\
		\bgsf{\ell}(\beta_n) &= \beta_{n-\ell}, & \bgsf{\ell}(\gamma_n) &= \gamma_{n+\ell}, &
		\bgsf{\ell}(\no{\beta\gamma}_n) &= \no{\beta\gamma}_n + \ell \delta_{n,0} \wun, & \bgsf{\ell}(T^{\bgsymb}_n) &= T^{\bgsymb}_n - \ell \no{\beta\gamma}_n - \tfrac{1}{2} \ell(\ell-1) \delta_{n,0} \wun.
	\end{aligned}
\end{equation}
The expected dihedral relation
\begin{equation} \label{eq:bgdihedral}
	\bgconj \bgsf{\ell} \bgconj^{-1} = \bgsf{-\ell}, \quad \ell \in \ZZ,
\end{equation}
is readily verified.
As in \cref{def:invfunctors}, these automorphisms lift to invertible endofunctors on the category of $\bgalg$-modules and thence to the category of $\bgvoa$-modules.
Again, conjugation preserves lower-boundedness while spectral flow need not.

Consider the $\bgvoa$-modules on which $\no{\beta\gamma}_0$ acts semisimply.
The irreducibles in this category are easy to classify.
\begin{proposition}[\cite{RidBos14}] \label{prop:bgclass}
	The irreducible $\bgvoa$-modules on which $\no{\beta\gamma}_0$ acts semisimply are exhausted, up to isomorphism, by the following list of mutually inequivalent modules:
	\begin{itemize}
		\item The spectral flows of the \hw\ Verma module $\bgver$.
		Its \hwv\ has $\no{\beta\gamma}_0$-eigenvalue $0$ and conformal weight $0$. ($\bgver$ is the vacuum module.)
		\item The spectral flows of the \rhw\ Verma modules $\bgrel{\mu}$, with $[\mu] \in \CC/\ZZ \setminus \set[\big]{[0]}$.
		Their top spaces are spanned by vectors $w_{\lambda}$, $\lambda \in [\mu]$, of $\no{\beta\gamma}_0$-eigenvalue $\lambda$ and conformal weight $0$.
	\end{itemize}
\end{proposition}

The conjugation functor acts on the irreducibles of \cref{prop:bgclass} by
\begin{equation}
	\bgconj(\bgver) \cong \bgsf{-1}(\bgver) \quad \text{and} \quad \bgconj(\bgrel{\mu}) \cong \bgrel{-\mu}
\end{equation}
and \eqref{eq:bgdihedral}.
We mention that the ``missing'' coset $[\mu] = [0]$ in the relaxed classification corresponds to reducible \rhw\ $\bgvoa$-modules.
This includes the direct sum $\bgver \oplus \bgconj(\bgver)$, but there also exist nonsplit extensions involving $\bgver$ and $\bgconj(\bgver)$ (both extension groups are one-dimensional).
For simplicity, we define $\bgrel{0} = \bgver \oplus \bgconj(\bgver)$.

We define the character of a $\bgvoa$-module $\Mod{M}$ to be
\begin{equation}
	\fch{\Mod{M}}{\yy;\qq} = \traceover{\Mod{M}} \yy^{\no{\beta\gamma}_0} \qq^{T^{\bgsymb}_0 - \ccbg/24}.
\end{equation}
The character of the vacuum module is thus
\begin{equation}
	\fch{\bgver}{\yy;\qq} = \frac{\qq^{-1/12}}{\prod_{i=1}^{\infty} (1-\yy\qq^i)(1-\yy^{-1}\qq^{i-1})} = \yy^{1/2} \frac{\eta(\qq)}{\ii\fjth{1}{\yy;\qq}}.
\end{equation}
Note that this character formula has obvious poles (as a function of $\yy$).
Substituting the well known modular S-transforms of $\eta$ and $\jth{1}$ will therefore only give the S-transform of the meromorphic extension of the character to $\yy \in \CC$ (as per \cref{rem:modularconvergence}).

As per the general formalism of \cite{CreLog13,RidVer14}, the remedy is to instead consider the characters of the \rhw\ $\bgvoa$-modules (and their spectral flows):
\begin{equation} \label{eq:ch_ghosts_no_sf}
	\fch{\bgrel{\mu}}{\yy;\qq} = \frac{\qq^{-1/12}}{\prod_{i=1}^{\infty} (1-\yy\qq^i)(1-\yy^{-1}\qq^i)} \sum_{n \in \ZZ} \yy^{\mu+n} = \yy^{\mu} \frac{\delta(\yy)}{\eta(\qq)^2}, \quad [\mu] \in \CC/\ZZ.
\end{equation}
Here, $\delta(\yy) = \sum_{n \in \ZZ} \yy^n$ is the ``delta function'' of formal power series.
\begin{proposition}[\protect{\cite[Thm.~2 and Cor.~3]{RidBos14}}] \label{prop:bgmod}
	\leavevmode
	\begin{itemize}
		\item The characters of the spectral flows of the $\bgrel{\mu}$, $[\mu] \in \CC/\ZZ$, form a topological basis for the space of characters of $\bgvoa$-modules on which $\no{\beta\gamma}_0$ acts semisimply.
		\item Writing $\yy = \ee^{2\pi\ii\theta}$ and $\qq=\ee^{2\pi\ii\tau}$, the S-transform of the character of $\bgsf{\ell}(\bgrel{\mu}) \equiv \bgrel{\mu}^{\ell}$, $\ell \in \ZZ$ and $[\mu] \in \RR/\ZZ$, is given, up to an omitted automorphy factor, by
		\begin{equation} \label{eq:bgrelSmatrix}
			\fch{\bgrel{\mu}^{\ell}}{\tfrac{\theta}{\tau};-\tfrac{1}{\tau}}
			= \sum_{\ell' \in \ZZ} \int_{\RR/\ZZ} \bgSmat{(\ell,\mu),(\ell',\mu')} \fch{\bgrel{\mu'}^{\ell'}}{\theta;\tau} \, \dd[\mu'], \quad
			\bgSmat{(\ell,\mu),(\ell',\mu')} = (-1)^{\ell+\ell'} \ee^{-2\pi\ii(\ell\mu'+\ell'\mu)}.
		\end{equation}
		\item The S-transform of the character of $\bgsf{\ell}(\bgver) \equiv \bgver^{\ell}$, $\ell \in \ZZ$, is given, up to an omitted automorphy factor, by
		\begin{equation} \label{eq:bgvacSmatrix}
			\fch{\bgver^{\ell}}{\tfrac{\theta}{\tau};-\tfrac{1}{\tau}}
			= \sum_{\ell' \in \ZZ} \int_{\RR/\ZZ} \bgSmat{(\ell,\bullet),(\ell',\mu')} \fch{\bgrel{\mu'}^{\ell'}}{\theta;\tau} \, \dd[\mu'], \quad
			\bgSmat{(\ell,\bullet),(\ell',\mu')} = (-1)^{\ell+\ell'+1} \frac{\ee^{-2\pi\ii(\ell+1/2)\mu'}}{2\ii\sin(\pi\mu')}.
		\end{equation}
	\end{itemize}
\end{proposition}
\noindent We emphasise that in the above, the modular group acts on characters of modules in the full subcategory of $\bgvoa$-modules on which $\no{\beta\gamma}_0$ acts semisimply \emph{with real eigenvalues}.

\begin{remark}
	This restriction to real eigenvalues seems to be universal in studies of the modularity of non-$C_2$-cofinite \voas, see for example \cite{CreLog13}. It is moreover consistent with physical applications in which these eigenvalues are identified with quantum observables such as momentum. We are nevertheless not aware of any reason why one cannot have a modular group action when complex eigenvalues are allowed.
\end{remark}

\subsection{The FMS bosonisation} \label{sec:fms}

Consider the abelian Lie algebra $\pi = \spn_{\CC} \set{a,d}$ equipped with the 
symmetric bilinear form defined by
\begin{equation}
	\varbilin{a}{a} = 0 = \varbilin{d}{d} \quad \text{and} \quad \varbilin{a}{d} = 2.
\end{equation}
The associated rank-$2$ Heisenberg \va\ $\heis$ is strongly and freely generated by fields $a(z)$ and $d(z)$ satisfying
\begin{equation} \label{ope:cd}
	a(z)a(w) \sim 0 \sim d(z)d(w), \quad a(z)d(w) \sim \frac{2\,\wun}{(z-w)^2}.
\end{equation}

The group algebra $\CC[\ZZ a] = \spn_{\CC} \set{\ee^{ma} \st m \in \ZZ}$ is a $\pi$-module with action
\begin{equation}
	h \ee^{ma} = \varbilin{h}{ma}\ee^{ma}, \quad h \in \pi.
\end{equation}
The corresponding lattice \va\ extension of $\heis$ will be denoted by $\lvoa$.
It is strongly generated by $a(z)$, $d(z)$ and the $\ee^{ma}(z)$ with $m \in \ZZ$ (actually $m=\pm1$ will do).
The defining \opes\ are \eqref{ope:cd} and
\begin{equation}
	a(z)\ee^{m'a}(w) \sim 0 \sim \ee^{ma}(z)\ee^{m'a}(w), \quad d(z)\ee^{m'a}(w) \sim \frac{2m'\,\ee^{m'a}(w)}{z-w}, \qquad m,m' \in \ZZ.
\end{equation}

An old result of Friedan, Martinec and Shenker is that the bosonic ghosts \va\ $\bgvoa$ embeds into $\lvoa$.
\begin{proposition}[\cite{FriCon86}] \label{prop:FMSembed}
	The following map defines an embedding $\bgvoa \ira \lvoa$ of vertex algebras:
	\begin{equation}
		\beta \mapsto \ee^{a}, \quad \gamma \mapsto \tfrac{1}{2} \no{(a+d)\ee^{-a}}(z)
	\end{equation}
\end{proposition}
\noindent This embedding is commonly known as the FMS bosonisation of $\bgvoa$.

The \va\ $\lvoa$ admits a two-parameter family of \emts.
Because of the application reviewed in \cref{sec:iqhr}, we choose the \emt
\begin{equation}
	T^{\lsymb}(z) = \tfrac{1}{2}\no{ad}(z) + \tfrac{\kk}{3} \partial a(z) - \tfrac{1}{2} \partial d(z),
\end{equation}
so that the central charge is
\begin{equation}
	\cclvoa = 8\kk+2.
\end{equation}
The generators $a$, $d$ and $\ee^{ma}$, $m\in\ZZ$, then have conformal weights $1$, $1$ and $m$, respectively.

The representation theory of $\lvoa$ was studied in \cite{BerRep01}.
The irreducible $\lvoa$-modules, on which $a_0$ and $d_0$ act semisimply, are labelled by cosets $[\nu] \in \CC/\ZZ$ and integers $n\in\ZZ$.
They may be realised as direct sums of Fock spaces for $\heis$, generated by the $\ee^{\lambda a + nd/2}$, where $\lambda \in [\nu]$ and $n$ is fixed.
Such an irreducible is lower bounded if and only if $n=-1$, in which case it is a \rhw\ Verma module.
We shall denote it by
\begin{equation}
	\lmod{\nu} = \lvoa \, \ee^{\nu a - d/2}, \quad [\nu] \in \CC/\ZZ.
\end{equation}
The action of the zero modes of the generators on the top space of $\lmod{\nu}$ is explicitly given by
\begin{equation}
	a_{0}\ee^{\lambda a - d/2} = -\ee^{\lambda a - d/2}, \quad
	d_{0}\ee^{\lambda a - d/2} = 2\lambda\,\ee^{\lambda a - d/2}, \quad
	\ee^{ma}_0\ee^{\lambda a - d/2} = \ee^{(\lambda+m) a - d/2}, \qquad \lambda \in [\nu],\ m\in\ZZ.
\end{equation}
Moreover, this top space has conformal weight $\frac{\kk}{3}$.

As the vertex algebra $\lvoa$ contains the rank-$2$ Heisenberg \va\ $\heis$, it admits a $2$-parameter family of spectral flow automorphisms $\lsf{\xi}$, where
$\xi=\frac{1}{2}(\xi^a a+\xi^d d)\in\heis$.
These spectral flows are explicitly given by
\begin{equation} \label{eq:fmssf}
	\begin{gathered}
		\lsf{\xi}(a_n)=a_n-\xi^d\delta_{n,0}\wun,\quad
		\lsf{\xi}(d_n)=d_n-\xi^a\delta_{n,0}\wun,\quad
		\lsf{\xi}(\ee^{ma}_n)=\ee^{ma}_{n-m\xi^d}, \\
		\lsf{\xi}(T^{\lsymb}_n)=T^{\lsymb}_n-\xi_n+\brac[\big]{\tfrac{1}{2}\xi^a\xi^d-\tfrac{1}{2}\xi^a+\tfrac{\kk}{3}\xi^d}\delta_{n,0}\wun.
	\end{gathered}
\end{equation}
Such automorphisms only preserve the moding (hence the untwisted module sector) if $\xi^d\in\ZZ$.
They moreover only preserve the property of being lower bounded if $\xi^d=0$, the effect of $\xi^a$ on $\lmod{\nu}$ being to simply shift the weight coset $[\nu] \in \CC/\ZZ$.
More precisely, we have
\begin{equation} \label{eq:lsfaction}
	\lsf{\xi}(\lmod{\nu})\simeq \lmod{\nu+\xi^a/2}^{\xi^d}, \quad \xi^a \in \CC,\ \xi^d \in \ZZ,\ [\nu] \in \CC/\ZZ,
\end{equation}
where $\lmod{\nu}^{\ell}$ is the $\lvoa$-module generated by $\ee^{\nu a + (\ell-1)d/2}$.

We define the character of a $\lvoa$-module $\Mod{M}$ to be
\begin{equation}
	\fch{\Mod{M}}{\zz_a,\zz_d;\qq} = \traceover{\Mod{M}} \zz_a^{a_0}\zz_d^{d_0} \qq^{T^{\lsymb}_0 - \cclvoa/24}.
\end{equation}
Hence, for $[\nu] \in \CC/\ZZ$, the character of $\lmod{\nu}$ is given by
\begin{equation} \label{eq:ch_half_lattice_no_sf}
	\fch{\lmod{\nu}}{\zz_a,\zz_d;\qq} = \frac{\zz_a^{-1}\qq^{\kk/3-\cclvoa/24}}{\prod_{i=1}^{\infty}(1-\qq^i)^2}\sum_{n\in\ZZ}\zz_d^{2(n+\nu)} = \zz_a^{-1}\zz_d^{2\nu}\frac{\delta(\zz_d^2)}{\eta(\qq)^2}, \quad [\nu] \in \CC/\ZZ,
\end{equation}
and that of the spectrally flowed module $\lmod{\nu}^{\ell}$ is
\begin{equation}
	\fch{\lmod{\nu}^\ell}{\zz_a,\zz_d;\qq} =\zz_a^{\ell} \qq^{-\kk\ell/3} \fch{\lmod{\nu}}{\zz_a,\zz_d \qq^{\ell/2};\qq}, \quad [\nu] \in \CC/\ZZ,\ \ell \in \ZZ.
\end{equation}

To describe the modular S-transforms of these characters, we write $\zz_a = \ee^{2\pi\ii\zeta_a}$, $\zz_d = \ee^{2\pi\ii\zeta_d}$ and $\qq = \ee^{2\pi\ii\tau}$, so that
\begin{equation}\label{eq:ch_half_lattice}
	\fch{\lmod{\nu}^\ell}{\zeta_a,\zeta_d;\tau}
	= \frac{\ee^{2\pi\ii\zeta_a(\ell-1)} \ee^{-2\pi\ii\kk\ell\tau/3}}{\eta(\tau)^2} \sum_{n\in\ZZ} \ee^{2\pi\ii n\nu} \delta(2\zeta_d+\ell\tau-n).
\end{equation}
Here, $\delta(2\zeta_d+\ell\tau-n)$ refers to the usual Dirac delta function.
The relation to the power series delta function used in \eqref{eq:ch_ghosts_no_sf} and \eqref{eq:ch_half_lattice_no_sf} is that
\begin{equation}
	\delta(\zz_d^2) = \sum_{n\in\ZZ} \zz_d^{2n} = \sum_{n\in\ZZ} \ee^{2\pi\ii(2\zeta_d)n} = \sum_{n\in\ZZ} \delta(2\zeta_d-n).
\end{equation}
We trust that the abuse of notation here will not cause confusion in what follows.
\begin{proposition} \label{prop:lmod}
	The character \eqref{eq:ch_half_lattice} of the irreducible $\lvoa$-module $\lmod{\nu}^{\ell}$, $\ell \in \ZZ$ and $[\nu] \in \RR/\ZZ$, has S-transform
	\begin{subequations}
		\begin{equation}
			\fch{\lmod{\nu}^\ell}{\tfrac{\zeta_a}{\tau}, \tfrac{\zeta_d}{\tau}; -\tfrac{1}{\tau}}
			= A(\zeta_a, \zeta_d; \tau) \sum_{\ell'\in\ZZ} \int_{\RR/\ZZ} \piSmat{(\ell,\nu),(\ell',\nu')} \fch{\lmod{\nu'}^{\ell'}}{\zeta_a, \zeta_d; \tau}\, \dd[\nu'],
		\end{equation}
		where
		\begin{equation} \label{eq:lSmatrix}
			\piSmat{(\ell,\nu),(\ell',\nu')} = \ee^{2\pi\ii(\ell+\ell')\kk/3} \ee^{-2\pi\ii(\ell\nu'+\ell'\nu)}
		\end{equation}
		and the automorphy factor is
		\begin{equation}
			A(\zeta_a, \zeta_d; \tau) = \frac{\abs{\tau}}{-\ii\tau} \ee^{4\pi\ii\zeta_a\zeta_d/\tau} \ee^{2\pi\ii\zeta_a(\tau-1)/\tau} \ee^{-4\pi\ii\kk\zeta_d(\tau-1)/3\tau}.
		\end{equation}
	\end{subequations}
\end{proposition}
\noindent This result may be verified by direct computation. We omit the details.

\subsection{The inverse reduction embedding} \label{sec:iqhr}

The existence of an ``inverse \qhr'' embedding relating $\ubp$ and $\usl$ was first deduced in \cite{AdaRel21}, see also \cite{FehInv23} for a generalisation to $\SLA{sl}{n}$.
This may be proven by combining the Wakimoto free-field realisation with FMS bosonisation, as in \cite{SemInv94,FehSub21}, or by using the combinatorics of \pbw\ bases, as in \cite{AdaRea20}.
Either way, such an embedding is not uniquely determined.
We give an explicit formula for the action of one such embedding on the strong generators of $\usl$.
For this, we let
\begin{equation}
	i^1 = \frac{2\kk}{3}a + \frac{1}{2}d \quad \text{and} \quad i^2 = -\frac{\kk}{3}a + \frac{1}{2}d.
\end{equation}
\begin{theorem} \label{thm:iqhr}
	\leavevmode
	\begin{itemize}
		\item There exists an embedding $\usl \ira \ubp \otimes \bgvoa \otimes \lvoa$, for all $\kk\ne-3$, specified by
		\begin{equation} \label{eq:embedding}
			\begin{gathered}
				e^{1} \mapsto \gamma\,\ee^{a}, \qquad e^2 \mapsto \beta, \qquad e^{3} \mapsto \ee^{a}, \qquad
				h^{1} \mapsto -J- \no{\beta \gamma} + i^1, \qquad
				h^{2} \mapsto 2J + 2\no{\beta \gamma} + i^2, \\
				f^{1} \mapsto G^{+}\ee^{-a} - J \beta\,\ee^{-a} + \beta \, \no{i^2 \ee^{-a}} + (\kk+1)\pd \beta\,\ee^{-a}, \qquad
				f^{2} \mapsto G^{-} - 2 J \gamma - \gamma i^2 - \no{\gamma \gamma \beta} - \kk \pd \gamma, \\
				\begin{aligned}
					f^3 \mapsto &\brac[\Big]{(\kk+3) L + \pd J - \no{JJ}} \ee^{-a} + \brac[\Big]{J \no{\beta\gamma} - G^+ \gamma - G^- \beta} \ee^{-a} \\
					&- (\kk+1) \no{\pd\beta\gamma} \ee^{-a} - (J + \no{\beta\gamma}) \no{i^2 \ee^{-a}} - \no{\brac[\Big]{(\kk+3) \pd i^2 + \no{i^2i^2}} \ee^{-a}}.
				\end{aligned}
			\end{gathered}
		\end{equation}
		\item \eqref{eq:embedding} is a conformal embedding, meaning that the Sugawara \emt\ \eqref{eq:emtsl3} satisfies
		\begin{equation} \label{eq:confembedding}
		    T(z) \mapsto L(z) + T^{\bgsymb}(z) + T^{\lsymb}(z).
		\end{equation}
		\item\cite[Thm.~5.2]{AdaRel21} \eqref{eq:embedding} also defines an embedding $\slminmod[\uu,\vv] \ira \bpminmod[\uu,\vv] \otimes \bgvoa \otimes \lvoa$ if and only if $\vv \ne 1$.
	\end{itemize}
\end{theorem}

An obvious consequence of these embeddings is that one obtains modules for $\usl$ or $\slminmod[\uu,\vv]$ by restriction.
Specialising to $\vv=2$, we have $\slminmod$-modules such as
\begin{equation}\label{eq:somesl3mods}
	\slrel{\lambda}{\mu,\nu} = \bpirr{\lambda} \otimes \bgrel{\mu} \otimes \lmod{\nu} \quad \text{and} \quad
	\slsem{\lambda}{\nu} = \bpirr{\lambda} \otimes \bgver \otimes \lmod{\nu}, \qquad \lambda \in \pwlat,\ [\mu], [\nu] \in \CC/\ZZ.
\end{equation}
Since \eqref{eq:embedding} is conformal, these are lower bounded $\slminmod$-modules.
We record the following facts.
\begin{theorem}[\protect{\cite[Thm.~7.1]{AdaRel21} and \cite[Prop.~5.2]{AdaRea20}}] \label{thm:almostirreducible}
	For $\uu \in \set{3,5,7,\dots}$:
	\begin{itemize}
		\item The $\slminmod$-modules of \eqref{eq:somesl3mods} are almost irreducible, even when $[\mu]=[0]$.
		\item An almost-irreducible module is irreducible if and only if its top space is irreducible as an $\slthree$-module.
	\end{itemize}
\end{theorem}
\noindent To these, we add the following easy observation.
\begin{proposition} \label{prop:almostirreducible}
	A module that is isomorphic to a submodule of an almost-irreducible module and a quotient of another almost-irreducible module is itself almost irreducible.
\end{proposition}
\noindent It follows that direct summands of almost-irreducible modules are almost irreducible, though general submodules and quotients of almost-irreducible modules need not be.
Counterexamples are easily found, for example by considering the almost-irreducible $\asltwo$-module with three composition factors described in \cite[\S4.5]{KawRel18}.

\section{Modularity of $\slthree$ minimal models} \label{sec:mod}

In this section, we study the fully relaxed $\slminmod$-modules constructed in \eqref{eq:somesl3mods} using inverse \qhr\ and determine the modular S-transforms of the (generalised) characters of their spectral flows.
We also deduce (co)resolutions for semirelaxed and \hw\ $\slminmod$-modules, which lead consequently to the corresponding modular results.
These are then used to predict some Grothendieck fusion rules, using the (conjectural) standard Verlinde formula of \cite{CreLog13,RidVer14}.
A strong consistency check of our results is that the fusion coefficients are found to be nonnegative integers.

\subsection{Degenerations and completeness} \label{sec:degen}

Consider the fully relaxed $\slminmod$-module $\slrel{\lambda}{\mu,\nu}$, for $\lambda \in \pwlat$ and $[\mu],[\nu] \in \CC/\ZZ$.
Let $v_{\lambda}$ denote the \hwv\ of $\bpirr{\lambda}$ and $w_s$, $s \in [\mu]$, denote the \rhwv\ of $\bgrel{\mu}$ whose $\no{\beta\gamma}_0$-eigenvalue is $s$.
Then, by \cref{prop:bphwtop}, a basis of the top space of $\slrel{\lambda}{\mu,\nu}$ is given by the weight vectors
\begin{equation} \label{eq:topspacefully}
	u_{r,s,t} = (G^{-}_{0})^{r} v_{\lambda} \otimes w_s \otimes \ee^{ta-d/2}, \quad r = 0,1,\dots,\lambda_2,\ s \in [\mu],\ t \in [\nu].
\end{equation}
Using \eqref{eq:embedding} and \eqref{eq:confembedding}, we find that the $\slthree$-weight and conformal weight of $u_{r,s,t}$ are
\begin{equation} \label{eq:uwts}
	(j_{\lambda}-r+s+\tfrac{\kk}{3}) \srt{2} + (t-\tfrac{\kk}{3}) \srt{3} \quad \text{and} \quad \bpconfwt{\lambda} + \tfrac{\kk}{3},
\end{equation}
respectively.
It follows that the top space of $\slrel{\lambda}{\mu,\nu}$ is a dense $\slthree$-module whose weights have multiplicity $\lambda_2+1$.
Moreover, if we fix $\lambda$ but allow $[\mu]$ and $[\nu]$ to vary, then the top spaces of the $\slrel{\lambda}{\mu,\nu}$ form a coherent family
\begin{equation}
	\Mod{C}_{\lambda} = \:\bigoplus_{\mathclap{[\mu],[\nu] \in \CC/\ZZ}}\: \operatorname{top}(\slrel{\lambda}{\mu,\nu}), \quad \lambda \in \pwlat,
\end{equation}
of $\slthree$-modules, as in \cref{sec:sl3cohfam}.
A natural question now is whether inverse reduction constructs a complete set of irreducible coherent families of the Zhu algebra of $\slminmod$ as top spaces.
If so, then every irreducible fully relaxed $\slminmod$-module is isomorphic to one of the $\slrel{\lambda}{\mu,\nu}$.

It was shown in \cite[\S4.3]{KawAdm21} that the Zhu algebra of $\slminmod$ has $\abs*{\pwlat} = \frac{1}{2}(\uu-1)(\uu-2)$ irreducible coherent families.
This matches the number of families constructed above, but we do not yet know if our families are irreducible and distinct.
To show that they are, we look for (twisted) \hw\ submodules of the $\Mod{C}_{\lambda}$.
Irreducibility will follow by finding a submodule whose maximal multiplicity matches that of $\Mod{C}_{\lambda}$ and distinctness will follow by demonstrating that the highest weights of the found submodules lie in distinct shifted Weyl orbits \eqref{eq:boundedequivclass} of the nonintegral weights of $\admwts$ \cite{MatCla00}.

It is easy to get started.
We simply determine parameters $[\mu],[\nu] \in \CC/\ZZ$ such that the fully relaxed $\slminmod$-module $\slrel{\lambda}{\mu,\nu}$ degenerates into ($\grp{D}_6$-twisted) semirelaxed modules and then into ($\grp{D}_6$-twisted) \hwms.
Recall that in \cref{sec:bghosts}, we defined $\bgrel{0}$ to be $\bgver \oplus \bgconj(\bgver)$.
Referring back to \eqref{eq:somesl3mods}, it follows that
\begin{equation} \label{eq:R=S+M}
	\slrel{\lambda}{0,\nu} \cong \slsem{\lambda}{\nu} \oplus \slver{}, \quad \text{where} \quad \slver{} = \bpirr{\lambda} \otimes \bgconj(\bgver) \otimes \lmod{\nu}.
\end{equation}
We will therefore investigate the (twisted) \hw\ submodules of $\slsem{\lambda}{\nu}$, when the latter is reducible.

For this, note first that a basis for the top space of the $\slminmod$-module $\slsem{\lambda}{\nu}$ is obtained from \eqref{eq:topspacefully} by restricting to $[\mu]=[0]$ and $s\in\ZZ_{\le0}$.
The coefficients of $\srt{2}$ in the $\slthree$-weights \eqref{eq:uwts} of the top space therefore have a maximal value, namely $j_{\lambda}+\frac{\kk}{3}$, while those of $\srt{3}$ are unbounded above and below.
The multiplicities of the top space weights of $\slsem{\lambda}{\nu}$ are thus constant in the $\srt{3}$-direction and increase linearly in the $-\srt{2}$-direction from $1$, when the coefficient of $\srt{2}$ is maximal, until they saturate at $\lambda_2+1$, see \cref{fig:sl3mults}.

\begin{figure}
	\begin{tikzpicture}[scale=0.7,<->,>=latex,baseline=(current bounding box.center)]
		\begin{scope}[shift={(1,0)}]
			\foreach \t in {-2,...,1} \node[rotate=60*\t] at (60*\t:5) {$\cdots$};
			\foreach \t in {-2,...,0} \node[rotate=60*\t+30] at (60*\t+30:4.5) {$\cdots$};
			\begin{scope}[shift={(120:2)}]
				\node[rotate=60] at (60:3) {$\cdots$};
				\node[rotate=60] at (60:-5) {$\cdots$};
			\end{scope}
			\begin{scope}[shift={(120:4)}]
				\node[rotate=60] at (60:1) {$\cdots$};
				\node[rotate=60] at (60:-5) {$\cdots$};
			\end{scope}
		\end{scope}
		\foreach \n in {-3,...,1} \node[scale=0.7,shift={(120:3)}] at (60:\n) {$1$};
		\foreach \n in {-3,...,2} \node[scale=0.7,shift={(120:2)}] at (60:\n) {$2$};
		\foreach \n in {-3,...,3} \node[scale=0.7,shift={(120:1)}] at (60:\n) {$3$};
		\foreach \t in {-2,...,2} \node[scale=0.7,rotate=-30+15*\t] at (60:7/4*\t+1/2) {$\cdots$};
		\foreach \n in {-3,...,5} \node[scale=0.7,shift={(120:-1)}] at (60:\n) {$\lambda_2$};
		\foreach \m in {-5,...,-2} {
			\pgfmathsetmacro{\l}{-\m-4}
			\foreach \n in {\l,...,5} \node[scale=0.7,shift={(120:\m)}] at (60:\n) {$\lambda_2+1$};
		}
		\begin{scope}[shift={(7,3)}]
			\draw[blue,->] (0:0) -- (120:1) node[left,black] {$\srt{2}$};
			\draw[blue,->] (0:0) -- (60:1) node[right,black] {$\srt{3}$};
		\end{scope}
		\begin{scope}[shift={(13,0)}]
			\foreach \t in {0,...,5} \node[rotate=60*\t] at (60*\t:5) {$\cdots$};
			\foreach \t in {0,...,5} \node[rotate=60*\t+30] at (60*\t+30:4.5) {$\cdots$};
			\foreach \n in {-4,...,4} \node[scale=0.7] at (\n,0) {$\lambda_2+1$};
			\foreach \m in {1,...,4} {
				\pgfmathsetmacro{\l}{\m-4}
				\foreach \n in {\l,...,4} \node[scale=0.7,shift={(120:\m)}] at (\n,0) {$\lambda_2+1$};
				\foreach \n in {\l,...,4} \node[scale=0.7,shift={(-120:\m)}] at (\n,0) {$\lambda_2+1$};
			}
		\end{scope}
	\end{tikzpicture}
	\caption{%
		Multiplicities of the weights of the top space of $\slsem{\lambda}{\nu}$ (left) and $\slrel{\lambda}{\mu,\nu}$ (right).
		The weights of the latter coincide with a translate of the root lattice $\rlat$ in $\csub^*$.
		The weights of the former effectively fill out half of such a translate.
	} \label{fig:sl3mults}
\end{figure}

Consider the subspace of vectors $u_{0,0,t}$ in the top space of $\slsem{\lambda}{\nu}$.
It consists of $1$-dimensional weight spaces that are annihilated by both $e^2_0$ and $f^1_0$.
Since $e^3_0$ always acts injectively, by \eqref{eq:embedding}, we search for vectors in this subspace that are also annihilated by $f^3_0$.
Such vectors are then twisted \hwvs.
Computing with \eqref{eq:embedding}, the action of $f^3_0$ on $u_{0,0,t}$ is found to be a multiple of $u_{0,0,t-1}$ that is quadratic in $t$.
This multiple vanishes for
\begin{equation} \label{eq:semizeroes}
	t = t^1_{\lambda} = -\tfrac{1}{3} \brac*{\lambda_{1}+2\lambda_{2}} - 1 + \tfrac{\uu}{3}
	\quad \text{and} \quad
	t = t^2_{\lambda} = \tfrac{1}{3} \brac*{2\lambda_{1}+\lambda_2} + 1 - \tfrac{\uu}{6}.
\end{equation}
The vector $u_{0,0,t^i_{\lambda}}$ is thus a twisted \hwv, see \cref{fig:twhwvs}, and \eqref{eq:uwts} gives its $\slthree$-weight as
\begin{equation} \label{eq:somesl3wts}
	\brac*{1 - \lambda_2}\fwt{1} + \brac*{\tfrac{\uu}{2} - 2 - \lambda_1}\fwt{2} \quad \text{and} \quad
	\brac*{\lambda_1 + 3 - \tfrac{\uu}{2}}\fwt{1} + \lambda_2 \fwt{2},
\end{equation}
for $i=1$ and $2$, respectively.
We remark that $t^1_{\lambda}-t^2_{\lambda} = \frac{\uu}{2} \ne 0 \bmod{\ZZ}$, so these twisted \hwvs\ belong to distinct semirelaxed modules.

\begin{figure}
	\begin{tikzpicture}[scale=0.9,->,>=latex,thick]
		\fill[fillcolourA] (0:0) -- (60:3) -- ($(60:3)+(0:3)$) -- ($(-60:3)$);
		\node[wt,label=above left:{$u_{0,0,t^i_{\lambda}}$}] (t) at (0:0) {};
		\draw (t) -- (60:3);
		\draw (t) -- (-60:3);
		\begin{scope}[shift={(-120:1)}]
			\fill[fillcolourC] (0:0) -- (0:3) -- ($(-120:3)+(0:3)$) -- ($(-120:3)$);
			\node[wt,label=left:{$u_{0,0,t^i_{\lambda}-1}$}] (t-1) at (0:0) {};
			\draw (t-1) -- (0:3);
			\draw (t-1) -- (-120:3);
		\end{scope}
		\draw[red] (t-1) to[out=90,in=-150] node[left] {$e^3_0$} (t);
		\begin{scope}[shift={(-2,2)}]
			\draw[blue] (0:0) -- (0:1) node[right,black] {$\srt{1}$};
			\draw[blue] (0:0) -- (120:1) node[left,black] {$\srt{2}$};
			\draw[blue] (0:0) -- (60:1) node[right,black] {$\srt{3}$};
		\end{scope}
	\end{tikzpicture}
	\caption{%
		The top space vector $u_{0,0,t^i_{\lambda}}$ generating a $\wref{1}\wref{2}$-twisted \hw\ submodule (blue) of $\slsem{\lambda}{t^i_{\lambda}}$, for $\lambda \in \pwlat$, $i=1,2$ and $t^i_{\lambda}$ given by \eqref{eq:semizeroes}.
		Also shown is the vector $u_{0,0,t^i_{\lambda}-1}$ whose image in the quotient generates a $\wref{1}$-twisted \hw\ submodule (green).
	} \label{fig:twhwvs}
\end{figure}

Examining \cref{fig:twhwvs}, we see that $u_{0,0,t^i_{\lambda}}$ generates a twisted \hwm\ that is obtained from an untwisted one by applying $\wref{1}\wref{2}$.
The highest weights of these untwisted \hwms\ are thus obtained by applying $\wref{2}\wref{1}$ to the weights \eqref{eq:somesl3wts}.
Lifting to $\aslthree$-weights, the resulting highest weights are
\begin{equation} \label{eq:defLambda}
	\begin{gathered}
		\Lambda^1_{\lambda}
		= \brac*{\tfrac{\uu}{2} - 2 - \lambda_2} \fwt{0} + \brac*{\tfrac{\uu}{2} - 2 - \lambda_{1}} \fwt{1} + \brac*{\tfrac{\uu}{2} - 2 - \lambda_{0}} \fwt{2}
		= \wref{1} \cdot \brac*{\lambda - \tfrac{\uu}{2}\fwt{1}} \\
		\text{and} \quad \Lambda^2_{\lambda} = \lambda_1 \fwt{0} + \lambda_2 \fwt{1} + \brac*{\lambda_0 - \tfrac{\uu}{2}} \fwt{2} = \nabla(\lambda) - \tfrac{\uu}{2}\fwt{2},
	\end{gathered}
\end{equation}
respectively, confirming that these highest weights lie in $\admwts$.
By \cref{thm:sl3hwclass}, the corresponding \hwms\ are irreducible and, hence, so are the submodules $\wref{1}\wref{2}(\slirr{\Lambda^i_{\lambda}})$ of $\slsem{\lambda}{t^i_{\lambda}}$, $i=1,2$.

We have therefore found an irreducible twisted \hw\ submodule $\wref{1}\wref{2}(\slirr{\Lambda^2_{\lambda}}) = \wref{1}\wref{2}(\slirr{\nabla(\lambda)-\uu\fwt{2}/2})$ of $\slsem{\lambda}{t^i_{\lambda}}$, hence of $\slrel{\lambda}{0,t^i_{\lambda}}$.
Its top space is thus an irreducible twisted \hw\ submodule of the coherent family $\Mod{C}_{\lambda}$.
By \cref{lem:ihwbounds}, the maximal multiplicity of this submodule is $\lambda_2+1$, matching that of $\Mod{C}_{\lambda}$.
It follows that the latter is an irreducible coherent family.
Moreover, the highest weights $\Lambda^2_{\lambda} = \nabla(\lambda)-\frac{\uu}{2}\fwt{2}$ with different $\lambda \in \pwlat$ lie in different equivalence classes \eqref{eq:boundedequivclass}, hence the $\Mod{C}_{\lambda}$ are distinct.
This proves the first part of the following completeness result.
\begin{theorem} \label{thm:completeness}
	Let $u \in \set{3,5,7,\dots}$.
	Then:
	\begin{itemize}
		\item Every irreducible fully relaxed $\slminmod$-module with finite multiplicities is isomorphic to one of the $\slrel{\lambda}{\mu,\nu}$ with $\lambda \in \pwlat$ and $[\mu],[\nu] \in \CC/\ZZ$.
		\item Every irreducible semirelaxed $\slminmod$-module with finite multiplicities is isomorphic to a $\grp{D}_6$-twist of one of the $\slsem{\lambda}{\nu}$ with $\lambda \in \pwlat$ and $[\nu] \in \CC/\ZZ$.
	\end{itemize}
\end{theorem}
\noindent For the second part, we need only recall that every irreducible semirelaxed $\slminmod$-module (with finite multiplicities) may be realised as a submodule of a reducible fully relaxed $\slminmod$-module (with finite multiplicities) \cite[\S4.2]{KawAdm21}.
(This is a peculiarity of the case $\vv=2$; when $\vv>2$, there are many more irreducible semirelaxed $\slminmod[\uu,\vv]$-modules than irreducible fully relaxed $\slminmod$-modules.)

We finish by noting that it is possible to prove \cref{thm:completeness} without recourse to the independent classification results of \cite{KawRel19,KawAdm21}.
But, we shall not do so here, referring only to \cite[\S3.3]{AdaWei23} for an illustration of the method.

\subsection{Fully relaxed $\slminmod$-modules} \label{sec:modfull}

We next turn to the modular properties of the characters of the fully relaxed modules $\slrel{\lambda}{\mu,\nu}$ and their spectral flows (which are not lower bounded in general).
It turns out that the spectral flow action on the $\slrel{\lambda}{\mu,\nu}$ can be expressed in terms of spectral flow actions on the component $\bpminmod$-, $\bgvoa$- and $\lvoa$-modules using the inverse \qhr\ embedding of \cref{thm:iqhr}.
Indeed, comparing \cref{eq:slsf,eq:bpsf,eq:bgauts,eq:fmssf} with the explicit embedding formulae in \eqref{eq:embedding}, we conclude that
\begin{equation} \label{eq:factoredsf}
	\slsf{g} = \bpsf{\bilin{\srt{2}}{g}} \otimes \bgsf{\bilin{\srt{2}}{g}} \otimes \lsf{\kk\bilin{\srt{1}}{g}a/3 + \bilin{\srt{3}}{g}d/2}, \quad g \in \cwlat.
\end{equation}
\cref{prop:bpconjsf,eq:lsfaction} now imply the following identification.
\begin{lemma} \label{lem:sf_factorised}
	For $\uu \in \set{3,5,7,\dots}$, $g \in \cwlat$, $\lambda \in \pwlat$ and $[\mu], [\nu] \in \CC/\ZZ$,
	\begin{equation} \label{eq:spectral_flow_relaxed}
		\slsf{g}\brac[\big]{\slrel{\lambda}{\mu,\nu}}
		\simeq \bpirr{\nabla^{\bilin{\srt{2}}{g}}(\lambda)} \otimes \bgrel{\mu}^{\bilin{\srt{2}}{g}} \otimes \lmod{\nu+\bilin{\srt{1}}{g}\uu/6}^{\bilin{\srt{3}}{g}}.
	\end{equation}
\end{lemma}

The characters of the $\slsf{g}(\slrel{\lambda}{\mu,\nu})$ are therefore just the products of the characters of the corresponding $\bpminmod$-, $\bgvoa$- and $\lvoa$-modules, appropriately specialised according to the inverse reduction embedding \eqref{eq:embedding}.
Unfortunately, these characters are known \cite[Cor.~5.2]{KawAdm21} to be linearly dependent, unless $\uu=3$, and so cannot be used to deduce modularity properties.
But, inverse reduction provides us with natural generalised characters since the two independent Cartan elements $h^1_0, h^2_0 \in \aslthree$ are subsumed into the $4$-dimensional space spanned by $J_0$, $\no{\beta\gamma}_0$, $a_0$ and $d_0$.
We therefore define generalised $\slminmod$-characters for the $\slsf{g}(\slrel{\lambda}{\mu,\nu})$ as follows:
\begin{equation} \label{eq:defgenchars}
	\begin{split}
		\fch{\slsf{g}(\slrel{\lambda}{\mu,\nu})}{\zz,\yy,\zz_a,\zz_d;\qq}
		&= \traceover{\slsf{g}(\slrel{\lambda}{\mu,\nu})} \zz^{J_0} \yy^{\no{\beta\gamma}_0} \zz_a^{a_0} \zz_d^{d_0} \qq^{L_0+T^{\bgsymb}_0+T^{\lsymb}_0} \\
		&= \fch{\bpirr{\nabla^{\bilin{\srt{2}}{g}}(\lambda)}}{\zz;\qq} \fch{\bgrel{\mu}^{\bilin{\srt{2}}{g}}}{\yy;\qq} \fch{\lmod{\nu+\bilin{\srt{1}}{g}\uu/6}^{\bilin{\srt{3}}{g}}}{\zz_a,\zz_d;\qq}.
	\end{split}
\end{equation}
(We use the same notation as for the usual characters for convenience, trusting that this will cause no confusion.)

The linear independence of these generalised characters follows immediately from that of the $\bpminmod$-, $\bgvoa$- and $\lvoa$-characters.
We deduce the modular S-transforms using \cref{thm:bpmod,prop:bgmod,prop:lmod}.
As usual, we restrict to real parameters.
\begin{theorem} \label{thm:Sfully}
	Given $\uu \in \set{3,5,7,\dots}$, the modular $S$-transforms of the generalised characters of the $\slsf{g}(\slrel{\lambda}{\mu,\nu})$, with $g \in \cwlat$, $\lambda \in \pwlat$ and $[\mu], [\nu] \in \RR/\ZZ$, are given, up to an omitted automorphy factor, by
	\begin{subequations}
		\begin{gather}
			\begin{multlined}[t]
				\fch{\slsf{g}(\slrel{\lambda}{\mu,\nu})}{\tfrac{\zeta}{\tau},\tfrac{\theta}{\tau},\tfrac{\zeta_a}{\tau},\tfrac{\zeta_d}{\tau};-\tfrac{1}{\tau}} \\
				= \sum_{g'\in\cwlat} \sum_{\lambda'\in\pwlat} \int_{\RR/\ZZ} \int_{\RR/\ZZ} \slSmat{(g,\lambda,\mu,\nu),(g',\lambda',\mu',\nu')} \fch{\slsf{g'}(\slrel{\lambda'}{\mu',\nu'})}{\zeta,\theta,\zeta_a,\zeta_d;\tau} \, \dd [\mu'] \dd [\nu'],
			\end{multlined}
			\\
				\slSmat{(g,\lambda,\mu,\nu),(g',\lambda',\mu',\nu')}
				= \bpSmat{\nabla^{\bilin{\srt{2}}{g}}(\lambda),\nabla^{\bilin{\srt{2}}{g'}}(\lambda')} \bgSmat{(\bilin{\srt{2}}{g},\mu),(\bilin{\srt{2}}{g'},\mu')} \piSmat{(\bilin{\srt{3}}{g},\nu+\bilin{\srt{1}}{g}\uu/6),(\bilin{\srt{3}}{g'},\nu'+\bilin{\srt{1}}{g'}\uu/6)}. \label{eq:slSfully}
		\end{gather}
	\end{subequations}
\end{theorem}
\begin{proof}
	This follows by substituting \eqref{eq:defgenchars} and the known modular S-transforms
	\begin{equation}
		\begin{aligned}
			\fch{\bpirr{\nabla^{\bilin{\srt{2}}{g}}(\lambda)}}{\tfrac{\zeta}{\tau};-\tfrac{1}{\tau}}
			&= \sum_{\lambda'\in\pwlat} \bpSmat{\nabla^{\bilin{\srt{2}}{g}}(\lambda),\lambda'} \fch{\bpirr{\lambda'}}{\zeta;\tau}, \\
			\fch{\bgrel{\mu}^{\bilin{\srt{2}}{g}}}{\tfrac{\theta}{\tau};-\tfrac{1}{\tau}}
			&= \sum_{\ell'\in\ZZ} \int_{\RR/\ZZ} \bgSmat{(\bilin{\srt{2}}{g},\mu),(\ell',\mu')} \fch{\bgrel{\mu'}^{\ell'}}{\theta;\tau} \, \dd [\mu'], \\
			\fch{\lmod{\nu+\bilin{\srt{1}}{g}\uu/6}^{\bilin{\srt{3}}{g}}}{\tfrac{\zeta_a}{\tau},\tfrac{\zeta_d}{\tau};-\tfrac{1}{\tau}}
			&= \sum_{m'\in\ZZ} \int_{\RR/\ZZ} \piSmat{(\bilin{\srt{3}}{g},\nu+\bilin{\srt{1}}{g}\uu/6),(m',\nu')} \fch{\lmod{\nu'}^{m'}}{\zeta_a,\zeta_d;\tau} \, \dd [\mu'],
		\end{aligned}
	\end{equation}
	ignoring automorphy factors as usual.
	Defining $g' = (m'-\ell')\fcwt{1} + \ell'\fcwt{2} \in \cwlat$, so that $\ell' = \bilin{\srt{2}}{g'}$ and $m' = \bilin{\srt{3}}{g'}$, and then replacing $\lambda'$ by $\nabla^{\bilin{\srt{2}}{g'}}(\lambda')$ and $\nu'$ by $\nu'+\bilin{\srt{1}}{g'}\frac{\uu}{6}$, we arrive at the desired result.
\end{proof}

The fully relaxed $\slminmod$ S-matrix thus factorises into a \bp\ S-matrix and free-field ones, in line with the inverse reduction realisation of \eqref{eq:somesl3mods}.
However, this realisation depends upon the embedding \eqref{eq:embedding} and so is in no way canonical nor even unique.
It is thus unsurprising that the result is a little complicated.
We therefore introduce a more natural parametrisation that substitutes the free-field data $\mu$ and $\nu$ for an $\slthree$-weight:
\begin{equation} \label{eq:defgamma}
	\gamma = \gamma(\lambda,\mu,\nu) = (j_{\lambda}+\mu+\tfrac{\uu}{6}) \srt{2} + (\nu-\tfrac{\uu}{6}) \srt{3} \in \csub^*.
\end{equation}
When employing this natural parametrisation, we shall write $\slrel{\lambda}{\gamma}$ instead of $\slrel{\lambda}{\mu,\nu}$.
Note that the $\slthree$-weight $\gamma$ is a representative of the coset in $\csub^*/\rlat$ of the weights of $\slrel{\lambda}{\mu,\nu}$, by \eqref{eq:uwts}.
Defining $\gamma'$ similarly, the fully relaxed S-matrix now simplifies dramatically.
\begin{corollary} \label{cor:Sfully}
	Given $\uu \in \set{3,5,7,\dots}$, the modular $S$-transforms of the generalised characters of the $\slsf{g}(\slrel{\lambda}{\gamma})$, with $g \in \cwlat$, $\lambda \in \pwlat$ and $[\gamma] \in \csub_{\RR}^*/\rlat$, are given, up to an omitted automorphy factor, by
	\begin{subequations}
		\begin{gather}
			\fch{\slsf{g}(\slrel{\lambda}{\gamma})}{\tfrac{\zeta}{\tau},\tfrac{\theta}{\tau},\tfrac{\zeta_a}{\tau},\tfrac{\zeta_d}{\tau};-\tfrac{1}{\tau}}
				= \sum_{g'\in\cwlat} \sum_{\lambda'\in\pwlat} \int_{\csub_{\RR}^*/\rlat} \slSmat{(g,\lambda,\gamma),(g',\lambda',\gamma')} \fch{\slsf{g'}(\slrel{\lambda'}{\gamma'})}{\zeta,\theta,\zeta_a,\zeta_d;\tau} \, \dd [\gamma'], \\
			\slSmat{(g,\lambda,\gamma),(g',\lambda',\gamma')}
				= \ee^{-2\pi\ii\brac*{\bilin{g}{g'}\uu/2 + \bilin{\gamma}{g'} + \bilin{g}{\gamma'}}} \bpSmat{\lambda,\lambda'}, \label{eq:slSfully!}
		\end{gather}
	\end{subequations}
	where $\csub_{\RR}^*$ denotes the real span of the simple roots $\srt{1}$ and $\srt{2}$.
\end{corollary}
\begin{proof}
	Simplify the $\bpminmod$ S-matrix in \eqref{eq:slSfully} using \cref{cor:bpSsymm} twice and substitute \eqref{eq:bgrelSmatrix} and \eqref{eq:lSmatrix} for the $\bgvoa$ and $\lvoa$ S-matrices.
	Replacing $\mu+\nu$ by $\bilin{\gamma}{\fcwt{2}}-j_{\lambda}$ and then $\nu$ by $\bilin{\gamma}{\fcwt{1}}+\frac{\uu}{6}$, as per \eqref{eq:defgamma} (and similarly for their primed variants), everything now is just arithmetic.
\end{proof}
\noindent We remark that setting $\uu = 3$ in this simplification, hence $\lambda = \lambda' = 0$, recovers the result deduced in \cite[Thm.~5.4]{KawAdm21} using different methods.
Our generalisation to $\uu>3$ is new.

\subsection{Semirelaxed $\slminmod$-modules} \label{sec:modsemi}

As in the fully relaxed case, inverse reduction constructs a natural family of semirelaxed $\slminmod$-modules $\slsem{\lambda}{\nu}$, given in \eqref{eq:somesl3mods}, that includes every irreducible semirelaxed $\slminmod$-module, if one also allows for twisting by elements of $\grp{D}_6$.
We may thus define generalised characters for them, in exactly the same manner as for their fully relaxed cousins.

In principle, we already have everything we need to determine the modular S-transforms of the generalised characters of the $\slsem{\lambda}{\nu}$: we simply replace the relaxed ghost S-matrix $\bgSmat{(\bilin{\srt{2}}{g},\mu),(\bilin{\srt{2}}{g'},\mu')}$ in \eqref{eq:slSfully} by the vacuum ghost S-matrix $\bgSmat{(\bilin{\srt{2}}{g},\bullet),(\bilin{\srt{2}}{g'},\mu')}$ given in \cref{prop:bgmod}.
However, this approach does not capitalise on the simplicity found with the reparametrisation \eqref{eq:defgamma}.
Following \cite{CreRel11}, we therefore describe a different route to the semirelaxed modular S-transforms, that of (co)resolving the semirelaxed modules in terms of fully relaxed ones.
This route will also be used to deduce the modularity of the \hw\ $\slminmod$-modules in \cref{sec:modhw}.
In this case, we have no alternative as these modules cannot be directly constructed using inverse reduction, only as degenerations of direct constructions.

To construct (co)resolutions for the $\slsem{\lambda}{\nu}$, we first identify the direct summand $\slver{}$ appearing in the decomposition \eqref{eq:R=S+M} of the reducible fully relaxed modules $\slrel{\lambda}{0,\nu}$.
Recall that $\slrel{\lambda}{0,\nu}$ is almost irreducible (\cref{thm:almostirreducible}).
It follows from \cref{prop:almostirreducible} that $\slver{}$ is too, hence that it is completely determined by its top space \cite{DeSFin05}.
But, the top space of $\slrel{\lambda}{0,\nu}$ has constant multiplicities, so those of the top spaces of $\slsem{\lambda}{\nu}$ and $\slver{}$ are complementary.

Write the weights of the top space of $\slver{}$ as linear combinations of $\srt{2}$ and $\srt{3}$, as per \eqref{eq:uwts}.
Referring to the discussion about $\slsem{\lambda}{\nu}$ in \cref{sec:degen}, this complementarity means that the coefficients of $\srt{2}$ have a minimal value, namely $j_{\lambda}+\frac{\kk}{3} - \lambda_2+1$, while those of $\srt{3}$ are unbounded above and below.
It follows that the multiplicities are constant in the $\srt{3}$ direction and increase linearly in the $\srt{2}$-direction from $1$, when the coefficient of $\srt{2}$ is minimal, until they saturate at $\lambda_2+1$.
Applying $\dynk$, we conclude that the top space of $\dynk(\slver{})$ is that of an untwisted semirelaxed $\slminmod$-module.
By \cref{thm:completeness}, we thus have $\slver{} \cong \dynk(\slsem{\lambda'}{\nu'})$, for some $\lambda' \in \pwlat$ and $[\nu'] \in \CC/\ZZ$.

\begin{remark} \label{rem:fudging}
	Technically, we should here assume here that $\slver{}$ is irreducible in order to apply \cref{thm:completeness}.
	While this will be the case for almost all $[\nu] \in \CC/\ZZ$, it is relatively harmless to continue the analysis even when $\slver{}$ is reducible because $\slver{}$ will still have the same generalised character as $\dynk(\slsem{\lambda'}{\nu'})$, even when they are not strictly isomorphic.
	As our focus is the generalised characters and their modular S-transforms, we will continue to apply \cref{thm:completeness} even in the reducible case without further comment.
\end{remark}

Since the top spaces of $\slsem{\lambda}{\nu}$ and $\slver{}$ have the same multiplicity saturation bound, it follows that $\lambda'_2 = \lambda_2$.
To determine $\lambda'_1$ and $[\nu']$, we apply $\dynk$ to the weights of the top space of $\slver{}$ with minimal $\srt{2}$-coefficient and set the result equal to $(j_{\lambda'}+\frac{\kk}{3}) \srt{2} + (t'-\frac{\kk}{3}) \srt{3}$, for some $t' \in [\nu']$.
This gives $j_{\lambda'} = -j_{\lambda}-\tfrac{\uu}{3}+\lambda_2+1$, hence $\lambda'_1 = \lambda_0$, and $[\nu'] = [\nu + j_{\lambda} + \tfrac{\uu}{6}]$.
We therefore conclude that
\begin{equation}
	\slver{} \cong \dynk(\slsem{\dynk\nabla(\lambda)}{\nu+j_{\lambda}+\uu/6}).
\end{equation}
One can also check that the conformal weight of the top space of $\slsem{\dynk\nabla(\lambda)}{\nu+j_{\lambda}+\uu/6}$ matches that of $\slsem{\lambda}{\nu}$, by \eqref{eq:uwts}.

To convert this into a (co)resolution, it is unfortunately not enough to identify $\slver{}$ as a $\grp{D}_6$-twist of some semirelaxed module.
We need to instead realise it as a spectral flow of a semirelaxed module.
Happily, \cite[Prop.~4.8]{KawAdm21} determines a complete set of instances in which the spectral flow of a lower-bounded irreducible $\slminmod$-module remains lower bounded.
Taking into account a difference in conventions (their semirelaxed modules are obtained from ours by twisting by $\wref{2}\dynk$), the only (nontrivial) lower-bounded spectral flow of $\slver{}$ is obtained by acting with $\slsf{\fcwt{3}}$, where we recall that $\fcwt{3} = \fcwt{1}-\fcwt{2} = \frac{1}{3}(\scrt{1}-\scrt{2})$.
\begin{lemma} \label{lem:dS=sfS}
	Given $\uu \in \set{3,5,7,\dots}$, $\lambda \in \pwlat$ and $[\nu] \in \CC/\ZZ$, we have $\dynk(\slsem{\lambda}{\nu}) \cong \slsf{\fcwt{3}}(\slsem{\dynk(\lambda)}{\nu+j_{\lambda}})$.
\end{lemma}
\begin{proof}
	Because $\slsf{\fcwt{3}}(\slsem{\lambda'}{\nu'})$ is a $\dynk$-twisted semirelaxed $\slminmod$-module \cite[Prop.~4.8]{KawAdm21}, we can identify $\lambda' \in \pwlat$ and $[\nu'] \in \CC/\ZZ$ by evaluating the effect of spectral flow on the weights in the top space of $\slsem{\lambda'}{\nu'}$ with maximal $\srt{2}$-coefficient and comparing with those in the top space of $\dynk(\slsem{\lambda}{\nu})$ with minimal $\srt{2}$-coefficient.
	Using \eqref{eq:sfwts}, the weight vectors corresponding to the former have $\slthree$-weights $(j_{\lambda'}-\tfrac{\kk}{3}) \srt{2} + t' \srt{3}$ and conformal weight $\bpconfwt{\lambda'} + \tfrac{\kk}{3} - j_{\lambda'}$, where $t' \in [\nu']$.
	The corresponding data for the latter is $(j_{\lambda}+\tfrac{\kk}{3}) \srt{1} + (t-\tfrac{\kk}{3}) \srt{3}$ and $\bpconfwt{\lambda} + \tfrac{\kk}{3}$,
	respectively.
	Equating these weights and conformal weights using \cref{prop:bphwjD} fixes $\lambda'$ and $[\nu']$ uniquely.
\end{proof}

Modulo the proviso of \cref{rem:fudging}, this finally leads us to the identification
\begin{equation} \label{eq:identM}
	\slver{} \cong \slsf{\fcwt{3}}(\slsem{\nabla(\lambda)}{\nu-\uu/6}).
\end{equation}
We now summarise the structure of $\slrel{\lambda}{0,\nu}$ not as a direct sum decomposition, but as a short exact sequence:
\begin{equation} \label{eq:sesSRS}
	\ses{\slsem{\lambda}{\nu}}{\slrel{\lambda}{0,\nu}}{\slsf{\fcwt{3}}(\slsem{\nabla(\lambda)}{\nu-\uu/6})}.
\end{equation}
Since spectral flow functors are exact, we get another short exact sequence by replacing $\lambda$ by $\nabla(\lambda)$, $\nu$ by $\nu-\frac{\uu}{6}$ and applying $\slsf{\fcwt{3}}$.
Splicing the two sequences gives a four-term exact sequence and repeating this results in an infinite coresolution for $\slsem{\lambda}{\nu}$ in terms of spectral flows of fully relaxed modules:
\begin{equation} \label{eq:coressemi}
	0
	\lra \slsem{\lambda}{\nu}
	\lra \slrel{\lambda}{0,\nu}
	\lra \slsf{\fcwt{3}}(\slrel{\nabla(\lambda)}{0,\nu-\uu/6})
	\lra \slsf{2\fcwt{3}}(\slrel{\nabla^2(\lambda)}{0,\nu-2\uu/6})
	\lra \cdots.
\end{equation}

The \ep\ principle now implies the following identity of generalised characters:
\begin{equation}
	\ch{\slsf{g}(\slsem{\lambda}{\nu})}
	= \sum_{n\ge0} (-1)^n \ch{\slsf{g+n\fcwt{3}}(\slrel{\nabla^n(\lambda)}{0,\nu-n\uu/6})}.
\end{equation}
From this, we obtain the semirelaxed S-matrix elements
\begin{equation}
	\slSmat{(g,\lambda,\nu),(g',\lambda',\mu',\nu')}^{\ssemi}
	= \sum_{n\ge0} (-1)^n \slSmat{(g+n\fcwt{3},\nabla^n(\lambda),0,\nu-n\uu/6),(g',\lambda',\mu',\nu')}.
\end{equation}
If we substitute the fully relaxed S-matrix given in \cref{thm:Sfully} and sum the resulting geometric series in $n$, then the result is precisely \cref{thm:Sfully} but with the ghost S-matrix replaced by its vacuum counterpart \eqref{eq:bgvacSmatrix} (as expected).

However, the result of this computation will be significantly nicer if we use the fully relaxed S-matrix of \cref{cor:Sfully}.
For this, we reparametrise $\nu$ in terms of $\gamma = \gamma(\lambda,0,\nu)$ as in \eqref{eq:defgamma}.
We emphasise that $[\gamma] \in \csub^*/\rlat$ is not arbitrary when parametrising a semirelaxed module.
In particular, inspecting \eqref{eq:uwts} shows that the parameters of $\slsem{\lambda}{\gamma}$ must satisfy
\begin{equation} \label{eq:atypsemi}
	\bilin{\gamma}{\fcwt{3}} = -j_{\lambda}-\tfrac{\uu}{6} \mod{\ZZ}.
\end{equation}
With this in mind, we shall in what follows frequently write $\slsem{\lambda}{\gamma}$ instead of $\slsem{\lambda}{\nu}$, trusting that context will distinguish $[\gamma] \in \csub^*/\rlat$ from $[\nu] \in \CC/\ZZ$.
Recalling that $\fwt{3}=\bilin{\fcwt{3}}{\blank}$ and noting that
\begin{equation}
	\gamma(\nabla^n(\lambda),0,\nu-\tfrac{n\uu}{6}) = \gamma(\lambda,0,\nu) - \tfrac{n\uu}{2}\fwt{3} \mod{\rlat},
\end{equation}
the short exact sequence \eqref{eq:sesSRS} and coresolution \eqref{eq:coressemi} for $\slsem{\lambda}{\gamma}$ translate into
\begin{equation} \label{eq:coressemi'}
	\begin{gathered}
		\ses{\slsem{\lambda}{\gamma}}{\slrel{\lambda}{\gamma}}{\slsf{\fcwt{3}}(\slsem{\nabla(\lambda)}{\gamma-\uu\fwt{3}/2})}, \\
		0
		\lra \slsem{\lambda}{\gamma}
		\lra \slrel{\lambda}{\gamma}
		\lra \slsf{\fcwt{3}}(\slrel{\nabla(\lambda)}{\gamma-\uu\fwt{3}/2})
		\lra \slsf{2\fcwt{3}}(\slrel{\nabla^2(\lambda)}{\gamma-2\uu\fwt{3}/2})
		\lra \cdots,
	\end{gathered}
\end{equation}
while the semirelaxed \ep\ sum of generalised characters becomes
\begin{equation} \label{eq:epsemi}
	\ch{\slsf{g}(\slsem{\lambda}{\gamma})}
	= \sum_{n\ge0} (-1)^n \ch{\slsf{g+n\fcwt{3}}(\slrel{\nabla^n(\lambda)}{\gamma-n\uu\fwt{3}/2})}.
\end{equation}
\begin{proposition} \label{prop:Ssemi}
	For $\uu \in \set{3,5,7,\dots}$, the modular $S$-transforms of the generalised characters of the $\slsf{g}(\slsem{\lambda}{\gamma})$, with $g \in \cwlat$, $\lambda \in \pwlat$ and $[\gamma] \in \csub^*_{\RR}/\rlat$ satisfying \eqref{eq:atypsemi}, are given, up to an omitted automorphy factor, by
	\begin{subequations}
		\begin{gather}
			\fch{\slsf{g}(\slsem{\lambda}{\gamma})}{\tfrac{\zeta}{\tau},\tfrac{\theta}{\tau},\tfrac{\zeta_a}{\tau},\tfrac{\zeta_d}{\tau};-\tfrac{1}{\tau}}
				= \sum_{g'\in\cwlat} \sum_{\lambda'\in\pwlat} \int_{\csub^*_{\RR}/\rlat} \slSmat{(g,\lambda,\gamma),(g',\lambda',\gamma')}^{\ssemi} \fch{\slsf{g'}(\slrel{\lambda'}{\gamma'})}{\zeta,\theta,\zeta_a,\zeta_b;\tau} \, \dd [\gamma'], \\
			\slSmat{(g,\lambda,\gamma),(g',\lambda',\gamma')}^{\ssemi} = \frac{\slSmat{(g,\lambda,\gamma),(g',\lambda',\gamma')}}{1+\ee^{-2\pi\ii(\bilin{\fcwt{3}}{\gamma'}+j_{\lambda'}-\uu/3)}}. \label{eq:scoeff_semirelaxed}
		\end{gather}
	\end{subequations}
\end{proposition}
\begin{proof}
	According to \eqref{eq:epsemi}, the semirelaxed S-matrix is given by
	\begin{equation}
		\slSmat{(g,\lambda,\gamma),(g',\lambda',\gamma')}^{\ssemi}
		= \sum_{n\ge0} (-1)^n \slSmat{(g+n\fcwt{3},\nabla^n(\lambda),\gamma-n\uu\fwt{3}/2),(g',\lambda',\gamma')}.
	\end{equation}
	Substituting \eqref{eq:slSfully!} and using \cref{cor:bpSsymm}, we factor out the dependence on $n$ to get
	\begin{equation}
		\slSmat{(g,\lambda,\gamma),(g',\lambda',\gamma')}^{\ssemi}
		= \slSmat{(g,\lambda,\gamma),(g',\lambda',\gamma')} \sum_{n\ge0} (-1)^n \ee^{-2\pi\ii n\bilin{\fcwt{3}}{\gamma'}} \ee^{-2\pi\ii n(j_{\lambda'}-\uu/3)}.
	\end{equation}
	Replacing this geometric series by its sum, we are done.
\end{proof}

We remark that the manipulation with the geometric sum is strictly speaking unjustified because we are summing the series at its radius of convergence.
One can however think of this sum as a convenient shorthand for the infinite series.
More interestingly, we mention that with this shorthand, the semirelaxed S-matrix elements exhibit a pole precisely when $\gamma'$ satisfies the condition \eqref{eq:atypsemi}, meaning that $\slrel{\lambda'}{\gamma'}$ is reducible.
This seems to be a standard feature of modularity in nonsemisimple categories, see \cite{RidVer14}.

As promised, this coresolution method can be easily adapted to compute the modular S-transform of the other classes of semirelaxed \hw\ $\slminmod$-modules.
As $\dynk$-twists are realised as spectral flows (\cref{lem:dS=sfS}), we give the result for Weyl twists.
This requires an easy \lcnamecref{lem:Wfully}, itself a consequence of the fact \cite{MatCla00} that the character of a coherent family is invariant under the action of the Weyl group.
\begin{lemma}[\protect{\cite[Prop.~4.6]{KawAdm21}}] \label{lem:Wfully}
	For $\uu \in \set{3,5,7,\dots}$, $\wref{} \in \grp{S}_3$, $\lambda \in \pwlat$ and $[\gamma] \in \csub^*/\rlat$, we have
	\begin{equation}
		\wref{}(\slrel{\lambda}{\gamma}) \simeq \slrel{\lambda}{\wref{}(\gamma)}.
	\end{equation}
\end{lemma}
\noindent As with \cref{rem:fudging}, this \lcnamecref{lem:Wfully} holds strictly when $\slrel{\lambda}{\gamma}$ is irreducible and is otherwise true at the level of generalised characters.
\begin{theorem} \label{thm:Ssemiw1}
	For $\uu \in \set{3,5,7,\dots}$, the modular $S$-transforms of the generalised characters of the $\slsf{g}\wref{}(\slsem{\lambda}{\nu})$, with $g\in\cwlat$, $\wref{} \in \grp{S}_3$, $\lambda\in\pwlat$ and $[\gamma]\in\csub^*_{\RR}/\rlat$ satisfying \eqref{eq:atypsemi}, are given, up to an omitted automorphy factor, by
	\begin{subequations}
		\begin{gather}
			\fch{\slsf{g}\wref{}(\slsem{\lambda}{\gamma})}{\tfrac{\zeta}{\tau},\tfrac{\theta}{\tau},\tfrac{\zeta_a}{\tau},\tfrac{\zeta_d}{\tau};-\tfrac{1}{\tau}}
				= \sum_{g'\in\cwlat} \sum_{\lambda'\in\pwlat} \int_{\csub^*_{\RR}/\rlat} \slSmat{(g,\lambda,\gamma),(g',\lambda',\gamma')}^{\ssemi,\wref{}} \fch{\slsf{g'}(\slrel{\lambda'}{\gamma'})}{\zeta,\theta,\zeta_a,\zeta_b;\tau} \, \dd [\gamma'], \\
			\slSmat{(g,\lambda,\gamma),(g',\lambda',\gamma')}^{\ssemi,\wref{}} = \frac{\slSmat{(g,\lambda,\wref{}(\gamma)),(g',\lambda',\gamma')}}{1+\ee^{-2\pi\ii(\bilin{\wref{}(\fcwt{3})}{\gamma'}+j_{\lambda'}-\uu/3)}}. \label{eq:scoeff_semirelaxed_w}
		\end{gather}
	\end{subequations}
\end{theorem}
\begin{proof}
	Start by applying the exact functor $\wref{}$ to the coresolution \eqref{eq:coressemi'}.
	Using \eqref{eq:dihedral} and \cref{lem:Wfully}, we get
\begin{equation} \label{eq:coressemi'w}
	0
	\lra \wref{}(\slsem{\lambda}{\gamma})
	\lra \slrel{\lambda}{\wref{}(\gamma)}
	\lra \slsf{\wref{}(\fcwt{3})}(\slrel{\nabla(\lambda)}{\wref{}(\gamma)-\uu\wref{}(\fwt{3})/2})
	\lra \slsf{2\wref{}(\fcwt{3})}(\slrel{\nabla^2(\lambda)}{\wref{}(\gamma)-2\uu\wref{}(\fwt{3})/2})
	\lra \cdots.
\end{equation}
The corresponding \ep\ sum thus takes the form
\begin{equation} \label{eq:epsemiw}
	\ch{\slsf{g}\wref{}(\slsem{\lambda}{\gamma})}
	= \sum_{n\ge0} (-1)^n \ch{\slsf{g+n\wref{}(\fcwt{3})}(\slrel{\nabla^n(\lambda)}{\wref{}(\gamma)-n\uu\wref{}(\fwt{3})/2})}
\end{equation}
and its evaluation proceeds in exactly the same fashion as in the proof of \cref{prop:Ssemi}.
\end{proof}

\subsection{Highest-weight $\slminmod$-modules} \label{sec:modhw}

Recall that the semirelaxed $\slminmod$-modules $\slsem{\lambda}{\nu}$ degenerate when $[\nu] = [t^i_{\lambda}]$, $i=1,2$, see \eqref{eq:semizeroes}.
More precisely, $\slsem{\lambda}{t^i_{\lambda}}$ is reducible with an irreducible submodule isomorphic to $\wref{1}\wref{2}(\slirr{\Lambda^i_{\lambda}})$, $i=1,2$, see \eqref{eq:defLambda}.
With respect to the $\slthree$-weight parametrisation \eqref{eq:defgamma}, we have $\slsem{\lambda}{t^i_{\lambda}} = \slsem{\lambda}{\gamma^i_{\lambda}}$, where $[\gamma^i_{\lambda}] \in \csub^*/\rlat$ is given by
\begin{equation} \label{eq:defgamma'}
	[\gamma^i_{\lambda}] = [\gamma(\lambda,0,t^i_{\lambda})] =
	\begin{cases*}
		[(3j_{\lambda}+\tfrac{\uu}{2})\fwt{2}] & if $i=1$, \\
		[(3j_{\lambda}+\tfrac{\uu}{2})\fwt{2}-\tfrac{\uu}{2}\srt{3}] & if $i=2$.
	\end{cases*}
\end{equation}

Consider now the quotient module
\begin{equation}
	\slext{i} = \slsem{\lambda}{\gamma^i_{\lambda}} / \wref{1}\wref{2}(\slirr{\Lambda^i_{\lambda}}), \quad \lambda \in \pwlat,\ i=1,2.
\end{equation}
Since $\wref{1}\wref{2}(\slirr{\Lambda^i_{\lambda}})$ is generated by the $\wref{1}\wref{2}$-twisted \hwv\ $u_{0,0,t^i_{\lambda}}$, it follows that the image of $u_{0,0,t^i_{\lambda}-1}$ is a $\wref{1}$-twisted \hwv\ in the quotient, see \cref{fig:twhwvs}.
Using \cref{thm:sl3hwclass}, this vector generates an irreducible submodule of $\slext{i}$.
By computing its weight, we find that this submodule is isomorphic to $\wref{1}(\slirr{\Lambda^2_{\lambda}})$ and $\wref{1}(\slirr{\Lambda^1_{\lambda}})$ for $i=1$ and $2$, respectively.

A natural question is whether these are all the submodules of $\slext{1}$ and $\slext{2}$.
The answer is yes because the $\slsem{\lambda}{\gamma^i_{\lambda}}$ are almost irreducible and the highest weights lie (after untwisting) in $\admwts$.
\cref{thm:hwlocalisation} therefore applies, proving that the quotient $\slext{i}$ is irreducible.
Consequently, we have the following (nonsplit) short exact sequences:
\begin{equation}\label{eq:ses_hwmod}
	\ses{\wref{1}\wref{2}(\slirr{\Lambda^i_{\lambda}})}{\slsem{\lambda}{\gamma^i_{\lambda}}}{\wref{1}(\slirr{\Lambda^{3-i}_{\lambda}})}, \quad i=1,2.
\end{equation}
Note that since the functor $\wref{3}$ is invertible, applying it to these
exact sequences results in two more such sequences, namely
\begin{equation}\label{eq:ses_hwmod_w3}
  \ses{\wref{1}(\slirr{\Lambda^i_{\lambda}})}{\wref{3}(\slsem{\lambda}{\gamma^i_{\lambda}})}{\wref{1}\wref{2}(\slirr{\Lambda^{3-i}_{\lambda}})},\quad i=1,2.
\end{equation}
A consequence is the following character identity:
\begin{equation} \label{lem:w3S=S*}
	\ch{\wref{3}(\slsem{\lambda}{\gamma^i_{\lambda}})}=\ch{\slsem{\lambda}{\gamma^{3-i}_{\lambda}}},\quad i=1,2.
\end{equation}

\begin{remark}
	In fact, it can be shown that for $i=1,2$, $\wref{3}(\slsem{\lambda}{\gamma^{3-i}_{\lambda}})$ is the contragredient dual of $\slsem{\lambda}{\gamma^i_{\lambda}}$, twisted by the involution of $\aslthree$ defined by
  \begin{equation}
    xt^m \mapsto x^{\intercal} t^{-m} \quad \text{and} \quad K \mapsto -K, \qquad x \in \slthree,\ m \in \ZZ,
  \end{equation}
  where ${}^{\intercal}$ is the usual transposition of matrices.
  We will not need this fact for what follows.
\end{remark}

To convert these exact sequences into coresolutions, we again realise the quotients as spectral flows.
For this, some special cases of \cite[Prop.~4.8 and Fig.~4]{KawAdm21} will be needed.
\begin{lemma} \label{lem:wL=sfL}
	For $\uu\in\set{3,5,7,\dots}$ and $\lambda\in\pwlat$, we have the following identifications of spectral flows and $\grp{S}_3$-twists of irreducible \hw\ $\slminmod$-modules:
	\begin{equation}
		\slsf{\fcwt{3}}\wref{1}(\slirr{\Lambda^1_{\lambda}}) \cong \wref{2}(\slirr{\Lambda^1_{\nabla^{-1}(\lambda)}}),
		\qquad
		\begin{aligned}
			\slsf{-\fcwt{2}}(\slirr{\Lambda^2_{\lambda}}) &\cong \slirr{\lambda-\uu\fwt{0}/2}, &
			\slsf{-\scrt{2}}(\slirr{\Lambda^2_{\lambda}}) &\cong \wref{2}(\slirr{\Lambda^2_{\lambda}}), \\
			\slsf{\fcwt{3}}(\slirr{\Lambda^2_{\lambda}}) &\cong \slirr{\nabla^{-1}(\lambda)-\uu\fwt{1}/2}, &
			\slsf{-\scrt{3}}(\slirr{\Lambda^2_{\lambda}}) &\cong \wref{1}\wref{2}(\slirr{\Lambda^2_{\lambda}}).
		\end{aligned}
	\end{equation}
\end{lemma}
\noindent It moreover follows from the middle column of identifications that if one has coresolutions for the irreducibles of highest weights $\Lambda^1_{\lambda} = \wref{1}\cdot(\lambda-\frac{\uu}{2}\fwt{1})$ and $\Lambda^2_{\lambda} = \nabla(\lambda)-\frac{\uu}{2}\fwt{2}$, then applying spectral flow results in coresolutions for all the irreducibles with highest weights in $\admwts$.

We now start deriving a coresolution for $\slirr{\Lambda^2_{\lambda}}$.
First, apply $\wref{2}\wref{1} = \wref{1}\wref{3}$ to the $i=2$ short exact sequence in \eqref{eq:ses_hwmod} and use \cref{lem:wL=sfL} to get
\begin{equation} \label{eq:ses2}
	\ses{\slirr{\Lambda^2_{\lambda}}}{\wref{1}\wref{3}(\slsem{\lambda}{\gamma^2_{\lambda}})}{\slsf{\fcwt{3}}\wref{1}(\slirr{\Lambda^1_{\nabla(\lambda)}})}.
\end{equation}
Meanwhile, applying this \lcnamecref{lem:wL=sfL} to the $i=1$ sequence in \eqref{eq:ses_hwmod_w3} gives
\begin{equation}
	\ses{\wref{1}(\slirr{\Lambda^1_{\lambda}})}{\wref{3}(\slsem{\lambda}{\gamma^1_{\lambda}})}{\slsf{-\scrt{3}}(\slirr{\Lambda^2_{\lambda}})}.
\end{equation}
By appropriately replacing $\lambda$ by its $\nabla$-permutations and applying spectral flow functors, we can splice these short exact sequences together and arrive at the desired coresolution:
\begin{multline} \label{eq:cores2}
	0
  \lra \slirr{\Lambda^2_{\lambda}}
  \lra \wref{1}\wref{3}(\slsem{\lambda}{\gamma^2_{\lambda}})
  \lra \slsf{\fcwt{3}}\wref{3}(\slsem{\nabla(\lambda)}{\gamma^1_{\nabla(\lambda)}})
  \lra \slsf{-2\fcwt{2}}\wref{1}\wref{3}(\slsem{\nabla(\lambda)}{\gamma^2_{\nabla(\lambda)}}) \\
  \lra \slsf{\fcwt{3}-2\fcwt{2}}\wref{3}(\slsem{\nabla^2(\lambda)}{\gamma^1_{\nabla^2(\lambda)}})
  \lra \slsf{-4\fcwt{2}}\wref{1}\wref{3}(\slsem{\nabla^2(\lambda)}{\gamma^2_{\nabla^2(\lambda)}})
  \lra \cdots.
\end{multline}
\begin{theorem} \label{thm:Shw}
	For $\uu \in \set{3,5,7,\dots}$, the modular $S$-transforms of the generalised characters of the $\slsf{g}(\slirr{\Lambda^2_{\lambda}})$, with $g\in\cwlat$ and $\lambda\in\pwlat$, are given, up to an omitted automorphy factor, by
	\begin{subequations}
		\begin{gather}
			\fch{\slsf{g}(\slirr{\Lambda^2_{\lambda}})}{\tfrac{\zeta}{\tau},\tfrac{\theta}{\tau},\tfrac{\zeta_a}{\tau},\tfrac{\zeta_d}{\tau};-\tfrac{1}{\tau}}
			= \sum_{g'\in\cwlat} \sum_{\lambda'\in\pwlat} \int_{\csub^*_{\RR}/\rlat} \slSmat{(g,\Lambda^2_{\lambda}),(g',\lambda',\gamma')}^{\shw} \fch{\slsf{g'}(\slrel{\lambda'}{\gamma'})}{\zeta,\theta,\zeta_a,\zeta_b;\tau} \, \dd [\gamma'], \\
			\slSmat{(g,\Lambda^2_{\lambda}),(g',\lambda',\gamma')}^{\shw}
			= \frac{\slSmat{(g,\lambda,\gamma^1_{\lambda}),(g',\lambda',\gamma')}}{(1+\ee^{2\pi\ii(\bilin{\fcwt{1}}{\gamma'}-j_{\lambda'}+\uu/3)}) (1+\ee^{2\pi\ii(\bilin{\fcwt{2}}{\gamma'}+j_{\lambda'}-\uu/3)}) (1+\ee^{-2\pi\ii(\bilin{\fcwt{3}}{\gamma'}+j_{\lambda'}-\uu/3)})}. \label{eq:scoeff_hw2}
		\end{gather}
	\end{subequations}
\end{theorem}
\begin{proof}
	This is proven similarly to \cref{prop:Ssemi}, so we provide only a sketch.
	The coresolution \eqref{eq:cores2} and the character identity \eqref{lem:w3S=S*} imply that
	\begin{equation}
		\slSmat{(g,\Lambda^2_{\lambda}),(g',\lambda',\gamma')}^{\shw} = \sum_{n\ge0} \sqbrac*{
			\slSmat{(g-2ng^2,\nabla^n(\lambda),\wref{1}(\gamma^1_{\nabla^n(\lambda)})),(g',\lambda',\gamma')}^{\ssemi,\wref{1}} -
			\slSmat{(g+g^3-2ng^2,\nabla^{n+1}(\lambda),\gamma^2_{\nabla^{n+1}(\lambda)}),(g',\lambda',\gamma')}^{\ssemi}
		}.
	\end{equation}
	We simplify this using \cref{thm:Ssemiw1}, then \cref{cor:Sfully,cor:bpSsymm}, and substituting \eqref{eq:defgamma'}.
	Several terms cancel resulting in the same $n$-dependence in both S-matrices in the sum.
	These $n$-dependent terms give a factor of
	\begin{equation}
		\frac{1}{1-\ee^{2\pi\ii(2\bilin{g^2}{\gamma'}+2j_{\lambda'}-2\uu/3)}}
	\end{equation}
	while putting the remaining parts of the S-matrices over a common denominator gives a numerator equal to
	\begin{equation}
		({1-\ee^{2\pi\ii(\bilin{g^2}{\gamma'}+j_{\lambda'}-\uu/3)}})\slSmat{(g,\lambda,\gamma^1_{\lambda}),(g',\lambda',\gamma')}.
	\end{equation}
	Simplifying, we arrive at \eqref{eq:scoeff_hw2}.
\end{proof}

One can now use \cref{lem:wL=sfL} to write down S-transforms for the generalised characters of the \hw\ irreducibles $\slirr{\lambda-\uu\fwt{0}/2}$ and $\slirr{\lambda-\uu\fwt{1}/2}$, $\lambda\in\pwlat$, and their spectral flows.
Those of the $\slirr{\Lambda^1_{\lambda}} = \slirr{\wref{1}\cdot(\lambda-\uu\fwt{1}/2)}$ also follow by applying $\wref{2}\wref{1}$ to the $i=1$
sequence in \eqref{eq:ses_hwmod}, obtaining
\begin{equation} \label{eq:ses1}
	\ses{\slirr{\Lambda^1_{\lambda}}}{\wref{1}\wref{3}(\slsem{\lambda}{\gamma^1_{\lambda}})}{\slsf{-\scrt{2}}(\slirr{\Lambda^2_{\lambda}})}.
\end{equation}
This gives a complete set of S-transforms for the generalised characters of the irreducible \hw\ $\slminmod$-modules and their spectral flows.
Moreover, one can use \cref{lem:Wfully} to extend this to include their $\grp{D}_6$-twists (this is left as an exercise).
Here, we conclude by specialising to the S-transform of the vacuum generalised character.
\begin{corollary} \label{cor:Svac}
	For $\uu \in \set{3,5,7,\dots}$, the modular $S$-transform of the generalised character of the vacuum $\slminmod$-module $\slirr{\kk\fwt{0}} \cong \slsf{-\fcwt{2}}(\slirr{\Lambda^2_{(\uu-3)\fwt{0}}})$ is given, up to an omitted automorphy factor, by
	\begin{subequations}
		\begin{gather}
			\fch{\slirr{\kk\fwt{0}}}{\tfrac{\zeta}{\tau},\tfrac{\theta}{\tau},\tfrac{\zeta_a}{\tau},\tfrac{\zeta_d}{\tau};-\tfrac{1}{\tau}}
			= \sum_{g'\in\cwlat} \sum_{\lambda'\in\pwlat} \int_{\csub^*_{\RR}/\rlat} \slSmat{(\vac),(g',\lambda',\gamma')}^{\shw} \fch{\slsf{g'}(\slrel{\lambda'}{\gamma'})}{\zeta,\theta,\zeta_a,\zeta_b;\tau} \, \dd [\gamma'], \\
			\slSmat{(\vac),(g',\lambda',\gamma')}^{\shw}
			= \frac{\ee^{-2\pi\ii(j_{\lambda'}-\uu/3)} \bpSmat{\vac,\lambda'}}{2\brac[\Big]{1 + \cos\brac[\big]{2\pi(\bilin{\fcwt{1}}{\gamma'}-j_{\lambda'}+\frac{\uu}{3})} + \cos\brac[\big]{2\pi(\bilin{\fcwt{2}}{\gamma'}+j_{\lambda'}-\frac{\uu}{3})} + \cos\brac[\big]{2\pi(\bilin{\fcwt{3}}{\gamma'}+j_{\lambda'}-\frac{\uu}{3})}}}. \label{eq:scoeff_vac}
		\end{gather}
	\end{subequations}
\end{corollary}
\noindent Restricting to $\uu=3$ and $\lambda'=0$ now reproduces \cite[Cor.~5.10]{KawAdm21}.
Note that the residual exponential factor in the numerator of \eqref{eq:scoeff_vac} cancels a similar factor in $\bpSmat{\vac,\lambda'}$, see \cref{thm:bpmod}.
This ``asymmetry'' may be traced back to our choice of \emt\ in \eqref{eq:bpopes} and the resulting unequal conformal weights for $G^+$ and $G^-$.

\subsection{Grothendieck fusion rules} \label{sec:fusion}

Having computed the S-transforms of the generalised characters of all the irreducible \rhw\ $\slminmod$-modules and their spectral flows, our final task is to substitute these results into the standard Verlinde formula of \cite{CreLog13,RidVer14}.
Conjecturally, this will compute the Grothendieck fusion coefficients of $\slminmod$.
This conjecture encompasses several claims that we shall leave unsubstantiated including:
\begin{itemize}
	\item The category of finitely generated weight $\slminmod$-modules with finite multiplicities is closed under fusion.
	\item Fusing with an arbitrary fixed weight $\slminmod$-module defines an exact functor on this category.
	\item The product $\fuse$ on the Grothendieck group, induced by the fusion product, has multiplicities given by the standard Verlinde formula.
\end{itemize}
We will implicitly assume that this conjecture holds in what follows.
Despite this reliance on conjecture, it is striking that our computations below always result in Grothendieck fusion rules in which the multiplicities are nonnegative integers.
We view this as a strong consistency check on our work (as well as evidence for the conjecture's truth).

We remark that the standard Verlinde formula for the (conjectural) Grothendieck fusion coefficients is computed from the modular S-transforms of the generalised characters of the $\slminmod$-modules.
Identifying the results with the structure constants of the Grothendieck fusion ring, rather than some quotient, is only possible because the generalised characters of the irreducible $\slminmod$-modules are linearly independent.

The key to the standard module formalism of \cite{CreLog13,RidVer14} is the identification of so-called standard modules.
Their (generalised) characters carry an action of the modular group and form a topological basis for the space of all (generalised) characters of the category under consideration.
In our study, the standard modules are the spectral flows of the fully relaxed modules $\slsf{g}(\slrel{\lambda}{\gamma})$, with $g\in\cwlat$, $\lambda\in\pwlat$ and $[\gamma]\in\csub^*_{\RR}/\rlat$.
(Technically, we should also admit the $\grp{D}_6$-twists of the reducible fully relaxed modules, but this is not important at the level of generalised characters.)
That they carry an action of the modular group is essentially \cref{thm:Sfully} (the T-transforms are trivial to compute) and their status as a topological basis is the content of coresolutions like \eqref{eq:coressemi'w} and \eqref{eq:cores2} as well as the short exact sequence \eqref{eq:ses1}.

Let $\Gr{\slver{}}$ denote the image of the $\slminmod$-module $\slver{}$ in the Grothendieck group.
The Grothendieck fusion rule for two arbitrary standard modules then takes the form
\begin{subequations} \label{eq:Verlinde}
	\begin{equation} \label{eq:grfusion}
		\Gr{\slsf{g}(\slrel{\lambda}{\gamma})} \fuse \Gr{\slsf{g'}(\slrel{\lambda'}{\gamma'})}
		= \sum_{g''\in\cwlat} \sum_{\lambda''\in\pwlat} \int_{\csub^*_{\RR}/\rlat} \slfuscoeff{(g,\lambda,\gamma)}{(g',\lambda',\gamma')}{(g'',\lambda'',\gamma'')} \Gr{\slsf{g''}(\slrel{\lambda''}{\gamma''})} \dd [\gamma''].
	\end{equation}
	Here, the Grothendieck fusion coefficients are (conjecturally) given by the standard Verlinde formula:
	\begin{equation} \label{eq:fusion_coeff}
		\slfuscoeff{(g,\lambda,\gamma)}{(g',\lambda',\gamma')}{(g'',\lambda'',\gamma'')}
		= \sum_{G\in\cwlat} \sum_{\Lambda\in\pwlat} \int_{\csub^*_{\RR}/\rlat} \frac{\slSmat{(g,\lambda,\gamma),(G,\Lambda,\Gamma)} \slSmat{(g',\lambda',\gamma'),(G,\Lambda,\Gamma)} \slSmat{(g'',\lambda'',\gamma''),(G,\Lambda,\Gamma)}^*}{\slSmat{(\vac),(G,\Lambda,\Gamma)}^{\shw}} \dd [\Gamma].
	\end{equation}
\end{subequations}
This generalises in the obvious way to accommodate semirelaxed and \hw\ $\slminmod$-modules.
One simply replaces $\slrel{\lambda}{\gamma}$ and $\slrel{\lambda'}{\gamma'}$ on the \lhs\ of \eqref{eq:grfusion} by the desired modules and the corresponding S-matrices in \eqref{eq:fusion_coeff} by their semirelaxed and \hw\ variants.
(The fully relaxed modules on the \rhs\ are not replaced, hence the conjugated and vacuum S-matrix elements remain unmodified.)

These Grothendieck fusion coefficients behave well under the action of the automorphisms of $\slminmod$.
\begin{lemma} \label{lem:fusauts}
	For $\uu\in\set{3,5,7,\dots}$, $g,g',g''\in\cwlat$, $\lambda,\lambda',\lambda''\in\pwlat$, $[\gamma],[\gamma'],[\gamma'']\in\csub^*_{\RR}/\rlat$ and $\wref{}\in\grp{S}_3$, the Grothendieck fusion coefficients defined by \eqref{eq:Verlinde} satisfy the following identities:
	\begin{equation}
		\begin{aligned}
			\slfuscoeff{(\wref{}(g),\lambda,\wref{}(\gamma))}{(\wref{}(g'),\lambda',\wref{}(\gamma'))}{(\wref{}(g''),\lambda'',\wref{}(\gamma''))}
			&= \slfuscoeff{(g,\lambda,\gamma)}{(g',\lambda',\gamma')}{(g'',\lambda'',\gamma'')}, \\
			\slfuscoeff{(\dynk(g),\dynk(\lambda),\dynk(\gamma))}{(\dynk(g'),\dynk(\lambda'),\dynk(\gamma'))}{(\dynk(g''),\dynk(\lambda''),\dynk(\gamma''))}
			&= \slfuscoeff{(g,\lambda,\gamma)}{(g',\lambda',\gamma')}{(g'',\lambda'',\gamma'')}, \\
			\slfuscoeff{(0,\lambda,\gamma)}{(0,\lambda',\gamma')}{(g''-g'-g,\lambda'',\gamma'')}
			&= \slfuscoeff{(g,\lambda,\gamma)}{(g',\lambda',\gamma')}{(g'',\lambda'',\gamma'')}.
		\end{aligned}
	\end{equation}
\end{lemma}
\begin{proof}
	Because $\wref{}$ is orthogonal, the first identity follows immediately from the $\grp{S}_3$-invariance of the fully relaxed S-matrix \eqref{eq:slSfully!}.
	The second uses the orthogonality of $\dynk$, but also the second identity of \cref{cor:bpfussymm}.
	The third is likewise an easy consequence of the explicit form \eqref{eq:slSfully!}.
\end{proof}
\noindent Let $\slver{}$ and $\slver{}'$ be weight $\slminmod$-modules with finite multiplicities.
Then, the first two identities above say that if we know the Grothendieck fusion of $\slver{}$ and $\slver{}'$, then we know that of $\Omega(\slver{})$ and $\Omega(\slver{}')$ for any $\Omega \in \grp{D}_6$: just apply $\Omega$ to each summand.
We may therefore reduce the general calculations to those for which only one of $\slver{}$ and $\slver{}'$ is $\grp{D}_6$-twisted.
Similarly, the third identity gives us the Grothendieck fusion of $\slsf{g}(\slver{})$ and $\slsf{g'}(\slver{}')$ for any $g,g'\in\cwlat$, assuming that we know it for $\slver{}$ and $\slver{}'$.
We may thus assume that both modules are untwisted by spectral flow.

This \lcnamecref{lem:fusauts} will be used to simplify all the Grothendieck fusion rules presented below.
\begin{theorem} \label{thm:relfus}
	For $\uu\in\set{3,5,7,\dots}$, $\lambda,\lambda'\in\pwlat$ and $[\gamma],[\gamma']\in\csub^*_{\RR}/\rlat$, the Grothendieck fusion rules for fully relaxed $\slminmod$-modules take the form
	\begin{equation} \label{eq:fusion_RxR}
		\begin{split}
			&\Gr{\slrel{\lambda}{\gamma}} \fuse \Gr{\slrel{\lambda'}{\gamma'}}
			= \sum_{\lambda''\in\pwlat} \bpfuscoeff{\lambda}{\lambda'}{\lambda''}
			\Bigl( 2 \Gr{\slrel{\nabla^{-1}(\lambda'')}{\gamma+\gamma'}} \Bigr. \\
			&\qquad + \Gr{\slsf{-g^1}(\slrel{\lambda''}{\gamma+\gamma'+\uu\fwt{1}/2})}
			  + \Gr{\slsf{-g^2}(\slrel{\nabla(\lambda'')}{\gamma+\gamma'+u\fwt{2}/2})}
			   + \Gr{\slsf{-g^3}(\slrel{\nabla(\lambda'')}{\gamma+\gamma'+\uu\fwt{3}/2})} \\ \Bigl.
			&\qquad + \Gr{\slsf{g^1}(\slrel{\nabla(\lambda'')}{\gamma+\gamma'-\uu\fwt{1}/2})}
			  + \Gr{\slsf{g^2}(\slrel{\lambda''}{\gamma+\gamma'-\uu\fwt{2}/2})}
			   + \Gr{\slsf{g^3}(\slrel{\lambda''}{\gamma+\gamma'-\uu\fwt{3}/2})} \Bigr).
		\end{split}
	\end{equation}
\end{theorem}
\begin{proof}
	We first compute $\slfuscoeff{(0,\lambda,\gamma)}{(0,\lambda',\gamma')}{(g'',\lambda'',\gamma'')}$, for $g''\in\cwlat$, $\lambda''\in\pwlat$ and $[\gamma'']\in\csub^*_{\RR}/\rlat$, by inserting the S-matrix coefficients \eqref{eq:slSfully!} and \eqref{eq:scoeff_vac} into the standard Verlinde formula \eqref{eq:fusion_coeff}.
	The sum over $\cwlat$ then evaluates to the Dirac delta function
	\begin{equation}
		\delta\brac[\big]{[\gamma'']-[\gamma+\gamma'-\tfrac{\uu}{2}\omega'']},
	\end{equation}
	where $\omega''=\bilin{g''}{\blank}$ is the image of $g''\in\cwlat$ in $\wlat$.
	Because of the form of the vacuum S-matrix coefficient $\slSmat{(\vac),(g',\lambda',\gamma')}^{\shw}$, what remains is a sum over $\Lambda \in \pwlat$ and integral over $[\Gamma] \in \csub^*_{\RR}/\rlat$ of seven terms, each of the form
	\begin{equation}
		\ee^{2\pi\ii \bilin{g''-\tilde{g}}{\Gamma}} \frac{\bpSmat{\lambda,\Lambda} \bpSmat{\lambda',\Lambda} (\bpSmat{\nabla^n(\lambda''),\Lambda})^*}{\bpSmat{\vac,\Lambda}}
	\end{equation}
	for some $n\in\set{-1,0,1}$ and $\tilde{g}\in\set{0,\pm g^i\st i=1,2,3}$.
	(Here, we have also used \cref{cor:bpSsymm} to introduce the $\nabla^n$ in the conjugated S-matrix coefficient.)
	The sum over $\Lambda$ now gives $\bpfuscoeff{\lambda}{\lambda'}{\nabla^n(\lambda'')}$, by \cref{thm:bpfusion}, whereas the integral over $\Gamma$ results in $\delta_{g'',\tilde{g}}$.
	The Grothendieck fusion rule is then obtained by substituting this evaluation of $\slfuscoeff{(0,\lambda,\gamma)}{(0,\lambda',\gamma')}{(g'',\lambda'',\gamma'')}$ into \eqref{eq:grfusion} (with $g=g'=0$).
\end{proof}
\noindent We can bring out the symmetry of the fully relaxed Grothendieck fusion rules by arranging the summands according to the coweights in their spectral flows:
\begin{equation}
	\begin{split}
		\Gr{\slrel{\lambda}{\gamma}} &\fuse \Gr{\slrel{\lambda'}{\gamma'}}
		= \sum_{\lambda''\in\pwlat} \bpfuscoeff{\lambda}{\lambda'}{\lambda''}
		\left\{
			\begin{tikzpicture}[xscale=2.75,yscale=1.5,baseline=(O.base)]
				\node (O) at (0:0) {$2\Gr{\slrel{\nabla^{-1}(\lambda'')}{\gamma+\gamma'}}$};
				\node at (30:1) {$\Gr{\slsf{g^1}(\slrel{\nabla(\lambda'')}{\gamma+\gamma'-\uu\fwt{1}/2})}$};
				\node at (90:1) {$\Gr{\slsf{g^2}(\slrel{\lambda''}{\gamma+\gamma'-\uu\fwt{2}/2})}$};
				\node at (150:1) {$\Gr{\slsf{-g^3}(\slrel{\nabla(\lambda'')}{\gamma+\gamma'+\uu\fwt{3}/2})}$};
				\node at (-150:1) {$\Gr{\slsf{-g^1}(\slrel{\lambda''}{\gamma+\gamma'+\uu\fwt{1}/2})}$};
				\node at (-90:1) {$\Gr{\slsf{-g^2}(\slrel{\nabla(\lambda'')}{\gamma+\gamma'+u\fwt{2}/2})}$};
				\node at (-30:1) {$\Gr{\slsf{g^3}(\slrel{\lambda''}{\gamma+\gamma'-\uu\fwt{3}/2})}$};
			\end{tikzpicture}
		\right\}.
	\end{split}
\end{equation}
This symmetry was previously observed for $\uu=3$ in \cite{KawAdm21}.
We note that in this case, $\lambda=\lambda'=\lambda''=0$ and so the $\nabla$-dependence and \bp\ fusion coefficients were not apparent.

One can similarly use the standard Verlinde formula \eqref{eq:Verlinde} to compute the Grothendieck fusion rules involving semirelaxed and \hw\ $\slminmod$-modules.
However, it is usually more convenient to combine the fully relaxed rules (the standard ones) with expressions for the semirelaxed and \hw\ generalised characters as (infinite) linear combinations of fully relaxed ones.
To this end, the short exact sequences and coresolutions \eqref{eq:coressemi'}, \eqref{eq:coressemi'w}, \eqref{eq:ses_hwmod}, \eqref{eq:cores2} and \eqref{eq:ses1} imply the following identities in the Grothendieck group:
\begin{equation} \label{eq:grids}
	\begin{gathered}
		\Gr{\slrel{\lambda}{\gamma}} = \Gr{\slsem{\lambda}{\gamma}} + \Gr{\slsf{\fcwt{3}}(\slsem{\nabla(\lambda)}{\gamma-\uu\fwt{3}/2})} \qquad \text{(when $\bilin{\gamma}{\fcwt{3}} = -j_{\lambda}-\tfrac{\uu}{6} \mod{\ZZ}$)}, \\
		\Gr{\slsem{\lambda}{\gamma^i_{\lambda}}} = \Gr{\wref{1}\wref{2}(\slirr{\Lambda^i_{\lambda}})} + \Gr{\wref{1}(\slirr{\Lambda^{3-i}_{\lambda}})} \qquad \text{($i=1,2$)}, \\
		\Gr{\slirr{\Lambda^1_{\lambda}}} = \Gr{\wref{1}(\slsem{\lambda}{\gamma^2_{\lambda}})} - \Gr{\slsf{\fcwt{3}}(\slirr{\lambda-\uu\fwt{0}/2})}, \\
		\Gr{\wref{}(\slsem{\lambda}{\gamma})} = \sum_{n\ge0} (-1)^n \Gr{\slsf{n\wref{}(g^3)}(\slrel{\nabla^n(\lambda)}{\wref{}(\gamma)-n\uu\wref{}(\fwt{3})/2})} \qquad \text{($w\in\grp{S}_3$)}, \\
		\Gr{\slirr{\lambda-\uu\fwt{0}/2}} = \sum_{n\ge0} \brac[\Big]{\Gr{\slsf{-(2n+1)\fcwt{2}}\wref{1}(\slsem{\nabla^n(\lambda)}{\gamma^1_{\lambda}+n\uu\fwt{2}})} - \Gr{\slsf{\fcwt{3}-(2n+1)\fcwt{2}}(\slsem{\nabla^{n+1}(\lambda)}{\gamma^2_{\lambda}+(n+1)\uu\fwt{2}})}}.
	\end{gathered}
\end{equation}

Combining these identities with \cref{lem:fusauts,thm:relfus}, we can deduce all remaining Grothendieck fusion rules.
A selection of these are recorded below.
\begin{corollary} \label{cor:allfusions}
	For $\uu\in\set{3,5,7,\dots}$, $\lambda,\lambda'\in\pwlat$, $[\gamma],[\gamma']\in\csub^*_{\RR}/\rlat$ and $\wref{}\in\grp{S}_3$, we have the following Grothendieck fusion rules of weight $\slminmod$-modules:
	\begin{subequations}
		\begin{align}
			\begin{split} \label{eq:fusion_relax_semi}
				\Gr{\wref{}(\slsem{\lambda}{\gamma})} \fuse \Gr{\slrel{\lambda'}{\gamma'}} &= \sum_{\lambda''\in\pwlat} \bpfuscoeff{\lambda}{\lambda'}{\lambda''}
				\Bigl( \Gr{\slrel{\nabla^{-1}(\lambda'')}{\wref{}(\gamma)+\gamma'}} + \Gr{\slsf{-\wref{}(\fcwt{1})}(\slrel{\lambda''}{\wref{}(\gamma)+\gamma'+\uu\wref{}(\fwt{1})/2})} \Bigr. \\
				&\qquad \Bigl. {} + \Gr{\slsf{\wref{}(\fcwt{2})}(\slrel{\lambda''}{\wref{}(\gamma)+\gamma'-\uu\wref{}(\fwt{2})/2})} + \Gr{\slsf{-\wref{}(\fcwt{3})}(\slrel{\nabla(\lambda'')}{\wref{}(\gamma)+\gamma'+\uu\wref{}(\fwt{3})/2})} \Bigr),
			\end{split}
			\\
			\begin{split} 
				\Gr{\slsem{\lambda}{\gamma}} \fuse \Gr{\slsem{\lambda'}{\gamma'}} &= \sum_{\lambda''\in\pwlat} \bpfuscoeff{\lambda}{\lambda'}{\lambda''}
				\Bigl( \Gr{\slsf{-\fcwt{1}}(\slsem{\lambda''}{\gamma+\gamma'+\uu\fwt{1}/2})} \Bigr. \\
				&\qquad \Bigl. {} + \Gr{\slsf{\fcwt{2}}(\slsem{\lambda''}{\gamma+\gamma'-\uu\fwt{2}/2})} + \Gr{\slsf{-\fcwt{3}}(\slrel{\nabla(\lambda'')}{\gamma+\gamma'+\uu\fwt{3}/2})} \Bigr),
			\end{split}
			\\ 
			\Gr{\wref{1}(\slsem{\lambda}{\gamma})} \fuse \Gr{\slsem{\lambda'}{\gamma'}} &= \sum_{\lambda''\in\pwlat} \bpfuscoeff{\lambda}{\lambda'}{\lambda''}
			\biggl( \Gr{\slrel{\nabla^{-1}(\lambda'')}{\wref{1}(\gamma) + \gamma'}} + \Gr{\slsf{\fcwt{2}}(\slrel{\lambda''}{\wref{1}(\gamma)+\gamma'-\uu\fwt{2}/2})} \biggr),
			\\ 
			\Gr{\slirr{\lambda - \uu \fwt{0}/2}} \fuse \Gr{\slrel{\lambda'}{\gamma'}} &= \sum_{\lambda''\in\pwlat} \bpfuscoeff{\lambda}{\lambda'}{\lambda''} \Gr{\slrel{\lambda''}{3j_{\lambda}\fwt{2} + \gamma'}},
			\\ 
			\Gr{\slirr{\lambda - \uu \fwt{0}/2}} \fuse \Gr{\wref{}(\slsem{\lambda'}{\gamma'})} &= \sum_{\lambda''\in\pwlat} \bpfuscoeff{\lambda}{\lambda'}{\lambda''} \Gr{\wref{}(\slsem{\lambda''}{3j_{\lambda}\fwt{2} + \gamma'})},
			\\ 
			\Gr{\slirr{\lambda - \uu \fwt{0}/2}} \fuse \Gr{\slirr{\lambda' - \uu \fwt{0}/2}} &= \sum_{\lambda''\in\pwlat} \bpfuscoeff{\lambda}{\lambda'}{\lambda''} \Gr{\slirr{\lambda'' - \uu \fwt{0}/2}},
			\\ 
			\Gr{\slirr{\lambda - \uu \fwt{0}/2}} \fuse \Gr{\slirr{\Lambda^{1}_{\lambda'}}} &= \sum_{\lambda''\in\pwlat} \bpfuscoeff{\lambda}{\lambda'}{\lambda''} \Gr{\slirr{\Lambda^{1}_{\lambda''}}}, \label{eq:fusion_L0xL1}
			\\ 
			\begin{split}
				\Gr{\slirr{\Lambda_{\lambda}^{1}}} \fuse \Gr{\slrel{\lambda'}{\gamma'}} &= \sum_{\lambda''\in\pwlat} \bpfuscoeff{\lambda}{\lambda'}{\lambda''}
				\Bigl( \Gr{\slrel{\nabla^{-1}(\lambda'')}{3j_{\lambda}\fwt{2} + \gamma' + \uu\fwt{3}/2}} \Bigr. \\
				&\qquad \Bigl. {} + \Gr{\slsf{\fcwt{1}}(\slrel{\nabla(\lambda'')}{3j_{\lambda}\fwt{2}+\gamma'-\uu\fwt{2}/2})} + \Gr{\slsf{\fcwt{2}}(\slrel{\lambda''}{3j_{\lambda}\fwt{2} +\gamma'-\uu\srt{2}/2})} \Bigr),
			\end{split}
			\\ 
			\Gr{\slirr{\Lambda_{\lambda}^{1}}} \fuse \Gr{\slsem{\lambda'}{\gamma'}} &= \sum_{\lambda''\in\pwlat} \bpfuscoeff{\lambda}{\lambda'}{\lambda''}
			\biggl( \Gr{\slsem{\nabla^{-1}(\lambda'')}{3j_{\lambda}\fwt{2} + \gamma' + \uu\fwt{3}/2}} + \Gr{\slsf{\fcwt{2}}(\slrel{\lambda''}{3j_{\lambda}\fwt{2} +\gamma'-\uu\srt{2}/2})} \biggr),
			\\ 
			\begin{split} \label{eq:fusion_L1xwS}
				\Gr{\slirr{\Lambda^{1}_{\lambda}}} \fuse \Gr{\wref{1}(\slsem{\lambda'}{\gamma'})} &= \sum_{\lambda''\in\pwlat} \bpfuscoeff{\lambda}{\lambda'}{\lambda''}
				\Bigl( \Gr{\slsf{\fcwt{1}}(\slrel{\nabla(\lambda'')}{3j_{\lambda}\fwt{2} + \wref{1}(\gamma') - \uu\fwt{2}/2})} \Bigr. \\
				&\qquad \Bigl. {} + \Gr{\slsf{\fcwt{2}}\wref{1}(\slsem{\lambda''}{3j_{\lambda}\fwt{2} +\gamma' -\uu\srt{3}/2})} \Bigr),
			\end{split}
			\\ 
			\begin{split} \label{eq:fusion_the_messy_one}
				\Gr{\slirr{\Lambda^{1}_{\lambda}}} \fuse \Gr{\slirr{\Lambda^{1}_{\lambda'}}} &= \sum_{\lambda''\in\pwlat} \bpfuscoeff{\lambda}{\lambda'}{\lambda''}
				\Bigl( \Gr{\slirr{\nabla(\lambda'') - \uu \fwt{0}/2}} + \Gr{\slsf{2\fcwt{1}}(\slirr{\nabla^{-1}(\lambda'') - \uu \fwt{0}/2})} \Bigr. \\
				&\qquad \Bigl. {} + \Gr{\slsf{2\fcwt{2}}(\slirr{\lambda'' - \uu \fwt{0}/2})} + 2 \Gr{\slsf{\scrt{3}}\slconj( \slirr{\Lambda^{1}_{\dynk( \lambda'')}} )} \Bigr).
			\end{split}
		\end{align}
	\end{subequations}
	Here, we recall that $\slconj$ denotes the conjugation functor of $\slminmod$ (\cref{sec:sl3auts}).
\end{corollary}
\begin{proof}
	We derive two of these rules, namely \eqref{eq:fusion_relax_semi} and \eqref{eq:fusion_the_messy_one}, to illustrate the methods used in general.
	For the first example, expand $\Gr{\slsem{\lambda}{\gamma}}$ as an infinite alternating series of relaxed modules as in \eqref{eq:grids}.
	This gives
	\begin{equation}
		\Gr{\slsem{\lambda}{\gamma}} \fuse \Gr{\slrel{\lambda'}{\gamma'}}
			= \sum_{n \ge 0} (-1)^{n} \Gr{\slsf{n\fcwt{3}}(\slrel{\nabla^n(\lambda)}{\gamma - n\uu\fwt{3}/2})} \fuse \Gr{\slrel{\lambda'}{\gamma'}}.
	\end{equation}
	Substituting the fully relaxed Grothendieck fusion rule \eqref{eq:fusion_RxR} then introduces a sum over $\lambda''\in\pwlat$ of \bp\ fusion coefficients $\bpfuscoeff{\nabla^n(\lambda)}{\lambda'}{\lambda''}$ multiplied by a sum of seven terms, each of which is the image in the Grothendieck group of a standard module.
	We apply the first identity of \cref{cor:bpfussymm} and then replace $\lambda''$ throughout by $\nabla^n(\lambda'')$.
	This allows us to write the Grothendieck fusion rule in the form
	\begin{equation}
		\begin{split}
			&\Gr{\slsem{\lambda}{\gamma}} \fuse \Gr{\slrel{\lambda'}{\gamma'}}
				= \sum_{\lambda''\in\pwlat} \bpfuscoeff{\lambda}{\lambda'}{\lambda''} \sum_{n \ge 0} (-1)^{n} \Bigl( 2\Gr{\slsf{n\fcwt{3}}(\slrel{\nabla^{n-1}(\lambda'')}{\gamma+\gamma' - n\uu\fwt{3}/2})} \Bigr. \\
			&\qquad + \Gr{\slsf{(n-1)\fcwt{3}}(\slrel{\nabla^{n-2}(\lambda'')}{\gamma+\gamma'-(n-1)\uu\fwt{3}/2})}
				+ \Gr{\slsf{(n+1)\fcwt{3}}(\slrel{\nabla^n(\lambda'')}{\gamma+\gamma'-(n+1)\uu\fwt{3}/2})} \\
			&\qquad + \Gr{\slsf{-\fcwt{1}+n\fcwt{3}}(\slrel{\nabla^{n}(\lambda'')}{\gamma+\gamma'+\uu\fwt{1}/2 - n\uu\fwt{3}/2 })}
				+ \Gr{\slsf{-\fcwt{2}+n\fcwt{3}}(\slrel{\nabla^{n+1}(\lambda'')}{\gamma+\gamma'+u\fwt{2}/2 - n\uu\fwt{3}/2 })} \\
			&\qquad \Bigl. {} + \Gr{\slsf{\fcwt{1}+n\fcwt{3}}(\slrel{\nabla^{n+1}(\lambda'')}{\gamma+\gamma'-\uu\fwt{1}/2 - n\uu\fwt{3}/2 })}
				+ \Gr{\slsf{\fcwt{2}+n\fcwt{3}}(\slrel{\nabla^{n}(\lambda'')}{\gamma+\gamma'-\uu\fwt{2}/2 - n\uu\fwt{3}/2 })} \Bigr).
		\end{split}
	\end{equation}
	Noting that $-\fcwt{1}+n\fcwt{3} = -\fcwt{2}+(n-1)\fcwt{3}$ and $\fcwt{2}+n\fcwt{3} = \fcwt{1}+(n-1)\fcwt{3}$ (and similarly with $\fcwt{i} \leftrightarrow \fwt{i}$), almost every term cancels and we are left with
	\begin{equation}
		\begin{split}
			\Gr{\slsem{\lambda}{\gamma}} \fuse \Gr{\slrel{\lambda'}{\gamma'}} &= \sum_{\lambda''\in\pwlat} \bpfuscoeff{\lambda}{\lambda'}{\lambda''}
				\biggl( \Gr{\slrel{\nabla^{-1}(\lambda'')}{\gamma+\gamma'}} + \Gr{\slsf{-\fcwt{1}}(\slrel{\lambda''}{\gamma+\gamma'+\uu\fwt{1}/2})} \biggr. \\
				&\qquad \biggl. {} + \Gr{\slsf{\fcwt{2}}(\slrel{\lambda''}{\gamma+\gamma'-\uu\fwt{2}/2})} + \Gr{\slsf{-\fcwt{3}}(\slrel{\nabla(\lambda'')}{\gamma+\gamma'+\uu\fwt{3}/2})} \biggr).
		\end{split}
	\end{equation}
	This generalises to the desired Grothendieck fusion rule \eqref{eq:fusion_relax_semi} using \cref{lem:fusauts,lem:Wfully}:
	\begin{equation}
		\Gr{\wref{}(\slsem{\lambda}{\gamma})} \fuse \Gr{\slrel{\lambda'}{\gamma'}}
			= \wref{}\brac[\Big]{\Gr{\slsem{\lambda}{\gamma}} \fuse \Gr{\wref{}^{-1}(\slrel{\lambda'}{\gamma'})}}
		= \wref{}\brac[\Big]{\Gr{\slsem{\lambda}{\gamma}} \fuse \Gr{\slrel{\lambda'}{\wref{}^{-1}(\gamma')}}}.
	\end{equation}

	We turn now to the second rule, utilising the third identity of \eqref{eq:grids} and substituting \eqref{eq:fusion_L0xL1} and \eqref{eq:fusion_L1xwS} to get
	\begin{equation}
    \begin{split}
        &\Gr{\slirr{\Lambda^1_{\lambda}}} \fuse \Gr{\slirr{\Lambda^1_{\lambda'}}}
				= \Gr{\slirr{\Lambda^1_{\lambda}}} \fuse \Gr{\wref{1}(\slsem{\lambda'}{\gamma^2_{\lambda'}})}
					- \Gr{\slirr{\Lambda^1_{\lambda}}} \fuse \Gr{\slsf{\fcwt{3}}(\slirr{\lambda'-\uu\fwt{0}/2})} \\
				&\qquad = \sum_{\lambda''\in\pwlat} \bpfuscoeff{\lambda}{\lambda'}{\lambda''} \brac[\bigg]{\Gr{\slsf{\fcwt{1}}(\slrel{\nabla(\lambda'')}{3j_{\lambda}\fwt{2}+\wref{1}(\gamma^2_{\lambda'})-\uu\fwt{2}/2}} + \Gr{\slsf{\fcwt{2}}\wref{1}(\slsem{\lambda''}{3j_{\lambda}\fwt{2}+\gamma^2_{\lambda'}-\uu\srt{3}/2})} - \Gr{\slsf{\fcwt{3}}(\slirr{\Lambda^1_{\lambda''}})}} \\
				&\qquad = \sum_{\lambda''\in\pwlat} \bpfuscoeff{\lambda}{\lambda'}{\lambda''} \brac[\bigg]{\Gr{\slsf{\fcwt{1}}(\slrel{\nabla(\lambda'')}{\gamma^2_{\nabla(\lambda'')}})} + \Gr{\slsf{\fcwt{2}}\wref{1}(\slsem{\lambda''}{\gamma^1_{\lambda''}})} - \Gr{\slsf{\fcwt{3}}(\slirr{\Lambda^1_{\lambda''}})}}. \label{eq:needs_to_cancel}
    \end{split}
	\end{equation}
	Here, we have inserted \eqref{eq:defgamma'} and recalled from the proof of \cref{thm:bpfusion} that $\bpfuscoeff{\lambda}{\lambda'}{\lambda''} = 0$ unless $j_{\lambda}+j_{\lambda'} = j_{\lambda''} \bmod{\ZZ}$.
	The semirelaxed module thus degenerates as per the second identity of \eqref{eq:grids}:
	\begin{equation} \label{eq:semi_degen}
			\Gr{\slsf{\fcwt{2}}\wref{1}(\slsem{\lambda''}{\gamma^1_{\lambda''}})}
				= \Gr{\slsf{\fcwt{2}}\wref{2}(\slirr{\Lambda^1_{\lambda''}})} + \Gr{\slsf{\fcwt{2}}(\slirr{\Lambda^2_{\lambda''}})}
				= \Gr{\slsf{\fcwt{2}}\wref{2}(\slirr{\Lambda^1_{\lambda''}})} + \Gr{\slsf{2\fcwt{2}}(\slirr{\lambda''-\uu\fwt{0}/2})}.
	\end{equation}
	As $\bilin{\gamma^2_{\lambda}}{\fcwt{3}} + j_{\lambda} + \frac{\uu}{6} = 0$, the fully relaxed module also degenerates.
	Since $\gamma^2_{\lambda}-\frac{\uu}{2}\fwt{3} = \gamma^2_{\nabla(\lambda)} \bmod{\rlat}$, this gives
	\begin{equation}
    \begin{split}
			&\Gr{\slsf{\fcwt{1}}(\slrel{\nabla(\lambda'')}{\gamma^2_{\nabla(\lambda'')}}} \\
      &\qquad = \Gr{\slsf{\fcwt{1}}\slsem{\nabla(\lambda'')}{\gamma^2_{\nabla(\lambda'')}}} + \Gr{\slsf{\scrt{1}}(\slsem{\nabla^{-1}(\lambda'')}{\gamma^1_{\nabla^{-1}(\lambda'')}})} \\
			&\qquad = \Gr{\slsf{\fcwt{1}}\wref{1}\wref{2}(\slirr{\Lambda^2_{\nabla(\lambda'')}})} + \Gr{\slsf{\fcwt{1}}\wref{1}(\slirr{\Lambda^1_{\nabla(\lambda'')}})}
        + \Gr{\slsf{\scrt{1}}\wref{1}\wref{2}(\slirr{\Lambda^1_{\nabla^{-1}(\lambda'')}})} + \Gr{\slsf{\scrt{1}}\wref{1}(\slirr{\Lambda^2_{\nabla^{-1}(\lambda'')}})} \\
      &\qquad = \Gr{\slirr{\lambda''-\uu\fwt{0}/2}} + \Gr{\slsf{\fcwt{1}}\wref{1}(\slirr{\Lambda^1_{\nabla(\lambda'')}})}
        + \Gr{\slsf{\fcwt{3}}(\slirr{\Lambda^1_{\lambda''}})} + \Gr{\slsf{2\fcwt{1}}(\slirr{\nabla^{-1}(\lambda'')-\uu\fwt{0}/2})},
    \end{split}
	\end{equation}
	where we have used \cref{lem:wL=sfL} in the final step.
	Putting this calculation together, the third term of the previous equation cancels against the last of \eqref{eq:needs_to_cancel} and we notice that the second term is equal to the first of \eqref{eq:semi_degen}, again by \cref{lem:wL=sfL}.
	We leave as a fun exercise the fact that the latter terms are also equal to $\Gr{\slsf{\scrt{3}}\slconj(\slirr{\Lambda^1_{d(\lambda'')}})}$.
	This ``symmetrisation'' of the result completes the proof of \eqref{eq:fusion_the_messy_one}.
\end{proof}

\addtocontents{toc}{\SkipTocEntry}
\section*{Declarations}

\paragraph{Data availability statement.}
Data sharing not applicable to this article as no datasets were generated or analysed during the current study.
\paragraph{Conflict of interests/Competing interests}
The authors declare that they have no conflict of interests or competing interests.

\flushleft

\begin{thebibliography}{10}

\bibitem{AbeRat02}
T~Abe, G~Buhl and C~Dong.
\newblock Rationality, regularity, and {$C_2$}-cofiniteness.
\newblock {\em Trans. Amer. Math. Soc.}, 356:3391--3402, 2004.
\newblock \opp{0204021}{math.QA}.

\bibitem{AdaRea17}
D~Adamovi\'{c}.
\newblock Realizations of simple affine vertex algebras and their modules: the cases $\widehat{sl(2)}$ and $\widehat{osp(1,2)}$.
\newblock {\em Comm. Math. Phys.}, 366:1025--1067, 2019.
\newblock \pp{1711.11342}{math.QA}.

\bibitem{AdaRel21}
D~Adamovi\'{c}, T~Creutzig and N~Genra.
\newblock Relaxed and logarithmic modules of $\widehat{\mathfrak{sl}_3}$.
\newblock {\em Math. Ann.}, 389:281--324, 2024.
\newblock \pp{2110.15203}{math.RT}.

\bibitem{AdaWei23}
D~Adamovi\'{c}, K~Kawasetsu and D~Ridout.
\newblock Weight module classifications for {Bershadsky}--{Polyakov} algebras.
\newblock {\em Comm. Contemp. Math.}, 26:2350063, 2024.
\pp{2303.03713}{math.QA}.

\bibitem{AdaRea20}
D~Adamovi\'{c}, K~Kawasetsu and D~Ridout.
\newblock A realisation of the {Bershadsky}--{Polyakov} algebras and their relaxed modules.
\newblock {\em Lett. Math. Phys.}, 111:38, 2021.
\newblock \pp{2007.00396}{math.QA}.

\bibitem{AdaVer95}
D~Adamovi\'{c} and A~Milas.
\newblock Vertex operator algebras associated to modular invariant representations of {$A_1^{(1)}$}.
\newblock {\em Math. Res. Lett.}, 2:563--575, 1995.
\newblock \opp{9509025}{q-alg}.

\bibitem{AraRep04}
T~Arakawa.
\newblock Representation theory of superconformal algebras and the {Kac}--{Roan}--{Wakimoto} conjecture.
\newblock {\em Duke Math. J.}, 130:435--478, 2005.
\newblock \opp{0405015}{math-ph}.

\bibitem{AraRat10}
T~Arakawa.
\newblock Rationality of {Bershadsky}--{Polyakov} vertex algebras.
\newblock {\em Comm. Math. Phys.}, 323:627--633, 2013.
\newblock \pp{1005.0185}{math.QA}.

\bibitem{AraAss10}
T~Arakawa.
\newblock Associated varieties of modules over {Kac}--{Moody} algebras and {$C_2$}-cofiniteness of {W}-algebras.
\newblock {\em Int. Math. Res. Not.}, 2015:11605--11666, 2015.
\newblock \pp{1004.1554}{math.QA}.

\bibitem{AraRat12b}
T~Arakawa.
\newblock Rationality of {W}-algebras: Principal nilpotent cases.
\newblock {\em Ann. Math.}, 182:565--604, 2015.
\newblock \pp{1211.7124}{math.QA}.

\bibitem{AraRat12}
T~Arakawa.
\newblock Rationality of admissible affine vertex algebras in the category
  $\mathcal{O}$.
\newblock {\em Duke Math. J.}, 165:67--93, 2016.
\newblock \pp{1207.4857}{math.QA}.

\bibitem{AraWei16}
T~Arakawa, V~Futorny and L-E Ramirez.
\newblock Weight representations of admissible affine vertex algebras.
\newblock {\em Comm. Math. Phys.}, 353:1151--1178, 2017.
\newblock \pp{1605.07580}{math.RT}.

\bibitem{AraRat19}
T~Arakawa and J~van Ekeren.
\newblock Rationality and fusion rules of exceptional {W}-algebras.
\newblock {\em J. Eur. Math. Soc.}, 25:2763--2813, 2023.
\newblock \pp{1905.11473}{math.RT}.

\bibitem{BeeInf13}
C~Beem, M~Lemos, P~Liendo, W~Peelaers, L~Rastelli and B~van Rees.
\newblock Infinite chiral symmetry in four dimensions.
\newblock {\em Comm. Math. Phys.}, 336:1359--1433, 2015.
\newblock \pp{1312.5344}{hep-th}.

\bibitem{BerRep01}
S~Berman, C~Dong and S~Tan.
\newblock Representations of a class of lattice type vertex algebras.
\newblock {\em J. Pure Appl. Algebra}, 176:27--47, 2002.
\newblock \opp{0109215}{math.QA}.

\bibitem{BerCon91}
M~Bershadsky.
\newblock Conformal field theories via {Hamiltonian} reduction.
\newblock {\em Comm. Math. Phys.}, 139:71--82, 1991.

\bibitem{BriSim95}
D~Britten, F~Lemire and V~Futorny.
\newblock Simple {$A_2$}-modules with a finite dimensional weight space.
\newblock {\em Comm. Algebra}, 23:467--510, 1995.

\bibitem{CanFus15b}
M~Canagasabey and D~Ridout.
\newblock Fusion rules for the logarithmic {$N=1$} superconformal minimal models {II}: Including the {Ramond} sector.
\newblock {\em Nucl. Phys.}, B905:132--187, 2016.
\newblock \pp{1512.05837}{hep-th}.

\bibitem{CosRem94}
A~Coste and T~Gannon.
\newblock Remarks on {Galois} symmetry in rational conformal field theories.
\newblock {\em Phys. Lett.}, B323:316--321, 1994.

\bibitem{CreTen23}
T~Creutzig.
\newblock Tensor categories of weight modules of $\widehat{\mathfrak{sl}}_2$ at admissible level.
 \newblock {\em J. Lond. Math. Soc.}, 110:e70037, 2024.
\newblock \pp{2311.10240}{math.RT}.

\bibitem{CreStr24}
T~Creutzig, J~Fasquel, A~Linshaw and S~Nakatsuka.
\newblock On the structure of {W}-algebras in type {A}.
\newblock {\em Jpn. J. Math.}, 20:1--111, 2025.
\newblock \pp{2403.08212}{math.RT}.

\bibitem{CreRig22}
T~Creutzig, R~McRae and J~Yang.
\newblock Rigid tensor structure on big module categories for some {W}-(super)algebras in type {A}.
\newblock {\em Comm. Math. Phys.}, 404:339--400, 2023.
\newblock \pp{2210.04678}{math.QA}.

\bibitem{CreMod12}
T~Creutzig and D~Ridout.
\newblock Modular data and {Verlinde} formulae for fractional level {WZW} models {I}.
\newblock {\em Nucl. Phys.}, B865:83--114, 2012.
\newblock \pp{1205.6513}{hep-th}.

\bibitem{CreLog13}
T~Creutzig and D~Ridout.
\newblock Logarithmic conformal field theory: beyond an introduction.
\newblock {\em J. Phys.}, A46:494006, 2013.
\newblock \pp{1303.0847}{hep-th}.

\bibitem{CreMod13}
T~Creutzig and D~Ridout.
\newblock Modular data and {Verlinde} formulae for fractional level {WZW} models {II}.
\newblock {\em Nucl. Phys.}, B875:423--458, 2013.
\newblock \pp{1306.4388}{hep-th}.

\bibitem{CreRel11}
T~Creutzig and D~Ridout.
\newblock Relating the archetypes of logarithmic conformal field theory.
\newblock {\em Nucl. Phys.}, B872:348--391, 2013.
\newblock \pp{1107.2135}{hep-th}.

\bibitem{deBMar91}
J~de~Boer and J~Goeree.
\newblock Markov traces and {II}$_1$ factors in conformal field theory.
\newblock {\em Comm. Math. Phys.}, 139:267--304, 1991.

\bibitem{DeSFin05}
A~{De~Sole} and V~Kac.
\newblock Finite vs affine {W}-algebras.
\newblock {\em Jpn. J. Math.}, 1:137--261, 2006.
\newblock \opp{0511055}{math-ph}.

\bibitem{FasOrt23}
J~Fasquel and S~Nakatsuka.
\newblock Orthosymplectic {Feigin}--{Semikhatov} duality.
\newblock {\em Sel. Math. New Ser.}, to appear.
\newblock \pp{2307.14574}{math.RT}.

\bibitem{FehInv23}
Z~Fehily.
\newblock Inverse reduction for hook-type {W}-algebras.
\newblock {\em Comm. Math. Phys.}, 405:214, 2024.
\newblock \pp{2306.14673}{math.QA}.

\bibitem{FehSub21}
Z~Fehily.
\newblock Subregular {W}-algebras of type {A}.
\newblock {\em Comm. Contemp. Math.}, 25:2250049, 2023.
\pp{2111.05536}{math.QA}.

\bibitem{FehCla20}
Z~Fehily, K~Kawasetsu and D~Ridout.
\newblock Classifying relaxed highest-weight modules for admissible-level
  {Bershadsky}--{Polyakov} algebras.
\newblock {\em Comm. Math. Phys.}, 385:859--904, 2021.
\newblock \pp{2007.03917}{math.RT}.

\bibitem{FehMod21}
Z~Fehily and D~Ridout.
\newblock Modularity of {Bershadsky}--{Polyakov} minimal models.
\newblock {\em Lett. Math. Phys.}, 122:46, 2022.
\newblock \pp{2110.10336}{math.QA}.

\bibitem{FeiAff92}
B~Feigin and E~Frenkel.
\newblock Affine {Kac}--{Moody} algebras at the critical level and {Gelfand}--{Dikii} algebras.
\newblock {\em Int. J. Mod. Phys.}, A7S1A:197--215, 1992.

\bibitem{FeiEqu97}
B~Feigin, A~Semikhatov and I~Tipunin.
\newblock Equivalence between chain categories of representations of affine $sl(2)$ and {$N = 2$} superconformal algebras.
\newblock {\em J. Math. Phys.}, 39:3865--3905, 1998.
\newblock \opp{9701043}{hep-th}.

\bibitem{FerLie90}
S~Fernando.
\newblock Lie algebra modules with finite-dimensional weight spaces, {I}.
\newblock {\em Trans. Amer. Math. Soc.}, 322:757--781, 1990.

\bibitem{DiFCon97}
P~Di Francesco, P~Mathieu and D~S\'{e}n\'{e}chal.
\newblock {\em Conformal Field Theory}.
\newblock Graduate Texts in Contemporary Physics. Springer--Verlag, New York, 1997.

\bibitem{FreCha92}
E~Frenkel, V~Kac and M~Wakimoto.
\newblock Characters and fusion rules for {W}-algebras via quantized {Drinfeld}--{Sokolov} reduction.
\newblock {\em Comm. Math. Phys.}, 147:295--328, 1992.

\bibitem{FreVer92}
I~Frenkel and Y~Zhu.
\newblock Vertex operator algebras associated to representations of affine and {Virasoro} algebras.
\newblock {\em Duke Math. J.}, 66:123--168, 1992.

\bibitem{FriCon86}
D~Friedan, E~Martinec and S~Shenker.
\newblock Conformal invariance, supersymmetry and string theory.
\newblock {\em Nucl. Phys.}, B271:93--165, 1986.

\bibitem{FutIrr96}
V~Futorny.
\newblock Irreducible non-dense {$A_1^{(1)}$}-modules.
\newblock {\em Pacific J. Math.}, 172:83--99, 1996.

\bibitem{GabFus01}
M~Gaberdiel.
\newblock Fusion rules and logarithmic representations of a {WZW} model at fractional level.
\newblock {\em Nucl. Phys.}, B618:407--436, 2001.
\newblock \opp{0105046}{hep-th}.

\bibitem{GepStr86}
D~Gepner and E~Witten.
\newblock String theory on group manifolds.
\newblock {\em Nucl. Phys.}, B278:493--549, 1986.

\bibitem{GorSim06}
M~Gorelik and V~Kac.
\newblock On simplicity of vacuum modules.
\newblock {\em Adv. Math.}, 211:621--677, 2007.
\newblock \opp{0606002}{math-ph}.

\bibitem{HuaVer04a}
Y-Z Huang.
\newblock Vertex operator algebras, the {Verlinde} conjecture, and modular tensor categories.
\newblock {\em Proc. Natl. Acad. Sci. USA}, 102:5352--5356, 2005.
\newblock \opp{0412261}{math.QA}.

\bibitem{HumRep08}
J~Humphreys.
\newblock {\em Representations of Semisimple {Lie} Algebras in the {BGG} Category {$\mathcal{O}$}}, volume~94 of {\em Graduate Studies in  Mathematics}.
\newblock American Mathematical Society, Providence, 2008.

\bibitem{KacInf74}
V~Kac.
\newblock Infinite-dimensional {Lie} algebras, and the {Dedekind} $\eta$-function.
\newblock {\em Funct. Anal. Appl.}, 8:68--70, 1974.

\bibitem{KacInf90}
V~Kac.
\newblock {\em Infinite-Dimensional {Lie} Algebras}.
\newblock Cambridge University Press, Cambridge, 1990.

\bibitem{KacInf84}
V~Kac and D~Peterson.
\newblock Infinite-dimensional {Lie} algebras, theta functions and modular forms.
\newblock {\em Adv. Math.}, 53:125--264, 1984.

\bibitem{KacQua03}
V~Kac, S~Roan and M~Wakimoto.
\newblock Quantum reduction for affine superalgebras.
\newblock {\em Comm. Math. Phys.}, 241:307--342, 2003.
\newblock \opp{0302015}{math-ph}.

\bibitem{KacMod88b}
V~Kac and M~Wakimoto.
\newblock Modular and conformal invariance constraints in representation theory of affine algebras.
\newblock {\em Adv. Math.}, 70:156--236, 1988.

\bibitem{KacMod88}
V~Kac and M~Wakimoto.
\newblock Modular invariant representations of infinite-dimensional {Lie} algebras and superalgebras.
\newblock {\em Proc. Nat. Acad. Sci. USA}, 85:4956--4960, 1988.

\bibitem{KacQua03b}
V~Kac and M~Wakimoto.
\newblock Quantum reduction and representation theory of superconformal algebras.
\newblock {\em Adv. Math.}, 185:400--458, 2004.
\newblock \opp{0304011}{math-ph}.

\bibitem{KawRel20}
K~Kawasetsu.
\newblock Relaxed highest-weight modules {III}: character formulae.
\newblock {\em Adv. Math.}, 393:108079, 2021.
\newblock \pp{2003.10148}{math.RT}.

\bibitem{KawRel18}
K~Kawasetsu and D~Ridout.
\newblock Relaxed highest-weight modules {I}: rank $1$ cases.
\newblock {\em Comm. Math. Phys.}, 368:627--663, 2019.
\newblock \pp{1803.01989}{math.RT}.

\bibitem{KawRel19}
K~Kawasetsu and D~Ridout.
\newblock Relaxed highest-weight modules {II}: classifications for affine vertex algebras.
\newblock {\em Comm. Contemp. Math.}, 24:2150037, 2022.
\newblock \pp{1906.02935}{math.RT}.

\bibitem{KawAdm21}
K~Kawasetsu, D~Ridout and S~Wood.
\newblock An admissible-level $\mathfrak{sl}_3$ model.
\newblock {\em Lett. Math. Phys.}, 112:96, 2022.
\newblock \pp{2107.13204}{math.QA}.

\bibitem{KohFus88}
I~Koh and P~Sorba.
\newblock Fusion rules and (sub)modular invariant partition functions in nonunitary theories.
\newblock {\em Phys. Lett.}, B215:723--729, 1988.

\bibitem{LiPhy97}
H~Li.
\newblock The physics superselection principle in vertex operator algebra theory.
\newblock {\em J. Algebra}, 196:436--457, 1997.

\bibitem{MatCla00}
O~Mathieu.
\newblock Classification of irreducible weight modules.
\newblock {\em Ann. Inst. Fourier (Grenoble)}, 50:537--592, 2000.

\bibitem{MatPri98}
P~Mathieu and M~Walton.
\newblock On principal admissible representations and conformal field theory.
\newblock {\em Nucl. Phys.}, B553:533--558, 1999.
\newblock \opp{9812192}{hep-th}.

\bibitem{MazLec10}
V~Mazorchuk.
\newblock {\em Lectures on {$\mathfrak{sl}_2(\mathbb{C})$}-Modules}.
\newblock Imperial College Press, London, 2010.

\bibitem{NakRel24}
H~Nakano, F~Orosz Hunziker, A~Ros Camacho and S~Wood.
\newblock Fusion rules and rigidity for weight modules over the simple admissible affine $\mathfrak{sl}(2)$  and $\mathcal{N}=2$  superconformal vertex operator superalgebras.
\newblock \pp{2411.11387}{math.QA}.

\bibitem{PolGau90}
A~Polyakov.
\newblock Gauge transformations and diffeomorphisms.
\newblock {\em Int. J. Mod. Phys.}, A5:833--842, 1990.

\bibitem{RidSL208}
D~Ridout.
\newblock $\widehat{\mathfrak{sl}}(2)_{-1/2}$: A case study.
\newblock {\em Nucl. Phys.}, B814:485--521, 2009.
\newblock \pp{0810.3532}{hep-th}.

\bibitem{RidFus10}
D~Ridout.
\newblock Fusion in fractional level $\widehat{\mathfrak{sl}}(2)$-theories with $k=-\tfrac{1}{2}$.
\newblock {\em Nucl. Phys.}, B848:216--250, 2011.
\newblock \pp{1012.2905}{hep-th}.

\bibitem{RidAdm17}
D~Ridout, J~Snadden and S~Wood.
\newblock An admissible level $\widehat{\mathfrak{osp}}(1\vert2)$-model: modular transformations and the {Verlinde} formula.
\newblock {\em Lett. Math. Phys.}, 108:2363--2423, 2018.
\newblock \pp{1705.04006}{hep-th}.

\bibitem{RidBos14}
D~Ridout and S~Wood.
\newblock Bosonic ghosts at $c=2$ as a logarithmic {CFT}.
\newblock {\em Lett. Math. Phys.}, 105:279--307, 2015.
\newblock \pp{1408.4185}{hep-th}.

\bibitem{RidRel15}
D~Ridout and S~Wood.
\newblock Relaxed singular vectors, {Jack} symmetric functions and fractional level $\widehat{\mathfrak{sl}}(2)$ models.
\newblock {\em Nucl. Phys.}, B894:621--664, 2015.
\newblock \pp{1501.07318}{hep-th}.

\bibitem{RidVer14}
D~Ridout and S~Wood.
\newblock The {Verlinde} formula in logarithmic {CFT}.
\newblock {\em J. Phys. Conf. Ser.}, 597:012065, 2015.
\newblock \pp{1409.0670}{hep-th}.

\bibitem{SemInv94}
A~Semikhatov.
\newblock Inverting the {Hamiltonian} reduction in string theory.
\newblock In {\em 28th International Symposium on Particle Theory, Wendisch-Rietz, Germany}, pages 156--167, 1994.
\newblock \opp{9410109}{hep-th}.

\bibitem{VerFus88}
E~Verlinde.
\newblock Fusion rules and modular transformations in {2D} conformal field theory.
\newblock {\em Nucl. Phys.}, B300:360--376, 1988.

\bibitem{WalFus90}
M~Walton.
\newblock Fusion rules in {Wess}--{Zumino}--{Witten} models.
\newblock {\em Nucl. Phys.}, B340:777--790, 1990.

\bibitem{WitNon84}
E~Witten.
\newblock Non-abelian bosonization in two dimensions.
\newblock {\em Comm. Math. Phys.}, 92:455--472, 1984.

\bibitem{ZhuMod96}
Y~Zhu.
\newblock Modular invariance of characters of vertex operator algebras.
\newblock {\em J. Amer. Math. Soc.}, 9:237--302, 1996.

\end{thebibliography}
\providecommand{\opp}[2]{\textsf{arXiv:\mbox{#2}/#1}}
\providecommand{\pp}[2]{\textsf{arXiv:#1 [\mbox{#2}]}}

\end{document}